\documentclass{article}
\usepackage[utf8]{inputenc}
\usepackage{amsmath}
\usepackage{amssymb}
\usepackage{stmaryrd}
\usepackage{multicol}
\usepackage{amsthm}
\usepackage{bbold}
\usepackage[left=3cm,right=3cm,top=3cm,bottom=3cm]{geometry}
\usepackage[ colorlinks=true, linkcolor=blue, citecolor=green]{hyperref}
\hypersetup{bookmarks=true}
\usepackage{bookmark}
\usepackage{mathrsfs}
\usepackage{graphicx}
\usepackage{appendix}
\usepackage{tikz}
\usepackage{subcaption}
\usetikzlibrary[patterns]

\newtheorem{theorem}{Theorem}[section]
\newtheorem{prop}{Proposition}[section]
\newtheorem{hyp}[prop]{Assumption}
\newtheorem{defi}[prop]{Definition}
\newtheorem{lemme}[prop]{Lemma}
\newtheorem{corr}[prop]{Corollary}
\newtheorem{conj}[theorem]{Conjecture}

\newcommand{\footremember}[2]{
    \footnote{#2}
    \newcounter{#1}
    \setcounter{#1}{\value{footnote}}
}

\title{Individual-based stochastic model with unbounded growth, birth and death rates: a tightness result}
\author{Virgile Brodu\footremember{trailer}{Université de Lorraine, Inria, IECL, F-54000 Nancy, France}}

\begin{document}
\maketitle

\begin{abstract}
We study population dynamics through a general growth/degrowth-fragmentation process, with resource consumption and unbounded growth/degrowth, birth and death rates. Our model is structured in a positive trait called \textit{energy} (which is a proxy for any biological parameter such as size, age, mass, protein quantity...), and the jump rates of the process can be arbitrarily high depending on individual energies, which has not been considered yet in the literature. After a preliminary study to construct well-defined objects (which is necessary contrary to similar works, because of the explosion of individual rates), we consider a classical sequence of renormalizations of the underlying process and obtain a tightness result for the associated laws in large-population asymptotics. We characterize the accumulation points of this sequence as solutions of an integro-differential system of equations, which proves the existence of measure solutions to this system. Furthermore, if such a measure solution is unique, then our tightness result becomes a convergence result towards this unique process. We illustrate our work with the case of \textit{allometric} rates (\textit{i.e.} they are assumed to be power functions) and eventually present numerical simulations in this allometric setting.
\end{abstract}

\noindent \textbf{Keywords:} individual-based model, Jumping Markov Process, large-population asymptotics, martingale problem, measure-valued process, tightness, unbounded jump rates.

\section{Introduction}

We propose a general framework to study a population, with competition between individuals through resource consumption, and design a stochastic individual-based model. We gather all individual trajectories into a measure-valued population process. This work is in line with a rich literature about similar individual-based models, originating with \cite{fournier2004microscopic}, followed by \cite{champagnat2005individualbased} with mutations, adding an age structure in \cite{chi_08}, an interaction with resources in \cite{FRITSCH20151}, and a diffusion term in \cite{tchouanti2024well}. The main contribution of all these papers is a convergence result in large-population asymptotics, for a suitable renormalization of the underlying stochastic processes, towards the solution $\mu^{*}$ to an integro-differential system of equations, which is deterministic conditionally to the initial condition (see the original Theorem 5.3. in \cite{fournier2004microscopic}, or the recent Theorem 2.1. and Theorem 3.2. in \cite{tchouanti2024well}). The main result of our article is Theorem~\ref{theo:convergence}, which is a tightness result similar to these convergences established in the literature. One common feature between previous papers is that growth, birth and death rates are bounded, and $\mu^{*}$ integrates bounded functions. Our work goes beyond these previous results for the following reasons. First, our jump rates are not necessarily bounded, which adds potential problems of explosion of individual energy or population size in finite time, and technical difficulties in the proofs, but allows more general and biologically relevant models. We develop the case of \textit{allometric} scalings (\textit{i.e.} the jump rates are power functions), which seem to be a key ingredient for modelling the behavior of species \cite{peters_1983,brown_04,malerba_2019,brodu2026}, in Section~\ref{subsec:examples}. Then, any accumulation point, still denoted as $\mu^{*}$, integrates a broader set of functions, not necessarily bounded (see Lemma~\ref{corr:plusdefonctions}). Hence, our work could be adapted to generalize existing results on growth-fragmentation models such as \cite{chi_08} or \cite{FRITSCH20151}. 
\\\\
We make a specific choice for the fragmentation modelling: every offspring in the population appears with energy $x_{0}>0$, which is a fixed parameter of the model. The introduction of this parameter is linked to the biological motivation of this work, presented in \cite{brodu2026}, where we make $x_{0}$ vary to obtain necessary conditions for our model to be biologically relevant with allometric functional parameters. On the contrary, in the present article, $x_{0}$ is fixed once and for all. Our work could certainly be extended to other various fragmentation patterns, like binary fission \cite{doumic2015statistical}, conservative and symmetrical division kernels \cite{coranicofab}, or self-similar splitting densities \cite{lauwal07}.
\\\\
Individual trajectories depend on the evolution of the resource over time, which in turn depends on individual energies. Thus, we study a process $(\mu_{t},R_{t})_{t \geq 0}$, where $\mu$ is a finite point measure representing the state of the population, and $R \in [0,R_{\max}]$ the amount of resources ($R_{\max}>0$ being a parameter of the model). This measure-valued process is constructed with Poisson point measures and is piecewise deterministic. Compared to other existing models \cite{chi_08,FRITSCH20151}, due to the choice of unbounded jump rates in our case, additional work is necessary to ensure that the process is well defined for every $t \geq 0$. Then, we introduce a sequence of renormalized processes $(\mu^{K}_{t},R^{K}_{t})_{t \geq 0}$ constructed with the same procedure. The parameter $K$ is a scaling parameter representing the number of individuals in the population at time 0. We want to understand the behavior of our model in a large population asymptotic, when the parameter $K$ goes to $+ \infty$. Importantly, unlike previous work on similar renormalizations of individual-based models \cite{fournier2004microscopic,champagnat2005individualbased,chi_08,FRITSCH20151,tchouanti2024well}, we have to add an additional assumption, due to the fact that our rates are unbounded. Precisely, we assume that there exists a ``Lyapunov-type'' smooth and non-decreasing function $\omega$, associated with the functional parameters of the model (see the upcoming $\mathbf{(H2)}$ in Theorem~\ref{theo:initial}). This will allow us to control stochastic integrals associated to our renormalized processes, and have two important consequences. First, we can prove martingale properties for our processes (see Proposition~\ref{prop:martingalerenormalisee}), which is a key ingredient of the proof in Section~\ref{sec:proof}. Then, we can control the expectation of the population size at any time uniformly on $K$ (see Proposition~\ref{lemme:controlenp}). Provided that such a function $\omega$ exists, we endow the set of measures integrating the function $\omega$ with a weighted topology which is not standard. We call it the $\omega$-weak topology, and write $(\mathcal{M}_{\omega}(\mathbb{R}^{*}_{+}),\mathrm{w})$ for the associated topological space (see Section~\ref{sec:theorem} and Appendix B.2 in \cite{broduthesis} for details).
\\\\
Let us give immediately a taste of our main result with as least technical details as possible. In the following, $b(x)$ is the birth rate for an individual with energy $x$, and similarly $d(x)$ is the death rate. The functions $b$ and $d$ are non-negative. The function $g : (x,R) \mapsto g(x,R)$ accounts for the instantaneous rate of growth or degrowth for an individual with energy $x$ and resources $R$. The function $\rho : (R,\mu) \mapsto \rho(R,\mu)$ is the instantaneous rate of increase or decrease of resources $R$ consumed by the population $\mu$. The functions $g$ and $\rho$ can change sign, depending on the resources $R$ and the state of the individual/population. For more details about the previous functional parameters, we refer the reader to Section~\ref{subsec:dyn}. In the following, we write $\langle \mu, \varphi \rangle := \displaystyle{\int_{\mathbb{R}^{*}_{+}} \varphi(x) \mathrm{d}\mu(x)}$ for every $\mu \in \mathcal{M}_{\omega}(\mathbb{R}^{*}_{+})$ and $\varphi$ measurable. A simplified summarization of our tightness result reads as follow:
\begin{theorem}
First, we assume that
\begin{itemize}
\item[$\mathbf{(H0)}$] $\forall x>0, \quad g(x,0) < 0 < g(x,R_{\max})$,
\item [$\mathbf{(H1)}$] $\forall x >0, \displaystyle{\int_{0}^{x} \dfrac{d(t)}{g(t,0)} \mathrm{d}t} = - \infty$.
\end{itemize}
Then, we assume that there exists a weight function $\omega$ such that
\begin{itemize}
\item[$\mathbf{(H2)}$]
$\omega \in \mathcal{C}^{1}(\mathbb{R}^{*}_{+})$ is positive, non-decreasing, Lipschitz continuous on $(1,+\infty)$, and
\begin{itemize}
\item[$\mathbf{(a)}$] $\exists c_{g}>0, \forall x>0, \quad  \sup_{R \in [0,R_{\max}]} |g(x,R)| \leq c_{g}\omega(x)$,
\item[$\mathbf{(b)}$] $\exists c_{b}>0, \forall x>1, \quad b(x) \leq c_{b}\omega(x)$,
\item[$\mathbf{(c)}$] $\exists c_{d}>0, \forall x>0, \quad d(x) \omega(x) \leq c_{d}$.
\end{itemize} 
\end{itemize}
Finally, we assume that the sequence $\left(\left(\mu^{K}_{t},R^{K}_{t}\right)_{t \geq 0}\right)_{K \in \mathbb{N}^{*}}$ of renormalized processes (formally defined in Section \ref{subsec:construcdeux}) is such that
\begin{itemize}
    \item[$\mathbf{(H3)}$] there exists a random variable $\mu_{0}^{*} \in \mathcal{M}_{\omega}(\mathbb{R}^{*}_{+})$ such that $(\mu_{0}^{K})_{K \in \mathbb{N}^{*}}$ converges in law towards $\mu_{0}^{*}$ in $(\mathcal{M}_{\omega}(\mathbb{R}^{*}_{+}),\mathrm{w})$,
    \item[$\mathbf{(H4)}$] there exists $p>1$ such that 
    \begin{align*}
    \underset{K \in \mathbb{N}^{*}}{\mathrm{sup}} \hspace{0.1 cm} \mathbb{E}\left(\left\langle \mu^{K}_{0}, 1 + \mathrm{Id} + \omega \right\rangle^{p}\right) < + \infty.
    \end{align*}
\end{itemize} 
Under $\mathbf{(H0)}$, $\mathbf{(H1)}$, $\mathbf{(H2)}$, $\mathbf{(H3)}$, $\mathbf{(H4)}$, for all $T \geq 0$, the sequence $\bigg(\left(\mu^{K}_{t},R^{K}_{t}\right)_{t \in [0,T]}\bigg)_{K \in \mathbb{N}^{*}}$ is tight in the Skorokhod space $\mathbb{D}([0,T], (\mathcal{M}_{\omega}(\mathbb{R}^{*}_{+}),\mathrm{w}) \times [0,R_{\max}])$. Any of its accumulation point $(\mu^{*}_{t},R^{*}_{t})_{t \in [0,T]}$ is a continuous process and verifies almost surely, for all $t \in [0,T]$,
\begin{align}
    R^{*}_{t} =  R^{*}_{0} + \displaystyle{\int_{0}^{t}\rho(R^{*}_{s},\mu^{*}_{s})\mathrm{d}s},
    \label{eq:resslim}
\end{align}
     and for every $\varphi :(t,x) \mapsto \varphi(t,x) =: \varphi_{t}(x)$ smooth enough and dominated by $\omega$,
     \begin{multline}
       \langle \mu^{*}_{t}, \varphi_{t} \rangle =  \langle \mu^{*}_{0}, \varphi_{0} \rangle + \displaystyle{\int_{0}^{t} \int_{\mathbb{R}^{*}_{+}} \bigg( \partial_{1}\varphi(s,x) + g(x,R^{*}_{s})\partial_{2}\varphi(s,x) }   \\
     + b(x)(\varphi_{s}(x_{0}) + \varphi_{s}(x-x_{0}) - \varphi_{s}(x)) - d(x)\varphi_{s}(x) \bigg) \mu^{*}_{s}(\mathrm{d}x) \mathrm{d}s,
     \label{eq:indivlim}
     \end{multline}
     where $\partial_{1}$, respectively $\partial_{2}$, are the partial derivatives with respect to the first, respectively second, variable.
\label{theo:initial}
\end{theorem}

Assumptions $\mathbf{(H0)}$ and $\mathbf{(H1)}$ are part of the model, and necessary for the good definition of individual trajectories (see Section~\ref{subsec:assumptions}). Then, Assumption $\mathbf{(H2)}$ is necessary to obtain martingale properties for our processes (see Section~\ref{subsec:martrenom}). Note that the case of bounded rates is included in Theorem~\ref{theo:initial}, simply by choosing $\omega \equiv 1$. Finally, Assumptions $\mathbf{(H3)}$ and $\mathbf{(H4)}$ are classical assumptions to prove this kind of tightness result (see the original Assumption C1.1 in \cite{fournier2004microscopic}, or Assumption 4 in \cite{chi_08}), but adapted here to our setting with the weight function $\omega$. Our main result, which is Theorem~\ref{theo:convergence} in Section~\ref{sec:theorem}, will be stronger than Theorem~\ref{theo:initial}, because it relies on weaker, though more complex, assumptions on the weight function $\omega$. Assumption $\mathbf{(H2)}$ will thus replaced by the combination of Assumption~\ref{hyp:poidsomega} and Assumption~\ref{ass:finallejd}. Note that the limiting system \eqref{eq:resslim}-\eqref{eq:indivlim} is very similar to those obtained in previous works (see for example Theorem 5.2 in \cite{FRITSCH20151}), but with two main differences. First, Equation~\eqref{eq:indivlim} is not only valid for bounded functions, but for functions dominated by $\omega$, which can be unbounded (see Figure~\ref{fig:exampeomega}). Then, the term multiplied by $b(x)$ in Equation~\eqref{eq:indivlim} can be interpreted as a birth term, and is naturally adapted to our specific fragmentation modelling with parameter $x_{0}$. To the best of our knowledge, this is the first tightness result with unbounded rates for this kind of individual-based model. The sketch of the proof follows the procedure initially proposed by Fournier and Méléard \cite{fournier2004microscopic}, but at each step, we encounter additional difficulties due to the interaction with the resource and unbounded rates. In particular, we extend results of Roelly \cite{roel_86} in Theorem~\ref{theo:roel}, and Méléard-Roelly \cite{meleard1993convergences} in Theorem \ref{theo:melroel} to work with the $\omega$-weak topology.
\\\\
In Section~\ref{subsec:dyn} and \ref{sec:construction}, we construct the process $(\mu_{t},R_{t})_{t}$. In Section \ref{subsec:assumptions} and \ref{subsec:refinement}, we provide assumptions under which this process is well-defined for every $t \geq 0$. We also introduce the weight function $\omega$. In Section~\ref{subsec:geninf}, we show martingale properties of the process. In Section~\ref{subsec:construcdeux}, we introduce the sequence of renormalized processes $(\mu^{K}_{t},R^{K}_{t})_{t \in [0,T]}$. We present our main result in Theorem~\ref{theo:convergence} of Section~\ref{sec:theorem}, along with possible extensions of this result and conjectures in Section~\ref{subsec:discussiontheo}. In particular, if there exists a unique solution to the system \eqref{eq:resslim}-\eqref{eq:indivlim}, then our tightness result becomes a convergence result. The uniqueness of the limiting process in the case of bounded rates is well-known (see Proposition II.5.7. in \cite{broduthesis}). It remains an open question to know if we can tackle more general cases with unbounded rates, and we provide some directions of research in Section~\ref{section:uniqueness}. In Section~\ref{subsec:examples}, we illustrate our results, both theoretically and numerically, with power functions for the jump rates. As far as individual-based models are concerned, this so-called \textit{allometric} setting has not been studied yet, except for our previous model with constant resources in \cite{brodu2026}. Finally, the proof of Theorem~\ref{theo:convergence} is developed in Section~\ref{sec:proof}. In Appendix~\ref{app:inter}, we provide the proofs of intermediate results, including Theorem~\ref{theo:roel} and Theorem \ref{theo:melroel}; and in Appendix~\ref{app:simu}, we give our simulation parameters and algorithms for Section~\ref{subsec:examples}. Note that we will also refer extensively to \cite{broduthesis} for technical, but classical details.

\section{Definitions and assumptions}
\label{sec:defass}

First, in Section~\ref{subsec:dyn}, we define individual dynamics, \textit{i.e.} deterministic metabolism and resource consumption, and random birth or death events. Then, in Section~\ref{sec:construction}, we introduce notations to gather all the individual flows and resource dynamics into one deterministic flow between random jumps. Then, we provide an algorithmic construction of the process $(\mu_{t},R_{t})_{t}$ with Poisson point measures, valid up to a stopping time accounting for two possible problematic events. The first one is the possibility for individual trajectories to reach 0 or $+ \infty$ in finite time between random jumps, and the second one is the possible accumulation of jump times at the population level. In Section~\ref{subsec:assumptions}, we provide an assumption to address the first problematic event. Then, in Section~\ref{subsec:refinement}, we resolve the second problem, by giving a general setting under which the population process is almost surely well-defined for $t \in \mathbb{R}^{+}$. Once this is done, in Section~\ref{subsec:geninf}, we show martingale properties for the process $(\mu_{t},R_{t})_{t \geq 0}$.

\subsection{Individual and resource dynamics}
\label{subsec:dyn}
In this section, we introduce the functional parameters and main mechanisms of our model at the individual level, but note that the formal construction of the population process is done in Section~\ref{sec:construction}. To distinguish between individuals, we define the set of indices as
$$ \mathcal{U} := \bigcup_{n \in \mathbb{N}}(\mathbb{N}^{*})^{n+1}.$$
Over time, every individual in the population will have an index of the form $ u := u_{1}...u_{n+1}$ with some $n \geq 0$, and some positive integers $u_{1},..., u_{n}$.
At time $t \geq 0$, an alive individual indexed by $u$ in the population is characterized by a trait called \textit{energy} and written $\xi^{u}_{t} \in \mathbb{R}^{*}_{+} \cup \{ \partial \}$, where $\partial$ is a cemetery state. Every individual compete for a fluctuating stock of resources $R_{t} \in \mathbb{R}^{+}$, can die and reach the cemetery state $\partial$, and can reproduce several times during its life. When a birth occurs, we add an individual to the population, with a new index in $\mathcal{U}$. In the population, at time $t \geq 0$, we denote as $V_{t}$ the set of alive individuals at this time (\textit{i.e.} whose energy is not $\partial$ at time $t$). Every alive individual $u \in V_{t}$ is represented by a Dirac mass at $\xi^{u}_{t}$. Thus, we will define the population process $\mu_{t}$ as a point measure given at time $t$ by 
\begin{align}
\mu_{t} := \sum\limits_{u \in V_{t}}{\delta_{\xi_{t}^{u}}}.
\label{eq:mudemfu}
\end{align} 
An individual trajectory is deterministic between some random jump times, corresponding to birth or death events. 
\\\\
\textbf{Birth}
\\\\
For $t \geq 0$, an individual indexed by $u \in V_{t-}$ with energy $\xi^{u}_{t-} \in \mathbb{R}^{*}_{+}$ gives birth to a single offspring at rate $b(\xi^{u}_{t-})$. This individual transfers a constant amount of energy $x_{0}$ to the offspring. The energy of the individual goes from $\xi^{u}_{t^{-}}$ to $\xi^{u}_{t} := \xi^{u}_{t^{-}}-x_{0}$. The point measure then jumps to $\mu_{t} := \mu_{t-}-\delta_{\xi^{u}_{t-}}+\delta_{\xi_{t-}^{u}-x_{0}}+\delta_{x_{0}}$. If the parent has the label $u := u_{1}...u_{n}$ and this is the $k$-th birth jump for this parent for $k \geq 1$, the index of the offspring is $uk := u_{1}...u_{n}k$. Then, we set $V_{t} := V_{t-} \cup \{uk\}$.
 
\begin{center}
\begin{tikzpicture}[xscale=1,yscale=1]

\tikzstyle{fleche}=[->,>=latex,thick]
\tikzstyle{noeud}=[circle,draw]

\def\DistanceInterNiveaux{3}
\def\DistanceInterFeuilles{2}

\def\NiveauA{(0)*\DistanceInterNiveaux}
\def\NiveauB{(1)*\DistanceInterNiveaux}

\def\InterFeuilles{(-1)*\DistanceInterFeuilles}

\node[noeud](R) at ({\NiveauA},{(0.5)*\InterFeuilles}) {} ;
\node(S) at ([xshift={-0.4 cm}]R) {$\xi^{u}_{t-}$} ;
\node[noeud](R_a) at ({\NiveauB},{(0)*\InterFeuilles}) {} ;
\node(T) at ([xshift={0.5 cm}]R_a) {$x_{0}$} ;
\node[noeud](R_b) at ({\NiveauB},{(1)*\InterFeuilles}) {} ;
\node(U) at ([xshift={1.1 cm}]R_b) {$\xi^{u}_{t-}-x_{0}$} ;

\draw[fleche] (R)--(R_a) {};
\draw[fleche] (R)--(R_b) {};
\end{tikzpicture}
\end{center}
We assume that the birth rate $b$ is equal to 0 for every $x \leq x_{0}$, so that no individual with negative energy appears during a birth event. Also, we assume that $b$ is non-negative and continuous for $x > x_{0}$.
\\\\
\textbf{Death}
\\\\
For $t \geq 0$, an individual indexed by $u \in V_{t-}$ with energy $\xi^{u}_{t-} \in \mathbb{R}^{*}_{+}$ dies at positive and continuous rate $d(\xi^{u}_{t-})$. Then, the individual process jumps to $\xi_{t}^{u} := \partial$ and we set $\xi_{s}^{u} := \partial$ for every $s \geq t$. The point measure jumps to $\mu_{t} :=\mu_{t-}-\delta_{\xi^{u}_{t-}}$, and we set $V_{t} :=V_{t-} \setminus \{u\}$.
\\\\
\textbf{Energy loss and resource consumption}
\begin{enumerate}
    \item For $t \geq 0$, an individual $u \in V_{t}$ with energy $\xi^{u}_{t} \in \mathbb{R}^{*}_{+}$ loses energy over time, at non-negative and $\mathcal{C}^{1}(\mathbb{R}^{*}_{+})$ rate $\ell(\xi^{u}_{t})$. We are interested in situations where $\ell$ is a positive function (\textit{i.e.} we consider a true growth/degrowth-fragmentation model with possibly decreasing energies over time) to model the metabolic rate of individuals \cite{brown_04,brodu2026}. However, note that we also allow $\ell \equiv 0$, so that our model can be adapted to generalize previous growth-fragmentation models where individual traits are only increasing \cite{chi_08,FRITSCH20151}.
    \item In order to balance this energy loss, this individual consumes the resource at rate $f(\xi^{u}_{t},R_{t})$. Importantly, we suppose that $f$ is of the form $f(\xi,R) := \phi(R)\psi(\xi),$ where $\psi$ is a $\mathcal{C}^{1}(\mathbb{R}^{*}_{+})$ positive function. Also, we assume that $\phi$ is a $\mathcal{C}^{1}(\mathbb{R}^{+})$ non-decreasing function on $\mathbb{R}^{+}$, verifying $\phi(0)=0$ and $\lim_{x \rightarrow + \infty} \phi(x) =1$. Finally, we suppose that $\phi$ is Lipschitz continuous on $\mathbb{R}^{+}$, meaning that
\begin{align}
\exists k >0, \forall R_{1},R_{2} \geq 0, \forall x>0, \quad |f(x,R_{1})-f(x,R_{2})| \leq k|R_{1}-R_{2}|\psi(x).
\label{eq:lipshitzphi}
\end{align}
\end{enumerate}
Thus, between two random jump times (due to birth or death events), the energy evolves according to the following equation:
\begin{align}
    \dfrac{\mathrm{d}\xi^{u}_{t}}{\mathrm{d}t} =  f(\xi^{u}_{t},R_{t}) - \ell(\xi^{u}_{t}) = : g(\xi^{u}_{t},R_{t}).
    \label{eq:indivenergy}
\end{align}
Remark that $g : (x,R) \mapsto g(x,R)$ is $\mathcal{C}^{1,1}(\mathbb{R}^{*}_{+}\times \mathbb{R}^{+})$, meaning that it is differentiable with continuous derivatives in both its variables. Also, one can replace $f$ by $g$ in \eqref{eq:lipshitzphi}. It is possible that at some time $t \geq 0$, $\xi^{u}_{t}$ reaches either 0 or $+ \infty$ (vanishing energy or explosion of \eqref{eq:indivenergy} in finite time before a jump event). We will make assumptions in Section~\ref{subsec:assumptions} to avoid this situation almost surely. For now, if this happens at some time $t$, we adopt the convention $\xi^{w}_{s} := \partial$ for every $s \geq t$ and $w \in V_{t-}$. Also, we set the point measure to $\mu_{s} :=0$ for all $s \geq t$, and $V_{s} := \varnothing $ for $s \geq t$. Finally, for any $t \geq 0$ and $u \in \mathcal{U}$, if $u \notin V_{t}$, we set $\xi^{u}_{t} := \partial$ (in particular, individuals that are not born yet at time $t$ have energy $\partial$).
\\\\
\textbf{Resource dynamics}
\\\\
Let  $\chi > 1$, $R_{\max}>0$ and $\varsigma \in \mathcal{C}^{1}(\mathbb{R}^{+})$, such that $\varsigma(0)\geq 0$ and $\varsigma(R)<0$ if $R \geq R_{\max}$. Between random jumps, the quantity of resource $R_{t} \in \mathbb{R}^{+}$ satisfies the following equation
\begin{align}
\label{eq:eqress}
    \dfrac{\mathrm{d}R_{t}}{\mathrm{d}t} = \rho(\mu_{t},R_{t}) := \varsigma(R_{t}) - \chi \displaystyle{\int_{\mathbb{R}^{*}_{+}} f(x,R_{t}) \mu_{t}(\mathrm{d}x)}.
\end{align}
The function $\rho$ is well-defined on $\mathcal{M}_{P}(\mathbb{R}^{*}_{+}) \times \mathbb{R}^{+}$, where $\mathcal{M}_{P}(\mathbb{R}^{*}_{+})$ is the space of finite point measures on $\mathbb{R}^{*}_{+}$. Equation \eqref{eq:eqress} means that in the absence of individuals in the population, the amount of resource eventually stabilizes at some $R_{\mathrm{eq}} \in [0,R_{\max}]$. For the renewal function $\varsigma$, one can think for example of a logistic growth $R \geq 0 \mapsto R(R_{\mathrm{eq}}-R)$ with $R_{\mathrm{eq}} \leq R_{\max}$. This is a classical assumption in literature for biotic resources \cite{brannstrom2011,yeakel2018dynamics, fritsch:hal-02129313}. Another example is the case of a chemostat, where we can take $\varsigma : R \geq 0 \mapsto D(R_{\mathrm{in}}-R)$, with a dilution rate $D>0$, and the constant $R_{\mathrm{in}} \leq R_{\max}$ can be interpreted as an abiotic nutrient flow in the chemostat \cite{FRITSCH20151}. The coefficient $\chi$ can be interpreted as the inverse of the conversion efficiency. The ratio $1/\chi <1$ represents the proportion of resource consumed by individuals effectively converted to energy. We suppose that $\chi$ is a constant, which is a usual assumption in literature \cite{loeuille2005}, even if in \cite{fritsch:hal-02129313}, the authors make $\chi$ depend on individual energy over time. The integral quantity represents the speed at which the whole population consumes the resource at time $t$. This non-linear term is an indirect source of competition between individuals.
\\\\
If we choose an initial condition $R_{0} \in [0,R_{\max}]$, Equation~\eqref{eq:eqress} enforces that $R_{t} \in [0,R_{\max}]$ for every $t \geq 0$, as long as $\displaystyle{\int_{\mathbb{R}^{*}_{+}} \psi(x) \mu_{t}(\mathrm{d}x)} < + \infty$ (in particular, this is the case if $\mu_{t} \in \mathcal{M}_{P}(\mathbb{R}^{*}_{+})$). Moreover, if $R_{0} > R_{\max}$, even with no individuals in the population, the resource will decrease to $R_{\max}$. Hence, without loss of generality, we assume that $R_{0} \in [0,R_{\max}]$. Also, we make the following assumption.
\begin{hyp}
$\forall x>0 \quad g(x,R_{\max})>0.$
\label{hyp:gainenergy}
\end{hyp}
If Assumption~\ref{hyp:gainenergy} is not verified, as $\phi$ is non-decreasing, we obtain
$$ \exists M>0, \forall x \leq M, \forall R \in [0,R_{\max}] \quad g(x,R) \leq 0.$$
Considering \eqref{eq:indivenergy}, this means that if the initial energy of an individual is in $(0,M]$, then it will remain in this compact set over time. In other terms, if Assumption~\ref{hyp:gainenergy} is not verified, we impose an \textit{a priori} upper bound on the maximal energy in our model. Similar mass-structured models where the maximal mass $M$ of an individual is deterministically bounded are already developed in previous works \cite{FRITSCH20151,coranicofab}. Although observed living species obviously have bounded masses, we want to design a model where this bound is not artificially imposed by the model, but results from interaction with a limiting resource. This is why in our setting, we allow individual energies to increase indefinitely, at least if there are sufficient resources (\textit{i.e.} with $R_{\max}$ resources are available), which is expressed in Assumption~\ref{hyp:gainenergy}. In all the rest of this article, we implicitly work under Assumption~\ref{hyp:gainenergy}. Note that if $\ell$ is a positive function, Assumption~\ref{hyp:gainenergy} corresponds to Assumption $\mathbf{(H0)}$ in Theorem~\ref{theo:initial}.

\begin{lemme}
Assumption~\ref{hyp:gainenergy} implies that
\begin{align*}
\phi(R_{\max})>0 \quad \mathrm{and} \quad \forall x>0, \; \psi(x) > \ell(x).
\end{align*}
\label{lemme:gainenergy}
\end{lemme}
\begin{proof}
\vspace{-0.5cm}
This is straightforward from our assumptions on $\psi$, $\phi$ and $\ell$.
\end{proof}

\subsection{Algorithmic construction of the process with Poisson point measures}
\label{sec:construction}

\subsubsection{Deterministic flow between random jumps}
\label{subssub:wouah}

We begin with the definition of the deterministic flow associated to individual energies and the amount of resources between random jumps. We provide $\mathcal{U}$ with the lexicographical order, denoted as $\prec$, and consider a finite subset $V \subseteq \mathcal{U}$ of cardinality $|V| \in \mathbb{N}$. It means that there exists $u_{1} \prec ... \prec u_{|V|}$ elements of $\mathcal{U}$, such that $V= \{u_{1},...,u_{|V|} \}$, with $V = \varnothing$ if $|V|=0$. Let us fix an initial condition $R_{0} \in [0,R_{\max}]$ and $(\xi^{u_{j}}_{0})_{1 \leq j \leq |V|} \in \left(\mathbb{R}^{*}_{+}\right)^{|V|}$ individual energies indexed by $V$. In the following, we will lighten this notation into $\Xi_{0} := (\xi^{u}_{0})_{u \in V}$. We write $\left((X^{u}_{t}(\Xi_{0},R_{0}))_{ u \in V},X^{\Re}_{t}(\Xi_{0},R_{0})\right)$ for a solution to the system of $|V|+1$ coupled equations 
\begin{align}
    \dfrac{\mathrm{d}X^{\Re}_{t}(\Xi_{0},R_{0})}{\mathrm{d}t} & = \rho\left( \sum\limits_{u \in V}{\delta_{X^{u}_{t}(\Xi_{0},R_{0})}}, X^{\Re}_{t}(\Xi_{0},R_{0})\right),\label{eq:regroupindiv} \\
    \dfrac{\mathrm{d}X^{u}_{t}(\Xi_{0},R_{0})}{\mathrm{d}t} & = g(X^{u}_{t}(\Xi_{0},R_{0}),X^{\Re}_{t}(\Xi_{0},R_{0})) \quad \mathrm{for} \hspace{0.1 cm} u \in V, \label{regeoupee}
\end{align}
with initial condition at time 0
\begin{center}
    $\begin{array}{lll}
    X^{\Re}_{0}(\Xi_{0},R_{0}) & = & R_{0}, \\
    X^{u}_{0}(\Xi_{0},R_{0}) & = & \xi^{u}_{0} \quad \mathrm{for} \hspace{0.1 cm} u \in V. \\
\end{array}$
\end{center}
\begin{prop}
Let $V \subseteq \mathcal{U}$ be finite and $((\xi^{u}_{0})_{u \in V},R_{0}) \in (\mathbb{R}^{*}_{+})^{|V|} \times [0,R_{\max}]$. Then, there exists a neighborhood $O \subseteq (\mathbb{R}^{*}_{+})^{|V|} \times [0,R_{\max}]$ of $((\xi^{u}_{0})_{u \in V},R_{0})$, and a neighborhood $J \subseteq \mathbb{R}^{+}$ of 0, such that
\begin{itemize}
\item[1.] For every $((\xi^{u})_{u \in V},R) \in O$, there exists a unique local solution with values in $(\mathbb{R}^{*}_{+})^{|V|} \times [0,R_{\max}]$ to the system of coupled equations \eqref{eq:regroupindiv}-\eqref{regeoupee}, starting from $((\xi^{u})_{u \in V},R)$ at time 0. This solution is at least defined on $J$, and denoted as $\left((X^{u}_{t}(\Xi,R))_{u \in V},X^{\Re}_{t}(\Xi,R)\right)$, with $\Xi := (\xi^{u})_{u \in V}$. 
\item[2.] The function $(t,(\xi^{u})_{u \in V},R) \in J \times O \mapsto \left((X^{u}_{t}(\Xi,R))_{u \in V},X^{\Re}_{t}(\Xi,R)\right)$ is $\mathcal{C}^{1,1}(J \times O)$, and $\mathcal{C}^{2}$ in the variable $t$.
\end{itemize}
\label{prop:cauchy}
\end{prop}
\begin{proof}
As $\varsigma \in \mathcal{C}^{1}(\mathbb{R}^{+})$, and $f$, $g$ are $\mathcal{C}^{1,1}(\mathbb{R}^{*}_{+}\times \mathbb{R}^{+})$, classical arguments entails the result (see Corollaire II.2. and Théorème II.10. in Chapter X of \cite{zuily1996elements}).
\end{proof}
\textbf{Remark:} The previous objects do not depend on the set of indices $V$. We use these notations to be able to identify any individual by an index and keep track of its energy over time, in the upcoming construction of our population process. 
\\\\
We introduce $t_{\mathrm{exp}}(\Xi_{0},R_{0}) \in (0,+\infty]$ the maximal time of existence of the solution to \eqref{eq:regroupindiv}-\eqref{regeoupee} starting from $(\Xi_{0},R_{0})$ at time 0 highlighted in Proposition~\ref{prop:cauchy}. The deterministic time $t_{\mathrm{exp}}(\Xi_{0},R_{0})$ is finite, if and only if one of the $X^{u}_{t}(\Xi_{0},R_{0})$ reaches 0 or $+ \infty$ in finite time.~Finally, we define a flow $X$ with a measure-valued first component, to be able to use it in the upcoming definition of the stochastic measure-valued process $(\mu_{t},R_{t})_{t}$. Suppose that $\mu_{0} \in \mathcal{M}_{P}(\mathbb{R}^{*}_{+})$ is such that
$$ \mu_{0} = \sum_{u \in V} \delta_{\xi^{u}_{0}},$$
with $\mu_{0}=0$ if $V=\varnothing$. Then, in the following, we write $t_{\exp}(\mu_{0},R_{0}) := t_{\exp}(\Xi_{0},R_{0})$, and this does not depend on the set of indices $V$. Also, for $t \in [0,t_{\mathrm{exp}}(\mu_{0},R_{0}))$, we define
\begin{center}
    $X_{t}(\mu_{0},R_{0}) := \left( \sum\limits_{u \in V}{\delta_{X^{u}_{t}(\Xi_{0},R_{0})}}, X^{\Re}_{t}(\Xi_{0},R_{0})\right),$
\end{center}
and this again does not depend on the indexing by $V$. Finally, we adopt the \hypertarget{conv}{convention} $X_{t}(\mu_{0},R_{0}) := (0,0)$ for every $ t \geq t_{\mathrm{exp}}(\mu_{0},R_{0})$. We thus have defined the deterministic flow with measure-valued first component 
\begin{align*}
    \begin{array}{ccrcl}
X & : & \mathcal{M}_{P}(\mathbb{R}^{*}_{+}) \times \mathbb{R}^{+} \times \mathbb{R}^{+} & \to & \mathcal{M}_{P}(\mathbb{R}_{+}^{*}) \times \mathbb{R}^{+}  \\
 & & (\mu,R,t) & \mapsto & X_{t}(\mu,R). \\
\end{array}
\end{align*}
In the following, for $\mu$ measure on $\mathbb{R}^{*}_{+}$ and $f$ measurable from $\mathbb{R}^{*}_{+}$ to $\mathbb{R}$, we write $$\langle \mu, f \rangle := \displaystyle{\int_{\mathbb{R}^{*}_{+}} f \mathrm{d}\mu }.$$

\subsubsection{Algorithmic construction of the process}
\label{sec:inductive}

First, we define individual energies $((\xi^{u}_{t})_{u \in \mathcal{U}})_{t}$, the set of alive individuals $(V_{t})_{t}$ and the amount of resources $(R_{t})_{t}$ inductively, by constructing a sequence of successive random jump times $(J_{n})_{n \geq 0}$, between which the dynamics are deterministic (note immediately that our construction will then be valid only up to time $\sup_{n \in \mathbb{N}} J_{n}$). Then, we gather individual processes into a measure-valued process $(\mu_{t})_{t}$. Between two jump times, the process $(\mu_{t},R_{t})_{t}$ will be deterministic and will follow the flow with measure-valued first component $X$ defined in Section~\ref{subssub:wouah}. We consider $\mathcal{N}(\mathrm{d}s,\mathrm{d}u,\mathrm{d}h)$ and $\mathcal{N}'(\mathrm{d}s,\mathrm{d}u,\mathrm{d}h)$ two independent Poisson point measures on $\mathbb{R}^{+} \times \mathcal{U} \times \mathbb{R}^{*}_{+}$, with intensity $\mathrm{d}s \times n(\mathrm{d}u) \times \mathrm{d}h$, with $n(\mathrm{d}u) := \sum_{w \in \mathcal{U}}{\delta_{w}(\mathrm{d}u)}$. The support of $\mathcal{N}$, respectively $\mathcal{N}'$, on $\mathbb{R}^{+} \times \mathcal{U} \times \mathbb{R}^{*}_{+}$ is a countable random set, denoted as $\mathrm{supp}(\mathcal{N})$, respectively $\mathrm{supp}(\mathcal{N}')$. This is a random variable verifying $\mathcal{N}(\mathrm{d}s,\mathrm{d}u,\mathrm{d}h) = \sum_{(s,u,h) \in \mathrm{supp}(\mathcal{N}) } \delta_{(s,u,h)}$, respectively the same equation with $\mathcal{N}'$.
\\\\
For the initial condition, let $(\mu_{0},R_{0})$ be a random variable taking values in $ \mathcal{M}_{P}(\mathbb{R}^{*}_{+}) \times [0,R_{\max}]$. We define $N_{0} := \langle \mu_{0}, 1 \rangle$ the initial number of individuals, $V_{0} := \{ 1,..., N_{0} \} \subseteq \mathcal{U}$, and $\Xi_{0} := (\xi^{u}_{0})_{u \in V_{0}}$ the initial individual energies. Thus, we index alive individuals at time 0 so that $\mu_{0} := \sum_{u \in V_{0}} \delta_{\xi^{u}_{0}}.$ Also, for $u \notin V_{0}$, we set $\xi^{u}_{0} := \partial$. The Poisson point measures $\mathcal{N}$ and $\mathcal{N}'$ are independent from $(\mu_{0},R_{0})$. The canonical filtration associated to $(\mu_{0},R_{0})$, $\mathcal{N}$ and $\mathcal{N}'$ is $(\mathcal{F}_{t})_{t \geq 0}$. 
\\\\
We now define the sequence $(J_{n})_{n \in \mathbb{N}}$ of successive jump times of the population process. First, we set $J_{0} := 0$, and then suppose that our process is described until some time $J_{n} < + \infty$ with $n \geq 0$. At time $J_{n}$, there exists a finite $V_{J_{n}} \subseteq \mathcal{U}$, associated to individual energies $\Xi_{J_{n}} := (\xi^{u}_{J_{n}})_{u \in V_{J_{n}}}$. With the convention $\inf(\varnothing) = + \infty$, we define
\begin{align*}
J^{b}_{n+1} &:= \inf \{ t \in (J_{n}, J_{n} + t_{\exp}(\Xi_{J_{n}},R_{J_{n}})), \, (t,u,h) \in \mathrm{supp}(\mathcal{N}), \\
& \hspace{6 cm} u \in V_{J_{n}}, \, h \leq b\left(X^{u}_{t-J_{n}}(\Xi_{J_{n}},R_{J_{n}})\right)  \},\\
J^{d}_{n+1} & := \inf \{t \in (J_{n}, J_{n} + t_{\exp}(\Xi_{J_{n}},R_{J_{n}})), \, (t,u,h) \in \mathrm{supp}(\mathcal{N}'), \\
& \hspace{6cm} u \in V_{J_{n}}, \, h \leq d\left(X^{u}_{t-J_{n}}(\Xi_{J_{n}},R_{J_{n}})\right)\}.
\end{align*} 
First if $J_{n} + t_{\exp}(\Xi_{J_{n}},R_{J_{n}}) = J^{b}_{n+1} \wedge J^{d}_{n+1} = + \infty$, it means that there are no jumps anymore, and no explosion of the solution to \eqref{eq:regroupindiv}-\eqref{regeoupee} starting from $(\Xi_{J_{n}},R_{J_{n}})$. Then, we set $J_{n+1} = + \infty$, and for $t \geq J_{n}$, $V_{t} = V_{J_{n}}$. Concerning individual energies, for $t \geq J_{n}$, if $u \notin V_{J_{n}}$, then $\xi^{u}_{t} = \partial$, and
$$\left((\xi^{u}_{t})_{u \in V_{J_{n}}},R_{t}\right) = \left((X^{u}_{t-J_{n}}(\Xi_{J_{n}},R_{J_{n}}))_{u \in V_{J_{n}}},X^{\Re}_{t-J_{n}}(\Xi_{J_{n}},R_{J_{n}})\right).$$
Else if $J_{n} + t_{\exp}(\Xi_{J_{n}},R_{J_{n}}) < J^{b}_{n+1} \wedge J^{d}_{n+1} = + \infty$, it means that one or several individual energies reach 0 or $+ \infty$ at time $J_{n} + t_{\exp}(\Xi_{J_{n}},R_{J_{n}})$ (which is an event that we will avoid almost surely in Section~\ref{subsec:assumptions}). Then, we set $J_{n+1}=J_{n} + t_{\exp}(\Xi_{J_{n}},R_{J_{n}})$, and for $t \in (J_{n},J_{n+1})$, we set $V_{t} = V_{J_{n}}$ and if $u \notin V_{J_{n}}$, then $\xi^{u}_{t} = \partial$. In addition, for $t \in (J_{n},J_{n+1})$, we set
$$\left((\xi^{u}_{t})_{u \in V_{J_{n}}},R_{t}\right) = \left((X^{u}_{t-J_{n}}(\Xi_{J_{n}},R_{J_{n}}))_{u \in V_{J_{n}}},X^{\Re}_{t-J_{n}}(\Xi_{J_{n}},R_{J_{n}})\right).$$
Next, we set $V_{J_{n+1}} = \varnothing$, $R_{J_{n+1}} = R_{J_{n+1}-} $ and for $u \in \mathcal{U}$, $\xi^{u}_{J_{n+1}} = \partial$. Remark that with these conventions, we necessarily get back to the first case described above for the definition of $J_{n+2}$ (so $J_{n+2} = +\infty$), and eventually obtain that for all $t \geq J_{n+1}$, $V_{t} = \varnothing$ and for $u \in \mathcal{U}$, $\xi^{u}_{t} = \partial$.
\\\\
Finally, if $J^{b}_{n+1} \wedge J^{d}_{n+1} < J_{n} + t_{\exp}(\Xi_{J_{n}},R_{J_{n}})$, it means that one birth or death event occurs. By property of Poisson point measures, we almost surely have $J^{b}_{n+1} \neq J^{d}_{n+1}$, and the infimum in the definition of $J^{b}_{n+1}$ or $J^{d}_{n+1}$ is reached at a single element $(t,u,h) \in \mathrm{supp}(\mathcal{N})$ or $\mathrm{supp}(\mathcal{N}')$. We distinguish again between two cases.
\begin{itemize}
\item First if $J^{b}_{n+1} < J^{d}_{n+1}$, it means that one birth event occurs. This event concerns an individual indexed by some $w \in V_{J_{n}}$. Then, we set $J_{n+1} = J^{b}_{n+1}$, and for all $t \in [J_{n},J_{n+1})$, $V_{t}=V_{J_{n}}$. Concerning individual energies, if $u \notin V_{J_{n}}$, we set $\xi^{u}_{t}  = \partial$ and 
$$\left((\xi^{u}_{t})_{u \in V_{J_{n}}},R_{t}\right) = \left((X^{u}_{t-J_{n}}(\Xi_{J_{n}},R_{J_{n}}))_{u \in V_{J_{n}}},X^{\Re}_{t-J_{n}}(\Xi_{J_{n}},R_{J_{n}})\right).$$
Then, at time $J_{n+1}$, a new individual appears in the population. We set $V_{J_{n+1}} = V_{J_{n}} \cup \{ wk \} $, where $k-1$ is the number of offspring individual $w$ already produced (\textit{i.e.} the cardinality of the set $\{ (t,u,h) \in \mathrm{supp}(\mathcal{N}), \, u=w, \; \exists \, 1 \leq m \leq n, \, J^{b}_{m} = t \}$). We also set $\xi^{wk}_{J_{n+1}} = x_{0}$, $\xi^{w}_{J_{n+1}} = \xi^{w}_{J_{n+1}-}-x_{0}$ and $\xi^{u}_{J_{n+1}} = \xi^{u}_{J_{n+1}-}$ for $u \in \mathcal{U} \setminus \{w,wk\}$. Finally, $R_{J_{n+1}} = R_{J_{n+1}-}$.
\item Else if $J^{d}_{n+1} < J^{b}_{n+1}$, it means that a death event occurs. This event concerns an individual indexed by some $w \in V_{J_{n}}$. Then, we set $J_{n+1} = J^{d}_{n+1}$, and for all $t \in [J_{n},J_{n+1})$, $V_{t}=V_{J_{n}}$. Concerning individual energies, if $u \notin V_{J_{n}}$, we set $\xi^{u}_{t}  = \partial$ and 
$$\left((\xi^{u}_{t})_{u \in V_{J_{n}}},R_{t}\right) = \left((X^{u}_{t-J_{n}}(\Xi_{J_{n}},R_{J_{n}}))_{u \in V_{J_{n}}},X^{\Re}_{t-J_{n}}(\Xi_{J_{n}},R_{J_{n}})\right).$$
Then, at time $J_{n+1}$, individual $w$ disappears from the population. We set $V_{J_{n+1}} = V_{J_{n}} \setminus \{ w\}$. We also set $\xi^{w}_{J_{n+1}} = \partial$ and $\xi^{u}_{J_{n+1}} = \xi^{u}_{J_{n+1}-}$ for $u \in \mathcal{U} \setminus \{w\}$. Finally, $R_{J_{n+1}} = R_{J_{n+1}-}$.
\end{itemize}
By convention, for all $n \in \mathbb{N}$, $J_{n+1}^{b}=J^{d}_{n+1} = J_{n+1} = + \infty$ if $J_{n}=+ \infty$. The sequence $(J_{n})_{n \in \mathbb{N}^{*}}$ is non-decreasing, so we can define $J_{\infty} := \lim_{n \rightarrow + \infty} J_{n}$. Eventually, for every $t \in [0,J_{\infty})$, we define $\mu_{t}$ as in \eqref{eq:mudemfu}, with $\mu_{t}=0$ if $V_{t}= \varnothing$. We verify immediately that, if no individual energy reaches 0 or explodes, for all $n \in \mathbb{N},$ for all $t \in [J_{n},J_{n+1})$, $\mu_{t}$ coincide with the deterministic flow $X_{t-J_{n}}(\mu_{J_{n}},R_{J_{n}})$, and is modified at any jump time according to the rules given in Section~\ref{subsec:dyn}. Note that $\mu_{t}$ (as well as the $(\xi^{u}_{t})_{u \in \mathcal{U}}$, $V_{t}$ and $R_{t}$) is well-defined only for $t \in [0,J_{\infty})$, with $J_{\infty}$ possibly finite or infinite. For all $t \in [0,J_{\infty})$, $V_{t}$ is the set containing the indices of alive individuals at time $t$ (if $u \notin V_{t}$, then $\xi^{u}_{t} = \partial$; and if $u \in V_{t}$, then $\xi^{u}_{t} \in \mathbb{R}^{*}_{+}$). The previously described update rules of the set $V_{t}$ at each jump event make it adapted with respect to the filtration $(\mathcal{F}_{t})_{t}$. In the following, we want to avoid almost surely the following situations:
\begin{itemize}
\item[] \hypertarget{sitone}{\textbf{Situation 1}} One of the individual energy vanishes/explodes.
\item[] \hypertarget{sitwo}{\textbf{Situation 2}} There is an accumulation of jump times.
\end{itemize}
We define
$$ \tau_{\exp} := \inf \{ J_{n}, \; J_{n+1}=J_{n} + t_{\exp}(\Xi_{J_{n}},R_{J_{n}}) < + \infty \}, $$
with the convention $\inf(\varnothing) = + \infty$. To avoid \hyperlink{sitone}{Situation 1}, respectively \hyperlink{sitwo}{Situation 2}, we need to ensure that almost surely, $\tau_{\exp} = + \infty$, respectively $J_{\infty} = + \infty.$ In Section~\ref{subsec:assumptions}, we give an assumption under which $\tau_{\exp} = + \infty$ holds true almost surely. In Section~\ref{subsec:refinement}, we work under the assumption of Section~\ref{subsec:assumptions}, and we define a general setting under which $J_{\infty} = + \infty$ holds true almost surely.

\subsubsection{Classical writing of the process on $[0,J_{\infty} \wedge \tau_{\exp})$}
\label{subsec:classicalw}

Before time $\tau_{\exp}$, individual energies never vanish or explode, hence the global jump rate of the population is finite at any time in $[0,J_{\infty} \wedge \tau_{\exp})$. Thus, for every $J_{n}<\tau_{\exp}$, the next jump time $J_{n+1}$ is either $+ \infty$ or a birth/death jump in the inductive construction of Section~\ref{sec:inductive}. Hence, for every $t \in [0,J_{\infty}\wedge \tau_{\exp})$, we can write, with the \hyperlink{conv}{convention} on the flow $X$ in mind,
\begin{align*}
(\mu_{t},R_{t}) = & \hspace{0.1 cm}  X_{t}(\mu_{0},R_{0}) \\
  & + \displaystyle{\int_{0}^{t}\int_{\mathcal{U} \times \mathbb{R}^{*}_{+} } \mathbb{1}_{\{u \in V_{s-} \}} \mathbb{1}_{\{h \leq b(\xi^{u}_{s-})\}}}  [ X_{t-s}(\mu_{s-} + \delta_{x_{0}} + \delta_{\xi^{u}_{s-}-x_{0}} -\delta_{\xi^{u}_{s-}},R_{s})  \\
    & \hspace{3 cm} - X_{t-s}(\mu_{s-},R_{s}) ] \mathcal{N}(\mathrm{d}s,\mathrm{d}u,\mathrm{d}h) & \\
     & + 
    \displaystyle{\int_{0}^{t}\int_{\mathcal{U} \times \mathbb{R}^{*}_{+} } \mathbb{1}_{\{ u \in V_{s-}\}} \mathbb{1}_{\{h \leq d(\xi^{u}_{s-})\}} [X_{t-s}(\mu_{s-} - \delta_{\xi^{u}_{s-}},R_{s})}  \\ 
    &  \hspace{3 cm} - X_{t-s}(\mu_{s-},R_{s}) ] \mathcal{N}'(\mathrm{d}s,\mathrm{d}u,\mathrm{d}h). \\
\end{align*}
This formal writing is classical in the literature (Definition 2.4. in \cite{chi_08}, Section 4.1 in  \cite{FRITSCH20151}), and should be understood as a telescopic sum. First, individual energies and resources evolve deterministically, following the flow $X_{t}(\mu_{0},R_{0})$. Then, at each birth or death event, we erase the current flow and replace it with a new flow, modified according to our birth and death rules.

\subsection{Assumption for non vanishing/exploding individual energies}
\label{subsec:assumptions}

In this section, we provide a framework under which $\tau_{\exp}=+ \infty$ almost surely. Recall that this event occurs if no individual energy reaches 0 or $+ \infty$ in finite time. First, without further assumption, we will show that interaction with limiting resources prevents individual energies from reaching $+ \infty$ in finite time. Then, we will introduce an additional assumption to prevent individual energies from reaching 0 in finite time (see Assumption~\ref{hyp:probamortel}). We begin with a classical result, associated to the writing of the process in Section~\ref{subsec:classicalw} as a telescopic sum (see for example Proposition 4.1 in \cite{FRITSCH20151}), and introduce some notations for this purpose.
\\\\
Let $\varphi : (t,x) \mapsto \varphi_{t}(x)$ be a $\mathcal{C}^{1,1}(\mathbb{R}^{+} \times \mathbb{R}^{*}_{+})$ function, which means that it is differentiable with continuous derivatives in both its variables. Recall that we write $\partial_{1}\varphi$, respectively $\partial_{2}\varphi$, for the first, respectively second partial derivative, and also for every $t \geq 0$ we define 
\begin{align}
\Phi_{t} : (R,x) \in \mathbb{R}^{+} \times \mathbb{R}^{*}_{+} \mapsto \partial_{1}\varphi(t,x) + g(x,R)\partial_{2}\varphi(t,x).
\label{eq:Phidef}
\end{align}
Note that $\Phi$ depends on $\varphi$, but to lighten the notations, we choose to write it this way in all the rest of this article. The notation $\Phi$ will always be related to the definition in \eqref{eq:Phidef} with a function $\varphi$ we work with without ambiguity.
\begin{lemme}
Let $\varphi \in \mathcal{C}^{1,1}(\mathbb{R}^{+} \times \mathbb{R}^{*}_{+})$, and $F \in \mathcal{C}^{1,1}([0,R_{\max}] \times \mathbb{R})$. The process defined in Section~\ref{sec:inductive} verifies that, for all $t \in [0,J_{\infty} \wedge \tau_{\exp})$,
\begin{align*}
F(R_{t},\langle \mu_{t}, \varphi_{t} \rangle) = & \hspace{0.1 cm} F(R_{0},\langle \mu_{0}, \varphi_{0} \rangle) \\
    & + \displaystyle{\int_{0}^{t} \bigg[ \rho(\mu_{s},R_{s}) \partial_{1}F(R_{s},\langle \mu_{s}, \varphi_{s} \rangle) + \langle \mu_{s}, \Phi_{s}(R_{s},.) \rangle \partial_{2} F(R_{s},\langle \mu_{s}, \varphi_{s} \rangle) \bigg] \mathrm{d}s} \\
    & + \displaystyle{\int_{0}^{t}\int_{\mathcal{U} \times \mathbb{R}^{*}_{+} } \mathbb{1}_{\{u \in V_{s-} \}} \mathbb{1}_{\{h \leq b(\xi^{u}_{s-})\}}  \bigg( F(R_{s},\langle \mu_{s-}+ \delta_{x_{0}} + \delta_{\xi^{u}_{s-}-x_{0}}-\delta_{\xi^{u}_{s-}}, \varphi_{s} \rangle)} \\
    & \hspace{4cm}- F(R_{s}, \langle \mu_{s-}, \varphi_{s} \rangle) \bigg) \mathcal{N}(\mathrm{d}s,\mathrm{d}u,\mathrm{d}h) & \\
     & +
    \displaystyle{\int_{0}^{t}\int_{\mathcal{U} \times \mathbb{R}^{*}_{+} } \mathbb{1}_{\{u \in V_{s-} \}} \mathbb{1}_{\{h \leq d(\xi^{u}_{s-})\}} }\bigg( F(R_{s},\langle \mu_{s-}-\delta_{\xi^{u}_{s-}}, \varphi_{s} \rangle) \\
    & \hspace{4cm}- F(R_{s}, \langle \mu_{s-}, \varphi_{s} \rangle) \bigg) \mathcal{N}'(\mathrm{d}s,\mathrm{d}u,\mathrm{d}h),
\end{align*}
with $\Phi$ associated to $\varphi$ as in \eqref{eq:Phidef}.
\label{lemme:decompocun}
\end{lemme}
\begin{proof}
This is a classical result for a process written like in Section~\ref{subsec:classicalw}, valid pathwisely until the random time $J_{\infty} \wedge \tau_{\exp}$, because our process is a finite variation process (see Definition page 39 in \cite{protter2005stochastic}) with right-continuous sample paths on this time window. We use Itô's formula (Theorem 31 page 78 in \cite{protter2005stochastic}) and a classical decomposition of $\langle \mu_{t}, \varphi_{t} \rangle$ (see for example Proposition 4.1 in \cite{FRITSCH20151}). For a detailed proof, we refer the reader to Lemma II.1.10. in \cite{broduthesis}.
\end{proof}

In the following, for $t \in [0,J_{\infty} \wedge \tau_{\exp})$, we define $E_{t} := \langle \mu_{t}, \mathrm{Id} \rangle$ the total energy of the population at time $t$. In the dynamics described in \eqref{eq:eqress}, as $\varsigma$ is $\mathcal{C}^{1}$, the speed of renewal of $R_{t}$ is upper bounded by $||\varsigma||_{\infty,[0,R_{\max}]} < + \infty$, for every $t \geq 0$. This resource renewal is the only income of biomass into the system described by $(\mu_{t},R_{t})_{t}$, so this shall give us a control on the total biomass of the system.
\begin{prop}
The process defined in Section~\ref{sec:inductive} verifies that
\begin{align}
\forall t \in [0,J_{\infty} \wedge \tau_{\exp}), \quad R_{t}+E_{t} \leq R_{0}+E_{0}+t||\varsigma||_{\infty,[0,R_{\max}]} < + \infty.
\label{eq:controlressource}
\end{align}
\label{prop:controlebiomasse}
\end{prop}

\begin{proof}
Let $t \in [0,J_{\infty} \wedge \tau_{\exp})$, by Lemma~\ref{lemme:decompocun} applied to $F: (r,x) \mapsto r+x$ and $\varphi :(t,x) \mapsto x$, we have
\begin{align*}
R_{t} + E_{t} = & \hspace{0.1 cm} R_{0}+E_{0} + \displaystyle{\int_{0}^{t} \bigg[ \rho(\mu_{s},R_{s}) + \langle \mu_{s}, \Phi_{s}(R_{s},.) \rangle \bigg] \mathrm{d}s} \\
     & -
    \displaystyle{\int_{0}^{t}\int_{\mathcal{U} \times \mathbb{R}^{*}_{+} } \mathbb{1}_{\{u \in V_{s-} \}} \mathbb{1}_{\{h \leq d(\xi^{u}_{s-})\}} }\xi^{u}_{s-} \mathcal{N}'(\mathrm{d}s,\mathrm{d}u,\mathrm{d}h) \\
    \leq & \hspace{0.1 cm} R_{0}+E_{0} + \displaystyle{\int_{0}^{t} \varsigma(R_{s}) \mathrm{d}s},    
\end{align*}
which concludes. Note that we used in particular the fact that the constant $\chi$ in the definition \eqref{eq:eqress} of $\rho$ is larger than 1.
\end{proof}

This control of the total biomass of the system entails that individual energies almost surely do not explode in finite time. In order to gather all our intermediate results into the upcoming Proposition~\ref{prop:etapeun}, we now introduce an assumption to prevent individual energies from reaching 0 in finite time. We adopt the convention that for any $x>0$, if $\ell(x)=0$, then $\frac{d}{\ell}(x) = +\infty$.
\begin{hyp}[\textbf{Individual energy avoids 0}]
For all $x>0$,
\begin{align*}
\displaystyle{\int_{0}^{x} \dfrac{d}{\ell}(y) \mathrm{d}y }  = + \infty.
\end{align*}
\label{hyp:probamortel}
\end{hyp}

The only random event that allows to avoid 0 is a death, and there is no energy gain in the worst possible case, which is $R=0$. Hence, it is natural to compare the death rate $d$ and the energy loss $\ell$ in a neighborhood of 0. Assumption~\ref{hyp:probamortel} expresses as an integral condition that the death rate $d$ should dominate the energy loss $\ell$ near 0. Note that if $\ell$ is a positive function, Assumption~\ref{hyp:probamortel} corresponds to Assumption $\mathbf{(H1)}$ in Theorem~\ref{theo:initial}.

\begin{prop}
Under Assumption~\ref{hyp:probamortel}, the process defined in Section~\ref{sec:inductive} verifies
$$ (J_{1}=+ \infty) \quad \mathrm{or} \quad (J_{1} < t_{\exp}(\mu_{0},R_{0})).$$
\label{prop:etapeun}
\end{prop}
The proof of Proposition~\ref{prop:etapeun} can be found in Appendix~\ref{app:prop27}. We conclude with the following corollary, using Markov property.

\begin{corr}
Under Assumption~\ref{hyp:probamortel}, the process defined in Section~\ref{sec:inductive} verifies $\tau_{\exp} = + \infty$ almost surely.
\label{prop:modelrelevant}
\end{corr}
\begin{proof}
By definition of $\tau_{\exp}$ and construction of the process, it suffices to show that for every $n \in \mathbb{N}$, we almost surely have
\begin{align*}
(J_{n+1}=+ \infty) \quad \mathrm{or} \quad (J_{n+1} < J_{n} + t_{\exp}(\mu_{J_{n}},R_{J_{n}})< + \infty).
\end{align*} 
Let us fix $n \in \mathbb{N}$ and work in the following conditionnally to the event $\{J_{n} + t_{\exp}(\mu_{J_{n}},R_{J_{n}}) < + \infty \} \cap \{ J_{n+1} < + \infty \}$. We aim to show that $J_{n+1} < J_{n} + t_{\exp}(\mu_{J_{n}},R_{J_{n}})$ almost surely. We define $(\tilde{\mu}_{0},\tilde{R}_{0})$ a random variable with same law as $(\mu_{J_{n}},R_{J_{n}})$, and $(\tilde{\mu}_{t},\tilde{R}_{t})_{t}$ a process starting from the random initial condition $(\tilde{\mu}_{0},\tilde{R}_{0})$ and constructed with the algorithmic procedure described in Section~\ref{sec:inductive}. The jump times with indices 0 and 1 associated to $(\tilde{\mu}_{t},\tilde{R}_{t})_{t}$ are naturally written $\tilde{J}_{0}$ and $\tilde{J}_{1}$, and note that by construction, $\tilde{J}_{0}:= 0$. Under the event $\{ J_{n+1} < + \infty \}$, from the strong Markov property for Poisson point processes (see Example 10.4(a) in \cite{daley2007introduction}), the law of $J_{n}+t_{\exp}(\mu_{J_{n}},R_{J_{n}})-J_{n+1}$ conditionnally to $\{J_{n}+t_{\exp}(\mu_{J_{n}},R_{J_{n}}) < + \infty \}$ and $\mathcal{F}_{J_{n}}$ is equal to the law of 
$\tilde{J}_{0}+t_{\exp}(\tilde{\mu}_{0},\tilde{R}_{0})-\tilde{J}_{1} = t_{\exp}(\tilde{\mu}_{0},\tilde{R}_{0})-\tilde{J}_{1}$, which concludes thanks to Proposition~\ref{prop:etapeun}.
\end{proof}

\subsection{Assumption for a well-defined population process $(\mu_{t})_{t}$ for every $t \geq 0$}
\label{subsec:refinement}

At this step, thanks to Assumption~\ref{hyp:probamortel}, we have $\tau_{\exp}=+ \infty$ almost surely (it is Corollary~\ref{prop:modelrelevant}), and we work under this event. The process $(\mu_{t},R_{t})_{t}$ is still well-defined only on $[0,J_{\infty})$, with $J_{\infty}$ possibly finite, \textit{i.e.} there is a possible accumulation of jump times. In Section~\ref{subsec:omega}, we assume the existence of an appropriate weight function $\omega$. Then, using this weight function, we give in Section~\ref{subsec:proofnonacc} a setting under which almost surely,  $J_{\infty} = + \infty$. We will even obtain in Proposition~\ref{prop:controlpop} a stronger result, which implies in particular that the expectation of the population size is finite for every $t \geq 0$. 

\subsubsection{Definition of the weight function $\omega$}
\label{subsec:omega}

We define a weight function $\omega$ adapted to the functional parameters $b$, $d$, $\phi$ and $\psi$ for two reasons. First, if such a weight function exists, we shall prove in Section~\ref{subsec:proofnonacc} that $(\mu_{t})_{t}$ is well-defined on $\mathbb{R}^{+}$. Then, we will obtain in Section~\ref{subsec:geninf} important martingale properties for our process. We write $\hbar : x \in \mathbb{R}^{+} \mapsto x+x^{2}$, and $\overline{g} : x >0 \mapsto \sup_{R \in [0,R_{\max}]} |g(x,R)|$. Note that for $x>0$, $\overline{g}(x) = \max(\ell(x), \phi(R_{\max})\psi(x)-\ell(x))$.
\begin{hyp}[\textbf{Existence of an appropriate weight function}]
There exists $\omega \in \mathcal{C}^{1}(\mathbb{R}^{*}_{+})$ positive and non-decreasing such that
\begin{itemize}
\item $\exists C_{g}>0, \forall x>0, \quad \overline{g}(x)(1+\omega'(x)) \leq C_{g}(1+x+\omega(x))$,
\item $\exists C_{b}>0, \forall x>0, \quad b(x)(1+\hbar\left(\left|\omega(x_{0})+\omega(x-x_{0}) -\omega(x)\right|\right)) \leq C_{b}(1+x+\omega(x)),$
\item $\exists C_{d}>0, \forall x>0, \quad d(x)\hbar(\omega(x)) \leq C_{d}(1+x+\omega(x)).$
\end{itemize} 
\label{hyp:poidsomega}
\end{hyp}

\textbf{Remark:} Recall that $b \equiv 0$ on $(0,x_{0}]$, so if $\omega$ is Lipschitz continuous on $(1,+\infty)$ (which is equivalent to $\omega'$ bounded on $(1,+\infty)$, and entails that $x>1 \mapsto \omega(x)/x$ is bounded on a neighborhood of $+ \infty$), the second point of Assumption~\ref{hyp:poidsomega} is equivalent to the lighter assumption
\begin{align}
\exists C_{b}>0, \forall x>1, \quad b(x) \leq C_{b}(1+x).
\label{eq:blipschitz}
\end{align}
Thus, we let the reader check that if $\mathbf{(H2)}$ in Theorem~\ref{theo:initial} is verified, then Assumption~\ref{hyp:poidsomega} holds true. Assumption $\mathbf{(H2)}$ is way more readable, though more restrictive on $\omega$. We believe, although it is not proven in this paper, that Assumption~\ref{hyp:poidsomega} on $\omega$ is sharp in our framework, in the sense that it is necessary and sufficient to obtain martingale properties for our process with the usual techniques developed originally in \cite{fournier2004microscopic}. In general, proving the existence of a weight function $\omega$ verifying Assumption~\ref{hyp:poidsomega} could be a difficult problem, similar to the search for Lyapunov functions associated to the extended generator of a Feller process (see for example condition (CD2) in Section 4.1. of \cite{meyn1993stability}, or Assumption (E) in Section 2 of \cite{champagnat2023general}).
\\\\
We present in Section~\ref{ex3} an \textit{allometric} setting with unbounded functional parameters, where jump rates are power functions, and prove that there exists weight functions $\omega$ verifying Assumption~\ref{hyp:poidsomega} in this specific context. It is of order $x^{\kappa_{1}}$, respectively $x^{\kappa_{2}}$, in a neighborhood of 0, respectively $+ \infty$, with $0 \leq \kappa_{1} \leq \kappa_{2} \leq 1$. The typical shape of the weight function $\omega$ in this setting is shown on Figure~\ref{fig:exampeomega}. 

\begin{figure}[h!]
\centering
\includegraphics[scale=0.5]{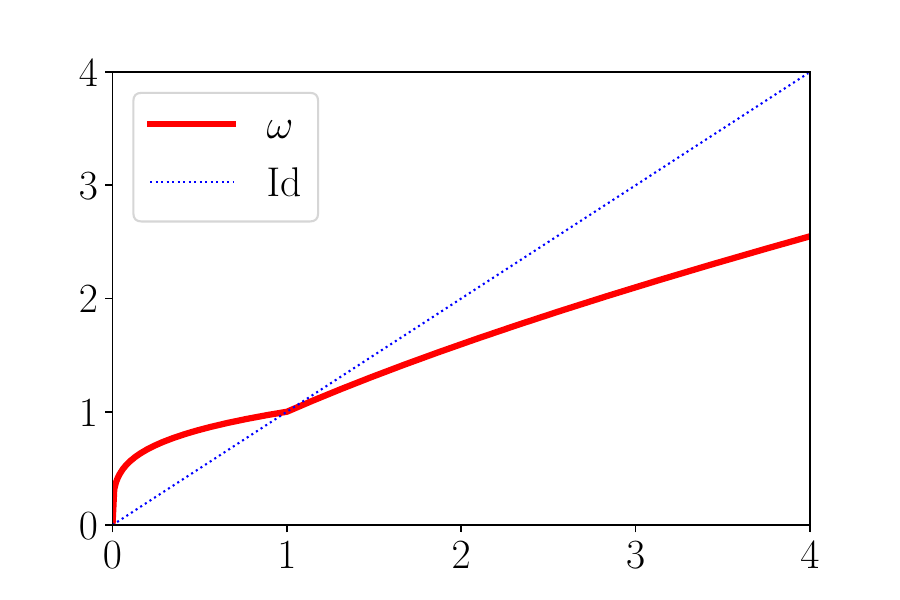} 
\caption{Possible shape of the weight function $\omega$ with an allometric choice of parameters.}
\label{fig:exampeomega}
\end{figure}

Notice that in general, the third point of Assumption~\ref{hyp:poidsomega} implies that $d\omega$ should be bounded in a neighborhood of 0. Hence, if $d(x) \xrightarrow[x \rightarrow  0]{} + \infty$, which can be a biological assumption for the death rate $d$, then $\omega(x) \xrightarrow[x \rightarrow  0]{} 0$ as shown on Figure~\ref{fig:exampeomega}. We will discuss in Section~\ref{sec:theorem} how this impacts the interpretation of the tightness result in Theorem~\ref{theo:convergence}.

\begin{lemme}
Assumption~\ref{hyp:poidsomega} is equivalent to the combination of the following properties. First, $\omega \in \mathcal{C}^{1}(\mathbb{R}^{*}_{+})$, is positive and non-decreasing, and then
\begin{multline}
 \exists \omega_{1}>0, \forall x>0, \\
 \overline{g}(x)(1+\omega'(x)) + b(x)\left(1+\left|\omega(x_{0})+\omega(x-x_{0}) -\omega(x)\right|\right) + d(x)\omega(x) \leq \omega_{1}(1+x+\omega(x)), \label{eq:conditionomega}
\end{multline}
and
\begin{align}
 \exists \omega_{2}>0, \forall x>0, \quad b(x)(\omega(x_{0})+\omega(x-x_{0}) -\omega(x))^{2}+ d(x)\omega^{2}(x) \leq \omega_{2} (1+x+\omega(x)).
\label{eq:conditionomegadeux}
\end{align}
\label{lemme:equivalence}
\end{lemme}
\begin{proof}
The fact that Assumption~\ref{hyp:poidsomega} implies \eqref{eq:conditionomega} and \eqref{eq:conditionomegadeux} is immediate. We can take $\omega_{1} = C_{g} + C_{b} + C_{d}$ and $\omega_{2} = C_{b} + C_{d}$. Conversely, if \eqref{eq:conditionomega} and \eqref{eq:conditionomegadeux} hold true, we can take $C_{g} = \omega_{1}$ and $C_{b} = C_{d} = \omega_{1}+ \omega_{2}$.
\end{proof} 
\textbf{Remark:}  Lemma~\ref{lemme:equivalence} gives us an insight on why Assumption~\ref{hyp:poidsomega} is interesting to obtain martingale properties for our process. Equation \eqref{eq:conditionomega} shall provide a control for the stochastic integrals appearing in the decomposition of $\langle \mu_{t}, \varphi_{t} \rangle$, using Lemma~\ref{lemme:decompocun} with $F : (r,x) \mapsto x$. Also, \eqref{eq:conditionomegadeux} is meant to control the quadratic variation of the martingale part of these stochastic integrals (see Corollary~\ref{corr:martingaleutile} in Section~\ref{subsec:geninf}). 

\subsubsection{Proof of the non-accumulation of jump times}
\label{subsec:proofnonacc}

For $t \in [0,J_{\infty})$, we define $N_{t} := \langle \mu_{t}, 1 \rangle$. It represents the number of individuals in the population at time $t$. We also define $\Omega_{t} := \langle \mu_{t}, \omega \rangle$, and for $M>0$, we introduce the stopping time
$$\tau_{M} := \inf \left\{ t \in [0,J_{\infty}), E_{t} + N_{t} + \Omega_{t} \geq M \right\},$$
with the convention $\inf(\varnothing) = + \infty$. Finally, for any positive function $w$ on $\mathbb{R}^{*}_{+}$, we write $\varphi \in \mathfrak{B}_{w}(\mathbb{R}^{*}_{+})$, if $\frac{\varphi}{w}$ is a bounded function on $\mathbb{R}^{*}_{+}$.

\begin{defi}[\hypertarget{gen}{\textbf{General setting}}]
In the following, we denote as `the general setting', the framework gathering the dynamics described in Sections~\ref{subsec:dyn} and \ref{sec:construction}, Assumptions~\ref{hyp:probamortel} and \ref{hyp:poidsomega}, and the additional assumption
$$\mathbb{E}(E_{0}+N_{0}+\Omega_{0}) <+ \infty.$$
\end{defi}
We are now ready to give the main result of this section.
\begin{prop}
Under the \hyperlink{gen}{general setting}, we almost surely have
\begin{align}
\tau_{M} \xrightarrow[M \rightarrow + \infty]{} + \infty.
\label{eq:taum}
\end{align} 
Then, for all $T \geq 0$, 
\begin{align}
\mathbb{E}\left(\sup\limits_{t \in [0,T \wedge J_{\infty})} (E_{t}+N_{t}+\Omega_{t})\right) < + \infty.
\label{eq:controleomega}
\end{align}
This immediately implies that for every $\varphi \in \mathfrak{B}_{1+\mathrm{Id}+\omega}(\mathbb{R}^{*}_{+})$, we have
\[\mathbb{E}\left(\underset{t \in [0,T \wedge J_{\infty})}{\mathrm{sup}} |\langle \mu_{t}, \varphi \rangle| \right) < + \infty. \]
\label{prop:controlpop}
\end{prop}

The proof of Proposition~\ref{prop:controlpop} can be found in Appendix~\ref{app:prop212}. It uses classical arguments, adapted to our \hyperlink{gen}{general setting} with the weight function $\omega$.

\begin{corr}
Under the \hyperlink{gen}{general setting}, the process $(\mu_{t})_{t}$ is almost surely well-defined for every $t \geq 0$, \textit{i.e.} $\mathbb{P}(\tau_{\exp}\wedge J_{\infty} = + \infty)=1$.
\label{corr:jinfini}
\end{corr}

\begin{proof}
This is a direct adaptation of the proof of point (i) of Theorem 3.1 in \cite{fournier2004microscopic}, using \eqref{eq:taum}.
\end{proof} 

\textbf{Remark:} It is possible to obtain the conclusion of Corollary~\ref{corr:jinfini} by replacing Assumption~\ref{hyp:poidsomega} with \eqref{eq:blipschitz}, using Proposition~\ref{prop:controlebiomasse} and classical arguments as in the proof of point (i) in Theorem 3.1 in \cite{fournier2004microscopic}. One can then wonder why we work under Assumption~\ref{hyp:poidsomega} instead of \eqref{eq:blipschitz}, which does not involve an additional weight function $\omega$. What makes Assumption~\ref{hyp:poidsomega} necessary in our work is that we also need martingale properties for our process (see Corollary~\ref{corr:martingaleutile}) to obtain the tightness result in Theorem~\ref{theo:convergence}.

\subsection{Martingale properties}
\label{subsec:geninf}

In this section, we work under the \hyperlink{gen}{general setting}, so the process $(\mu_{t},R_{t})_{t \geq 0}$ is almost surely well-defined thanks to Corollary~\ref{corr:jinfini}. We write $\mathfrak{F}$ for the set of functions of the form $F_{\varphi} : (\mu,r) \in \mathcal{M}_{P}(\mathbb{R}^{*}_{+}) \times [0,R_{\max}] \mapsto F(r,\langle \mu,\varphi \rangle) \in \mathbb{R}$, with $F :(r,x) \mapsto F(r,x)$ in $\mathcal{C}^{1,1}([0,R_{\max}] \times \mathbb{R})$ and $\varphi \in \mathcal{C}^{1}(\mathbb{R}^{*}_{+})$. For such a function, we write $\partial_{1}F_{\varphi}(\mu,r) := \partial_{1}F(r,\langle \mu, \varphi \rangle)$ and $\partial_{2}F_{\varphi}(\mu,r) := \partial_{2}F(r,\langle \mu, \varphi \rangle)$. Finally, for $F \in \mathcal{C}^{1,1}([0,R_{\max}] \times \mathbb{R})$ and $\varphi \in \mathcal{C}^{1,1}(\mathbb{R}^{+} \times \mathbb{R}^{*}_{+})$ (so that for all $t \geq 0$, $F_{\varphi_{t}}$ is well-defined in $\mathfrak{F}$), we define
\begin{multline}
\mathfrak{M}_{F,\varphi,t} :=  F_{\varphi_{t}}(R_{t},\mu_{t}) - \hspace{0.1 cm} F_{\varphi_{0}}(R_{0},\mu_{0}) \\
 \hspace{1.2cm} - \displaystyle{\int_{0}^{t} \rho(\mu_{s},R_{s}) \partial_{1}F_{\varphi_{s}}(R_{s},\mu_{s}) + \langle \mu_{s}, \Phi_{s}(R_{s},.) \rangle \partial_{2} F_{\varphi_{s}}(R_{s},\mu_{s}) \mathrm{d}s} \\
     \hspace{2.6cm}- \displaystyle{\int_{0}^{t}\int_{\mathbb{R}^{*}_{+} } b(x)}  \bigg[ F_{\varphi_{s}}(R_{s},\mu_{s}+ \delta_{x_{0}} + \delta_{x-x_{0}}-\delta_{x})- F_{\varphi_{s}}(R_{s}, \mu_{s}) \bigg] \mu_{s}(\mathrm{d}x) \mathrm{d}s \\
     - \displaystyle{\int_{0}^{t}\int_{\mathbb{R}^{*}_{+} } d(x)}  \bigg[ F_{\varphi_{s}}(R_{s},\mu_{s}-\delta_{x})- F_{\varphi_{s}}(R_{s}, \mu_{s}) \bigg] \mu_{s}(\mathrm{d}x) \mathrm{d}s,
    \label{eq:maringaloss}
\end{multline}
with $\Phi$ associated to $\varphi$ as in \eqref{eq:Phidef}. The process $(\mathfrak{M}_{F,\varphi,t})_{t \geq 0}$ is almost surely well-defined under the \hyperlink{gen}{general setting} by Lemma~\ref{lemme:decompocun} and Corollary~\ref{corr:jinfini}. In the following, quadratic variations of square-integrable martingales are predictable quadratic variation defined as in Theorem 4.2. in \cite{Jacod1987LimitTF}. Also, we write $\widetilde{\mathcal{N}}$ and $\widetilde{\mathcal{N}'}$ for the compensated measures associated with the Poisson point measures $\mathcal{N}$ and $\mathcal{N}'$ (\textit{i.e.} $\widetilde{\mathcal{N}}(\mathrm{d}s,\mathrm{d}u,\mathrm{d}h) := \mathcal{N}(\mathrm{d}s,\mathrm{d}u,\mathrm{d}h) - \mathrm{d}s \mathrm{d}u \mathrm{d}h$, and the same definition with $\mathcal{N}'$).
\begin{prop}
Under the \hyperlink{gen}{general setting}, let $F \in \mathcal{C}^{1,1}([0,R_{\max}] \times \mathbb{R})$ and $\varphi \in \mathcal{C}^{1,1}(\mathbb{R}^{+} \times \mathbb{R}^{*}_{+})$. 
\begin{itemize}
\item[$\mathrm{(i)}$] Assume that for all $t \geq 0$,
$$\mathbb{E}\left( \displaystyle{\int_{0}^{t}\int_{\mathbb{R}^{*}_{+}}}b(x)\bigg| F_{\varphi_{s}}(R_{s},\mu_{s}+\delta_{x_{0}}+\delta_{x-x_{0}}-\delta_{x}) - F_{\varphi_{s}}(R_{s},\mu_{s}) \bigg| \mu_{s}(\mathrm{d}x) \mathrm{d}s \right) < + \infty,$$
and
$$\mathbb{E}\left( \displaystyle{\int_{0}^{t}\int_{\mathbb{R}^{*}_{+}}}d(x)\bigg| F_{\varphi_{s}}(R_{s},\mu_{s}-\delta_{x}) - F_{\varphi_{s}}(R_{s},\mu_{s}) \bigg| \mu_{s}(\mathrm{d}x) \mathrm{d}s \right) < + \infty.$$
Then $(\mathfrak{M}_{F,\varphi,t})_{t \geq 0}$ is a $(\mathcal{F}_{t})_{t \geq 0}$-martingale.
\item[$\mathrm{(ii)}$] Suppose in addition that for all $t \geq 0$,
$$\mathbb{E}\left( \displaystyle{\int_{0}^{t}\int_{\mathbb{R}^{*}_{+}}}b(x)\bigg[ F_{\varphi_{s}}(R_{s},\mu_{s}+\delta_{x_{0}}+\delta_{x-x_{0}}-\delta_{x}) - F_{\varphi_{s}}(R_{s},\mu_{s}) \bigg]^{2} \mu_{s}(\mathrm{d}x) \mathrm{d}s \right) < + \infty,$$
and
$$\mathbb{E}\left( \displaystyle{\int_{0}^{t}\int_{\mathbb{R}^{*}_{+}}}d(x)\bigg[ F_{\varphi_{s}}(R_{s},\mu_{s}-\delta_{x}) - F_{\varphi_{s}}(R_{s},\mu_{s}) \bigg]^{2} \mu_{s}(\mathrm{d}x) \mathrm{d}s \right) < + \infty.$$
Then $(\mathfrak{M}_{F,\varphi,t})_{t \geq 0}$ is a square-integrable martingale, with predictable quadratic variation given for all $t \geq 0$ by
\begin{align*}
\left\langle \mathfrak{M}_{F,\varphi}  \right\rangle_{t}  := & \hspace{0.1 cm} \displaystyle{\int_{0}^{t}\int_{\mathbb{R}^{*}_{+}}}b(x)\bigg[ F_{\varphi_{s}}(R_{s},\mu_{s}+\delta_{x_{0}}+\delta_{x-x_{0}}-\delta_{x}) - F_{\varphi_{s}}(R_{s},\mu_{s}) \bigg]^{2} \mu_{s}(\mathrm{d}x) \mathrm{d}s  \\
& + \displaystyle{\int_{0}^{t}\int_{\mathbb{R}^{*}_{+}}}d(x)\bigg[ F_{\varphi_{s}}(R_{s},\mu_{s}-\delta_{x}) - F_{\varphi_{s}}(R_{s},\mu_{s}) \bigg]^{2} \mu_{s}(\mathrm{d}x) \mathrm{d}s
\end{align*}
\end{itemize}
\label{prop:martingaledeouf}
\end{prop}
\begin{proof}
We observe from Lemma~\ref{lemme:decompocun} that
\begin{align*}
\mathfrak{M}_{F,\varphi,t} & = \displaystyle{\int_{0}^{t}\int_{\mathcal{U} \times \mathbb{R}^{*}_{+} } \mathbb{1}_{\{u \in V_{s-} \}} \mathbb{1}_{\{h \leq b(\xi^{u}_{s-})\}}  \bigg( F(R_{s},\langle \mu_{s-}+ \delta_{x_{0}} + \delta_{\xi^{u}_{s-}-x_{0}}-\delta_{\xi^{u}_{s-}}, \varphi_{s} \rangle)} \\
    & \hspace{4cm}- F(R_{s}, \langle \mu_{s-}, \varphi_{s} \rangle) \bigg) \widetilde{\mathcal{N}}(\mathrm{d}s,\mathrm{d}u,\mathrm{d}h) & \\
     & \hspace{0.1 cm} +
    \displaystyle{\int_{0}^{t}\int_{\mathcal{U} \times \mathbb{R}^{*}_{+} } \mathbb{1}_{\{u \in V_{s-} \}} \mathbb{1}_{\{h \leq d(\xi^{u}_{s-})\}} }\bigg( F(R_{s},\langle \mu_{s-}-\delta_{\xi^{u}_{s-}}, \varphi_{s} \rangle) \\
    & \hspace{4cm}- F(R_{s}, \langle \mu_{s-}, \varphi_{s} \rangle) \bigg) \widetilde{\mathcal{N}'}(\mathrm{d}s,\mathrm{d}u,\mathrm{d}h).
\end{align*}
Then, Proposition~\ref{prop:martingaledeouf} follows from classical results from Ikeda and Watanabe on stochastic integrals with respect to Poisson point measures (see p.62 in \cite{ikeda2014stochastic}).
\end{proof}
In the following, we define $\mathcal{C}^{1,1}_{\omega}(\mathbb{R}^{+} \times \mathbb{R}^{*}_{+})$ the set of functions $\varphi \in \mathcal{C}^{1,1}(\mathbb{R}^{+} \times \mathbb{R}^{*}_{+})$ such that $\varphi : (t,x) \mapsto \omega(x)$ or
\begin{align*}
\exists C>0, \forall x>0, \quad \sup_{t \in \mathbb{R}^{+}} \bigg( \left|\varphi(t,x)\right| + |\partial_{1} \varphi(t,x)|\dfrac{\omega(x)}{1+x+\omega(x)} + |\partial_{2} \varphi(t,x)|\omega(x)  \bigg) \leq C \omega(x).
\end{align*}
Note that the function $\varphi : (t,x) \mapsto \omega(x)$ does not necessarily verify the previous condition (in particular, $\omega'$ is not necessarily bounded). However, we include this specific function in $\mathcal{C}^{1,1}_{\omega}(\mathbb{R}^{+} \times \mathbb{R}^{*}_{+})$, because we need to be able to apply the following results to this function for the proof of Theorem~\ref{theo:convergence}.
\begin{corr}
Under the \hyperlink{gen}{general setting}, let $\varphi \in \mathcal{C}^{1,1}_{\omega}(\mathbb{R}^{+} \times \mathbb{R}^{*}_{+})$. Then the process $(\langle \mu_{t}, \varphi_{t} \rangle)_{t \geq 0}$ is a semi-martingale, with for all $t \geq 0$, a finite variation part given by
\begin{align*}
V_{\varphi,t} := & \hspace{0.1 cm} \langle \mu_{0}, \varphi_{0} \rangle + \displaystyle{\int_{0}^{t}} \langle \mu_{s}, \Phi_{s}(R_{s},.) \rangle \mathrm{d}s - \displaystyle{\int_{0}^{t}\int_{\mathbb{R}^{*}_{+} }} d(x)\varphi_{s}(x) \mu_{s}(\mathrm{d}x) \mathrm{d}s \\
    &  + \displaystyle{\int_{0}^{t}}\int_{\mathbb{R}^{*}_{+} } b(x)\bigg( \varphi_{s}(x_{0})+\varphi_{s}(x-x_{0})-\varphi_{s}(x)\bigg) \mu_{s}(\mathrm{d}x) \mathrm{d}s,
\end{align*}
with $\Phi$ associated to $\varphi$ as in \eqref{eq:Phidef}, and a square-integrable martingale part $\heartsuit_{\varphi,t} := \langle \mu_{t}, \varphi_{t} \rangle - V_{\varphi,t} $ whose predictable quadratic variation is given by
\begin{align*}
\langle \heartsuit_{\varphi} \rangle_{t} =  \displaystyle{\int_{0}^{t}\int_{\mathbb{R}^{*}_{+}}}b(x)\bigg[ \varphi_{s}(x_{0})+\varphi_{s}(x-x_{0})-\varphi_{s}(x) \bigg]^{2} \mu_{s}(\mathrm{d}x) \mathrm{d}s + \displaystyle{\int_{0}^{t}\int_{\mathbb{R}^{*}_{+}}}d(x)\varphi_{s}^{2}(x) \mu_{s}(\mathrm{d}x) \mathrm{d}s.
\end{align*}
\label{corr:martingaleutile}
\end{corr}
\begin{proof}
We apply Proposition~\ref{prop:martingaledeouf} to $F : (R,x) \in [0,R_{\max}] \times \mathbb{R} \mapsto x $ and $\varphi$. We let the reader verify that we can do so, thanks to the assumption $\varphi \in \mathcal{C}^{1,1}_{\omega}(\mathbb{R}^{*}_{+})$ (in particular, if $\varphi \neq \omega$, we use the fact that there exists a constant $C>0$ such that $|\partial_{2}\varphi(t,x)| < C$ for every $t \geq 0$, $x >0$, so $|\varphi_{s}(x-x_{0})-\varphi_{s}(x)| \leq C x_{0}$ for every $s \in [0,t]$ and $x>0$), Lemma~\ref{lemme:equivalence} and finally Proposition~\ref{prop:controlpop}.
\end{proof}

\textbf{Remark:} At this step, one can show that $(\mu_{t},R_{t})_{t \geq 0}$ is a Jumping Markov Process (JMP). This particular type of Feller process was initially introduced by Davis \cite{davis1984piecewise} for $\mathbb{R}^{n}$-valued processes, and called Piecewise Deterministic Markov Processes (PDMP). Then, Jacod and Skorokhod introduced in \cite{jacod1996jumping} the general definition of a JMP, adapted to our measure-valued setting. We can further characterize the Feller process $(\mu_{t},R_{t})_{t \geq 0}$ \textit{via} its extended generator (see p.45 in \cite{jacod1996jumping}), and refer the reader to Proposition II.1.20. in \cite{broduthesis} for details.

\section{Renormalization of the process}
\label{subsec:construcdeux}

In this section, we work under Assumptions~\ref{hyp:probamortel} and \ref{hyp:poidsomega}. In Section~\ref{subsec:algorenom}, we define a sequence $\left(\left(\mu^{K}_{t},R^{K}_{t}\right)_{t \geq 0}\right)_{K \in \mathbb{N}^{*}}$, where every process $\left(\mu^{K}_{t},R^{K}_{t}\right)_{t \geq 0}$ is a renormalization of the initial process $(\mu_{t},R_{t})_{t \geq 0}$ defined in Section~\ref{sec:construction}, and $K$ is a scaling parameter representing the population size at time 0, meant to diverge towards $+ \infty$. In Section~\ref{subsec:martrenom}, we use the results of Section~\ref{sec:defass} to obtain martingale and control properties for the renormalized process $(\mu^{K}_{t},R^{K}_{t})_{t \geq 0}$.

\subsection{Definition of the renormalized process}
\label{subsec:algorenom}

We follow a classical procedure, first described in \cite{fournier2004microscopic}, and then reproduced in many articles \cite{champagnat2005individualbased, chi_08, FRITSCH20151,tchouanti2024well}. First, for every $K \in \mathbb{N}^{*}$, we will define an auxiliary process $(\nu^{K}_{t},R^{K}_{t})_{t \geq 0}$, following the exact same construction as in Section~\ref{sec:construction}, but with an inverse conversion efficiency $\chi_{K} := \chi/K$. Thus, all the results of Section~\ref{sec:defass} will apply to $(\nu^{K}_{t},R^{K}_{t})_{t \geq 0}$, simply replacing $\chi$ with $\chi_{K}$. We begin with the definition of the renormalized deterministic flow followed by individual energies between jumps. We consider an initial condition $(\nu^{K}_{0},R_{0}) \in \mathcal{M}_{P}(\mathbb{R}^{*}_{+})\times [0,R_{\max}]$ at time 0, which means that there exists $N \in \mathbb{N}$ and $(\xi^{u,K}_{0})_{1 \leq u \leq N}$ in $(\mathbb{R}^{*}_{+})^{N}$ such that $$\nu^{K}_{0} := \sum\limits_{u=1}^{N} \delta_{\xi^{u,K}_{0}},$$  
where $\nu^{K}_{0}=0$ if $N=0$. We write $\left((X^{u,K}_{t}(\nu^{K}_{0},R_{0}))_{1 \leq u \leq N},X^{\Re,K}_{t}(\nu^{K}_{0},R_{0})\right)$ for a solution to the system of $N+1$ coupled equations 
\begin{align}
    \dfrac{\mathrm{d}X^{\Re,K}_{t}(\nu^{K}_{0},R_{0})}{\mathrm{d}t} & = \varsigma(X^{\Re,K}_{t}(\nu^{K}_{0},R_{0})) - \dfrac{\chi}{K} \displaystyle{\int_{\mathbb{R}^{*}_{+}} f(x,X^{\Re,K}_{t}(\nu^{K}_{0},R_{0})) \tilde{\nu}^{K}_{t}(\mathrm{d}x)},\label{eq:regroupindivdeux} \\
    \dfrac{\mathrm{d}X^{u,K}_{t}(\nu^{K}_{0},R_{0})}{\mathrm{d}t} & = g(X^{u,K}_{t}(\nu^{K}_{0},R_{0}),X^{\Re,K}_{t}(\nu^{K}_{0},R_{0})) \quad \mathrm{for} \hspace{0.1 cm} 1 \leq u \leq N, \label{regeoupeedeux}
\end{align}
where $\tilde{\nu}^{K}_{t}:= \sum_{i=1}^{N} \delta_{X^{u,K}_{t}(\nu^{K}_{0},R_{0})} $, and with initial condition at time 0
\begin{center}
    $\begin{array}{lll}
    X^{\Re,K}_{0}(\nu^{K}_{0},R_{0}) & = & R_{0}, \\
    X^{u,K}_{0}(\nu^{K}_{0},R_{0}) & = & \xi^{u,K}_{0} \quad \mathrm{for} \hspace{0.1 cm} 1 \leq u \leq N. \\
\end{array}$ 
\end{center}
Note that the system of equations \eqref{eq:regroupindivdeux}-\eqref{regeoupeedeux} is similar to \eqref{eq:regroupindiv}-\eqref{regeoupee}, where we only replace $\chi$ by $\chi/K$. With the same arguments as in Proposition~\ref{prop:cauchy}, we can define for $t$ in a neighborhood of 0, denoted as $[0,t^{K}_{\exp}(\nu_{0}^{K},R_{0}))$ the renormalized flow 
\begin{center}
    $X^{K}_{t}(\nu^{K}_{0},R_{0}) := \left( \sum\limits_{u =1}^{N}{\delta_{X^{u,K}_{t}(\nu^{K}_{0},R_{0})}}, X^{\Re,K}_{t}(\nu^{K}_{0},R_{0})\right),$
\end{center}
and it benefits from the same regularity properties as $X$ of Section~\ref{sec:construction}. We also adopt the \hyperlink{conv}{convention} depicted in the remark after Corollary~\ref{prop:modelrelevant} to make sense of the previous notation for $t \geq t^{K}_{\exp}(\nu_{0}^{K},R_{0})$. 
\begin{defi}[\textbf{Renormalized process}]
Let $R_{0} \in [0,R_{\max}]$ be a random variable. Let $(\nu^{K}_{0})_{K \geq 1}$ be a sequence of random variables in $\mathcal{M}_{P}(\mathbb{R}^{*}_{+})$, such that
    \begin{align}
    \underset{K \geq 1}{\mathrm{sup}} \bigg( \dfrac{1}{K} \hspace{0.1 cm} \mathbb{E}\left(\langle \nu^{K}_{0},1+\mathrm{Id}+\omega \rangle\right) \bigg) < + \infty.
    \label{eq:assumptioninticonditionK}
    \end{align}
We use the Poisson point measures $\mathcal{N}$ and $\mathcal{N}'$ of Section~\ref{sec:construction}, independent from $R_{0}$ and $(\nu^{K}_{0})_{K \geq 1}$. For every $K \geq 1$, the renormalized process $(\mu^{K}_{t},R^{K}_{t})_{t \geq 0}$ with initial condition $(\mu^{K}_{0},R_{0})$, is given for every $t \geq 0$ by $\mu^{K}_{t} := \dfrac{\nu^{K}_{t}}{K}$ and

\begin{align*}
    (\nu^{K}_{t},R^{K}_{t}) = &  \hspace{0.1 cm}X^{K}_{t}(\nu^{K}_{0},R_{0})  \\  & +  \displaystyle{\int_{0}^{t}\int_{\mathcal{U} \times \mathbb{R}^{*}_{+} } \mathbb{1}_{\{u \in V^{K}_{s-} \}} \mathbb{1}_{\{h \leq b\left(\xi^{u,K}_{s-}\right)\}}}  [ X^{K}_{t-s}(\nu^{K}_{s-} + \delta_{x_{0}} + \delta_{\xi^{u,K}_{s-}-x_{0}} -\delta_{\xi^{u,K}_{s-}},R^{K}_{s})  \\
    &  \hspace{3 cm} - X^{K}_{t-s}(\nu^{K}_{s-},R^{K}_{s}) ] \mathcal{N}(\mathrm{d}s,\mathrm{d}u,\mathrm{d}h)  \\
     & + 
    \displaystyle{\int_{0}^{t}\int_{\mathcal{U} \times \mathbb{R}^{*}_{+} } \mathbb{1}_{\{ u \in V^{K}_{s-}\}} \mathbb{1}_{\{h \leq d\left(\xi^{u,K}_{s-}\right)\}} [X^{K}_{t-s}(\nu^{K}_{s-} - \delta_{\xi^{u,K}_{s-}},R^{K}_{s})}  \\ 
    & \hspace{3 cm} - X^{K}_{t-s}(\nu^{K}_{s-},R^{K}_{s}) ] \mathcal{N}'(\mathrm{d}s,\mathrm{d}u,\mathrm{d}h), 
\end{align*}
where for all $t \geq 0$, $V^{K}_{t}$ is the set containing alive individuals at time $t$, and $\xi^{u,K}_{t}$ are individual energies, defined and actualized over time with the same conventions as in Section~\ref{sec:construction}, simply replacing the flow $X$ by $X^{K}$.
\end{defi}
Under Assumptions~\ref{hyp:probamortel} and \ref{hyp:poidsomega}, from Corollary~\ref{corr:jinfini}, for every $K \in \mathbb{N}^{*}$, the renormalized process $(\mu^{K}_{t},R^{K}_{t})_{t \geq 0}$ is almost surely well-defined, \textit{i.e.} individual energies do not vanish/explode and there is no accumulation of jumps in finite time. Our motivation is to keep the same amount of resources and temporal dynamics, but to consider population sizes going to $+ \infty$. Intuitively, the way we proceed is to make every individual in the population smaller, and the interaction between individuals via resource consumption proportional to their typical size. The parameter $1/K$ represents the amount of resource consumed by a single individual, and modelling a population of $K$ such individuals leads us back to the temporal dynamics of Section~\ref{sec:construction}.

\begin{defi}[\hypertarget{ren}{\textbf{Renormalized setting}}]
In the following, we denote as `the renormalized setting', the framework refering to the previous renormalization. Hence, we work with Assumptions~\ref{hyp:probamortel}, \ref{hyp:poidsomega}, and the condition \eqref{eq:assumptioninticonditionK}. 
\end{defi}
\textbf{Remark:} Note that for any $K \geq 1$, Assumptions~\ref{hyp:probamortel} and \ref{hyp:poidsomega} do not depend on $K$. In particular, we work with a fixed weight function $\omega$ that does not depend on $K$, and verifies Assumption~\ref{hyp:poidsomega}. Moreover, the \hyperlink{ren}{renormalized setting} implies the \hyperlink{gen}{general setting} (the construction of Section~\ref{sec:construction} accounts for the case $K=1$). 

\subsection{Properties of the renormalized process}
\label{subsec:martrenom}

Most of the results of Section~\ref{sec:defass} can be adapted to the study of the renormalized processes. For example, we recover martingale properties. In the following, we naturally write $N^{K}_{t} := \langle \mu^{K}_{t}, 1 \rangle$, $E^{K}_{t} := \langle \mu^{K}_{t}, \mathrm{Id} \rangle$ and $\Omega^{K}_{t} := \langle \mu^{K}_{t}, \omega \rangle$ for $t \geq 0$ and $K \in \mathbb{N}^{*}$.

\begin{prop}
Under the \hyperlink{ren}{renormalized setting}, let $\varphi \in \mathcal{C}^{1,1}_{\omega}(\mathbb{R}^{+} \times \mathbb{R}^{*}_{+})$ and $K \in \mathbb{N}^{*}$. Then the process $(\langle \mu^{K}_{t}, \varphi_{t} \rangle)_{t \geq 0}$ is a semi-martingale, with for all $t \geq 0$, a finite variation part given by
\begin{align*}
V^{K}_{\varphi,t} := & \hspace{0.1 cm} \langle \mu^{K}_{0}, \varphi_{0} \rangle + \displaystyle{\int_{0}^{t}} \langle \mu^{K}_{s}, \Phi_{s}(R^{K}_{s},.) \rangle \mathrm{d}s - \displaystyle{\int_{0}^{t}\int_{\mathbb{R}^{*}_{+} }} d(x)\varphi_{s}(x) \mu^{K}_{s}(\mathrm{d}x) \mathrm{d}s \\
    &  + \displaystyle{\int_{0}^{t}}\int_{\mathbb{R}^{*}_{+} } b(x)\bigg( \varphi_{s}(x_{0})+\varphi_{s}(x-x_{0})-\varphi_{s}(x)\bigg) \mu^{K}_{s}(\mathrm{d}x) \mathrm{d}s,
\end{align*}
with $\Phi$ associated to $\varphi$ as in \eqref{eq:Phidef}, and a square-integrable martingale part $\heartsuit^{K}_{\varphi,t} := \langle \mu^{K}_{t}, \varphi_{t} \rangle -V^{K}_{\varphi,t} $ whose predictable quadratic variation is given by
\begin{align*}
\langle \heartsuit^{K}_{\varphi} \rangle_{t} & = \dfrac{1}{K} \Bigg( \displaystyle{\int_{0}^{t}\int_{\mathbb{R}^{*}_{+}}}b(x)\bigg[ \varphi_{s}(x_{0})+\varphi_{s}(x-x_{0})-\varphi_{s}(x) \bigg]^{2} \mu^{K}_{s}(\mathrm{d}x) \mathrm{d}s \\
& \hspace{8cm} + \displaystyle{\int_{0}^{t}\int_{\mathbb{R}^{*}_{+}}}d(x)\varphi_{s}^{2}(x) \mu^{K}_{s}(\mathrm{d}x) \mathrm{d}s\Bigg).
\end{align*}
\label{prop:martingalerenormalisee}
\end{prop}

We do not develop here the proof of Proposition~\ref{prop:martingalerenormalisee}, since it is similar to the proofs in Section~\ref{sec:defass}, and refer the reader to Proposition II.2.5. in \cite{broduthesis}. We go further with properties that holds true uniformly in $K$, which will be useful in Section~\ref{sec:proof} for the proof of Theorem~\ref{theo:convergence}.

\begin{prop}
Let $p \geq 1$, and under the \hyperlink{ren}{renormalized setting}, assume in addition that
\begin{align}
    \underset{K \in \mathbb{N}^{*}}{\mathrm{sup}} \hspace{0.1 cm} \mathbb{E}\left(\left(E_{0}^{K} +N_{0}^{K} + \Omega^{K}_{0}\right)^{p}\right) < + \infty.
    \label{eq:lesmom}
    \end{align}
Then, for all $T \geq 0$, we have
\[ \sup_{K \in \mathbb{N}^{*}} \mathbb{E}\left( \sup_{t \in [0,T]}\left(E_{t}^{K} +N_{t}^{K} + \Omega^{K}_{t}\right)^{p} \right) < + \infty. \]
This immediately implies that for every $\varphi \in \mathfrak{B}_{1+\mathrm{Id}+\omega}(\mathbb{R}^{*}_{+})$, we have
\[ \sup_{K \in \mathbb{N}^{*}} \mathbb{E}\left(\underset{t \in [0,T]}{\mathrm{sup}} \left|\langle \mu^{K}_{t}, \varphi \rangle\right|^{p} \right) < + \infty. \] 
\label{lemme:controlenp}
\end{prop}

The proof of Proposition~\ref{lemme:controlenp} can be found in Appendix~\ref{app:lemmcontr}.

\begin{corr}
Under the \hyperlink{ren}{renormalized setting}, let $\varphi \in \mathcal{C}^{1,1}_{\omega}(\mathbb{R}^{*}_{+})$ and $K \in \mathbb{N}^{*}$. Assume that there exists $p>1$ such that \eqref{eq:lesmom} holds true. Then, for all $t \geq 0$, the family of square-integrable martingales $\left(\heartsuit^{K}_{\varphi,t}\right)_{K \in \mathbb{N}^{*}}$ defined in Proposition~\ref{prop:martingalerenormalisee} is uniformly integrable.
\label{corr:uniformly}
\end{corr}
\begin{proof}
This follows from Lemma~\ref{lemme:equivalence}, Proposition~\ref{prop:martingalerenormalisee}, Proposition~\ref{lemme:controlenp} applied to $\varphi$ and $p$, and Proposition 2.2 p.494 in \cite{ek_2005}.
\end{proof}

\section{Main results and conjectures}
\label{sec:theorem}

In Section~\ref{subsec:main}, we give our main tightness result in Theorem~\ref{theo:convergence}, and the sketch of its proof. Then in Section~\ref{subsec:discussiontheo}, we present two lines of research to extend Theorem~\ref{theo:convergence}.

\subsection{Main theorem and sketch of the proof}
\label{subsec:main}

We begin with preliminary definitions and consider a function $w : \mathbb{R}^{*}_{+} \rightarrow \mathbb{R}^{*}_{+}$. We write $\mathcal{M}_{w}(\mathbb{R}^{*}_{+})$ for the set of positive measures $\mu$ on $\mathbb{R}^{*}_{+}$ such that $\langle \mu, w \rangle < + \infty$.  In particular, $\mathcal{M}_{P}(\mathbb{R}^{*}_{+}) \subseteq \mathcal{M}_{w}(\mathbb{R}^{*}_{+})$. We define $\mathcal{C}_{c}(\mathbb{R}^{*}_{+})$ the space of continous functions with compact support, and $\mathcal{C}_{w}(\mathbb{R}^{*}_{+})$ the space of continuous functions $f$ such that $f \in \mathfrak{B}_{w}(\mathbb{R}^{*}_{+})$. The vague, respectively $w$-weak, topology on $\mathcal{M}_{w}(\mathbb{R}^{*}_{+})$ is the finest topology for which the applications $\mu \mapsto \langle \mu, f \rangle$ are continous, with $f$ in $\mathcal{C}_{c}(\mathbb{R}^{*}_{+})$, respectively in $\mathcal{C}_{w}(\mathbb{R}^{*}_{+})$. We write $(\mathcal{M}_{w}(\mathbb{R}^{*}_{+}),v)$, respectively $(\mathcal{M}_{w}(\mathbb{R}^{*}_{+}),\mathrm{w})$, when we endow $\mathcal{M}_{w}(\mathbb{R}^{*}_{+})$ with the vague topology, respectively the $w$-weak topology.
\\\\
The latter notation is not standard in the literature and can be seen as a weighted version of the usual weak topology, which corresponds to the case $w \equiv 1$. We introduce it because in Theorem~\ref{theo:convergence}, our processes will take values in such weigthed spaces of measures. The $w$-weak topology is always finer than the vague topology, but depending on the weight function $w$, it is not necessarily comparable to the usual weak topology. With classical techniques, we show that both spaces $(\mathcal{M}_{w}(\mathbb{R}^{*}_{+}),v)$ and $(\mathcal{M}_{w}(\mathbb{R}^{*}_{+}),\mathrm{w})$ are Polish spaces (see Appendix B.2.1 in \cite{broduthesis}). We naturally endow $[0,R_{\max}]$ with the usual topology, and $(M_{w}(\mathbb{R}^{*}_{+}),v) \times [0,R_{\max}]$ or $(M_{w}(\mathbb{R}^{*}_{+}),\mathrm{w}) \times [0,R_{\max}]$ with the product topology, and these are again Polish spaces. 
\\\\
For $i=v$ or $\mathrm{w}$, we write $\mathbb{D}([0,T], (\mathcal{M}_{w}(\mathbb{R}^{*}_{+}),i) \times [0,R_{\max}])$ for the space of càdlàg functions from $[0,T]$ to $(\mathcal{M}_{w}(\mathbb{R}^{*}_{+}),i) \times [0,R_{\max}]$, and $\mathcal{C}([0,T], (\mathcal{M}_{w}(\mathbb{R}^{*}_{+}),i) \times [0,R_{\max}])$ for continuous ones. These spaces are endowed with the usual Skorokhod topology, hence are Polish spaces (Theorem 5.6 p.121 in \cite{ek_2005}). Finally, for every $T \geq 0$, we \hypertarget{defomega}{define} $\mathcal{C}^{1,1}_{\omega,T}(\mathbb{R}^{+} \times \mathbb{R}^{*}_{+})$ the set of functions $\varphi \in \mathcal{C}^{1,1}(\mathbb{R}^{+} \times \mathbb{R}^{*}_{+})$ such that
\begin{align*}
\exists C>0, \forall x>0, \quad \sup_{t \in [0,T]}\bigg( \left|\varphi(t,x)\right|(1+d(x)) + |\partial_{1} \varphi(t,x)| + |\partial_{2} \varphi(t,x)|\omega(x)  \bigg) \leq C \omega(x).
\end{align*}
For technical reasons, just before stating our main theorem, we formulate the following additional assumption. We will use it in particular in Section~\ref{subsec:identification}.
\begin{hyp}
There exists $\eta \in (0,1)$ such that the functions $\omega$ (weight function), $b$ (birth rate) and $\overline{g}$ (maximal speed of energy gain) verify
\begin{itemize}
\item[-] $\exists \tilde{c}_{b}>0, \forall x>1, \quad b(x) \leq \tilde{c}_{b}(\varpi(x)+ \omega(x))$,
\item[-] $\exists \tilde{c}_{g}>0, \forall x>0, \quad \overline{g}(x) \leq \tilde{c}_{g}(\varpi(x)+ \omega(x))$,
\end{itemize} 
where $\varpi : x>0 \mapsto (x-x_{0})^{1-\eta}\mathbb{1}_{x-x_{0}>0}$. In addition, we ask for $1/\omega$ to be bounded in a neighborhood of $+ \infty$.
\label{ass:finallejd}
\end{hyp}
\textbf{Remark:} Assumption~\ref{ass:finallejd} may look redundant with Assumption~\ref{hyp:poidsomega}. We believe, although it is still a conjecture, that it is possible to obtain Theorem~\ref{theo:convergence} without Assumption~\ref{ass:finallejd}. This is discussed in Section~\ref{subsec:determinsiticattempt}. Also, note that $\mathbf{(H2)}$ in Theorem~\ref{theo:initial} implies Assumption~\ref{ass:finallejd}.

\begin{theorem}
We work under the \hyperlink{ren}{renormalized setting} and Assumption~\ref{ass:finallejd}. Let $\left(\left(\mu^{K}_{t},R^{K}_{t}\right)_{t \geq 0}\right)_{K \in \mathbb{N}^{*}}$ the sequence of renormalized processes defined in Section \ref{subsec:construcdeux} be such that $\mathbf{(H3)}$ and $\mathbf{(H4)}$ hold true. Then, for all $T \geq 0$, $\bigg(\left(\mu^{K}_{t},R^{K}_{t}\right)_{t \in [0,T]}\bigg)_{K \in \mathbb{N}^{*}}$ is tight in $\mathbb{D}([0,T], (\mathcal{M}_{\omega}(\mathbb{R}^{*}_{+}),\mathrm{w}) \times [0,R_{\max}])$. Any of its accumulation point $(\mu^{*}_{t},R^{*}_{t})_{t \in [0,T]}$ is in $\mathcal{C}([0,T], (\mathcal{M}_{\omega}(\mathbb{R}^{*}_{+}),\mathrm{w}) \times [0,R_{\max}])$; and for all $t \in [0,T]$, for every $\varphi \in \mathcal{C}^{1,1}_{\omega,T}(\mathbb{R}^{+} \times \mathbb{R}^{*}_{+})$, it verifies almost surely \eqref{eq:resslim} and \eqref{eq:indivlim}. 
\label{theo:convergence}
\end{theorem}
On the one hand, the tightness result of Theorem~\ref{theo:convergence} is an important contribution to the study of individual-based models as in Section~\ref{sec:construction} with unbounded growth, birth and/or death rates. In the literature, when the previously mentioned rates are bounded, one of the main technical point of the proof is to provide uniform (on $K$ and $t \in [0,T]$) bounds and martingale properties for quantities of the form $\langle \mu^{K}_{t}, 1 \rangle$. With Proposition~\ref{prop:martingalerenormalisee} in mind, one can relate $\langle \mu^{K}_{t}, 1 \rangle$ and quantities of the form $\langle \mu^{K}_{t}, b \rangle$ and $\langle \mu^{K}_{t}, d \rangle$. When rates are bounded, $\langle \mu^{K}_{t}, b \rangle$ and $\langle \mu^{K}_{t}, d \rangle$ are themselves controlled by $C\langle \mu^{K}_{t}, 1 \rangle$ with a constant $C>0$. One classically deduces a functional equation verified by $\langle \mu^{K}_{t}, 1 \rangle$ and concludes with Gronwall lemma. It is then possible to make sense of limiting quantities of the form $\langle \mu^{*}_{t}, b \rangle$ and $\langle \mu^{*}_{t}, d \rangle$. With unbounded rates, the previous technique does not work anymore, and it is even possible that for $t \in [0,T]$, quantities of the form $\langle \mu^{*}_{t}, b \rangle$ and $\langle \mu^{*}_{t}, d \rangle$ are infinite. Hence, instead of controlling quantities of the form $\langle \mu^{K}_{t}, 1 \rangle$, we search for a function $\omega$ such that  $\langle \mu^{K}_{t}, \omega \rangle$ is related to $\langle \mu^{K}_{t}, b \omega \rangle$ and $\langle \mu^{K}_{t}, d \omega \rangle$, and the latter quantities are themselves controlled by $C\langle \mu^{K}_{t}, \omega \rangle$ with a constant $C>0$. Also, we want to be able to define $\langle \mu^{*}_{t}, b \omega \rangle$ and $\langle \mu^{*}_{t}, d \omega \rangle$ as finite quantities. The constraints that have to be verified by $\omega$ are expressed in Assumptions~\ref{hyp:poidsomega} and \ref{ass:finallejd}. We thus work in a weighted space with respect to the function $\omega$ and apply the same procedure as in the classical case. We recover a classical result of tightness for measure-valued processes, initiated in \cite{fournier2004microscopic}, but without \textit{a priori} bounds on the growth, birth and/or death rates.
\\\\
On the other hand, there is a price to pay to obtain this general tightness result. It holds true only in the weighted space $\mathbb{D}([0,T], (\mathcal{M}_{\omega}(\mathbb{R}^{*}_{+}),\mathrm{w}) \times [0,R_{\max}])$, with a weight function $\omega$ verifying Assumptions~\ref{hyp:poidsomega} and \ref{ass:finallejd}. In particular, if $d(x) \xrightarrow[x \rightarrow  0]{} + \infty$, then $\omega(x) \xrightarrow[x \rightarrow  0]{} 0$. For biological reasons, this is the typical case we want to investigate if we think of an unbounded death rate (the Metabolic Theory of Ecology assumes a death rate of the form $x>0 \mapsto x^{-\delta}$ with $\delta >0$ \cite{malerba_2019}). The tightness result of Theorem~\ref{theo:convergence} is weaker `near 0' than a tightness in $\mathbb{D}([0,T], (\mathcal{M}_{1}(\mathbb{R}^{*}_{+}),\mathrm{w}) \times [0,R_{\max}])$ (\textit{i.e.} with the usual weak topology on $\mathcal{M}_{1}(\mathbb{R}^{*}_{+})$), in the sense that Equation~\eqref{eq:indivlim} is not valid for $\varphi \equiv 1$, because this function does not converge to 0 at 0. Still, it is possible that $\omega(x) \xrightarrow[x \rightarrow +\infty]{} +\infty$ (see Section~\ref{subsec:examples} for a specific example), so that the tightness result of Theorem~\ref{theo:convergence} is stronger `near $+\infty$' than a tightness with the usual weak topology on $\mathcal{M}_{1}(\mathbb{R}^{*}_{+})$, in the sense that Equation~\eqref{eq:indivlim} is valid for functions $\varphi$ going to $+\infty$ near $+ \infty$. 
\\\\
Finally, remark that for a given initial condition $\mu^{*}_{0}$, Equations~\eqref{eq:resslim} and \eqref{eq:indivlim} are the weak formulation of a PDE system, and Theorem~\ref{theo:convergence} provides the existence of measure solutions to this problem. The reader can already consider the system \eqref{eq:indivedp}, \eqref{eq:ressedp}, \eqref{eq:boundaryedp} to have a clearer idea of the kind of deterministic PDE we obtain when there exists functions solution.
\\\\
\hypertarget{sketch}{\textbf{Sketch of the proof of Theorem~\ref{theo:convergence}}:} 
\\\\
In the following, we work under the assumptions of Theorem~\ref{theo:convergence} and fix $T \geq 0$. For any $K \in \mathbb{N}^{*}$, we write $\mathscr{L}^{K}$ for the law of the process $\bigg(\left(\mu^{K}_{t},R^{K}_{t}\right)_{t \in [0,T]}\bigg)_{K \in \mathbb{N}^{*}}$. Every $\mathscr{L}^{K}$ is a probability measure on $\mathbb{D}([0,T],\mathcal{M}_{\omega}(\mathbb{R}^{*}_{+}) \times [0,R_{\max}])$. Note that $\mathscr{L}^{K}$ does not depend on the choice of the topology on $\mathcal{M}_{\omega}(\mathbb{R}^{*}_{+})$, if we choose among the vague topology or the $\omega$-weak topology (see Lemma B.2.10. in \cite{broduthesis}). Our aim in the following proof is first to prove the tightness of $(\mathscr{L}^{K})_{K \in \mathbb{N}^{*}}$ in $\mathbb{D}([0,T], (\mathcal{M}_{\omega}(\mathbb{R}^{*}_{+}),\mathrm{w}) \times [0,R_{\max}])$, and then to characterize any accumulation point with \eqref{eq:resslim}-\eqref{eq:indivlim}.
We divide the proof in four steps.
\begin{itemize}
\item First, in Section~\ref{subsec:tension}, we show that $(\mathscr{L}^{K})_{K \in \mathbb{N}^{*}}$  is tight in $\mathbb{D}([0,T], (\mathcal{M}_{\omega}(\mathbb{R}^{*}_{+}),v) \times [0,R_{\max}])$. Remark that $\mathcal{M}_{\omega}(\mathbb{R}^{*}_{+})$ is endowed with the vague topology at this step. We extend a criterion of Roelly \cite{roel_86} to our weighted space of measures (see Theorem~\ref{theo:roel}), which reduces the problem to proving the tightness of a sequence in $\mathbb{D}([0,T], \mathbb{R})$. To do so, we use a criterion of Aldous and Rebolledo \cite{ald_86} and Proposition~\ref{prop:martingalerenormalisee}.
\item In Section~\ref{subsec:continuitylimit}, we prove that any limit of a subsequence of $\bigg(\left(\mu^{K}_{t},R^{K}_{t}\right)_{t \in [0,T]}\bigg)_{K \in \mathbb{N}^{*}}$ converging in law in $\mathbb{D}([0,T], (\mathcal{M}_{\omega}(\mathbb{R}^{*}_{+}),v) \times [0,R_{\max}])$ is in $\mathcal{C}([0,T], (\mathcal{M}_{\omega}(\mathbb{R}^{*}_{+}),\mathrm{w}) \times [0,R_{\max}])$ (note that the limit is continuous for $\mathcal{M}_{\omega}(\mathbb{R}^{*}_{+})$ endowed with the $\omega$-weak topology). We adapt the reasoning of Step 2 in Section 5 of \cite{jourdain_11} to our weighted setting.
\item Thanks to the continuity of any accumulation point, in Section~\ref{subsec:topotricks}, we extend a result of Méléard and Roelly \cite{meleard1993convergences} to our weighted space of measures (see Theorem~\ref{theo:melroel}), and prove that $(\mathscr{L}^{K})_{K \in \mathbb{N}^{*}}$ is tight in $\mathbb{D}([0,T], (\mathcal{M}_{\omega}(\mathbb{R}^{*}_{+}),\mathrm{w}) \times [0,R_{\max}])$. In particular, we use the previous step to control the finite variation and martingale parts of $\langle
\mu^{K}_{t}, \omega \rangle$ for $K \geq 1$ and $t \in [0,T]$.
\item Finally in Section~\ref{subsec:identification}, we characterize the limit $(\mu^{*},R^{*})$ of any converging subsequence of our process in $\mathbb{D}([0,T], (\mathcal{M}_{\omega}(\mathbb{R}^{*}_{+}),\mathrm{w}) \times [0,R_{\max}])$, still written $(\mu^{K},R^{K})_{K \geq 1}$, with Equations~\eqref{eq:resslim} and \eqref{eq:indivlim}. It is precisely at this step that we use the additional Assumption~\ref{ass:finallejd}, to be able to control, uniformly on $K \geq 1$ and $t \in [0,T]$, quantities of the form $\mathbb{E}\left(\langle \mu^{K}_{t} - \mu^{*}_{t}, b + \overline{g} \rangle \right)$.
\end{itemize}

\subsection{Possible extensions of Theorem~\ref{theo:convergence}}
\label{subsec:discussiontheo}

First in Section~\ref{section:uniqueness}, we present the difficulties encountered for showing that there exists a unique measure solution to the system \eqref{eq:resslim}-\eqref{eq:indivlim}. If this uniqueness holds true, the tightness result of Theorem~\ref{theo:convergence} is in fact a convergence in law towards the unique limit identified by \eqref{eq:resslim}-\eqref{eq:indivlim}. Then in Section~\ref{subsec:determinsiticattempt}, we conjecture an extension of Theorem~\ref{theo:convergence} to a tightness result in a broader set of measure-valued processes, with additional regularity and control assumptions on the solutions to \eqref{eq:resslim}-\eqref{eq:indivlim}.

\subsubsection{Uniqueness of a solution to \eqref{eq:resslim}-\eqref{eq:indivlim}}
\label{section:uniqueness}

Classical results depicted in the literature (Theorem 5.3. in \cite{fournier2004microscopic}, Corollary 3.3. in \cite{chi_08}, Theorem 5.2 in \cite{FRITSCH20151}) establish convergence in law towards a deterministic limit (conditionally to the initial condition $\mu^{*}_{0}$) and not only tightness of sequences of renormalizations as the one described in Section~\ref{subsec:algorenom}. They use a compactness-uniqueness argument summarized as follows. The law of an accumulation point of $\left((\mu^{K}_{t},R^{K}_{t})_{t \in [0,T]}\right)_{K \in \mathbb{N}^{*}}$ in $ \mathbb{D}([0,T], (\mathcal{M}_{\omega}(\mathbb{R}^{*}_{+}) ,\mathrm{w})\times [0,R_{\max}])$ is always a mixture between the law of $\mu^{*}_{0}$ and the law of solutions to \eqref{eq:resslim}-\eqref{eq:indivlim}. Now if for any fixed $\mu^{*}_{0}$, a solution to \eqref{eq:resslim}-\eqref{eq:indivlim} is unique, then the law of such an accumulation point is unique conditionally to $\mu^{*}_{0}$, and our tightness result becomes a convergence result towards this unique limit. 
\\\\
We refer the reader to Proposition II.5.7. in \cite{broduthesis} for a proof of this uniqueness result in the case of bounded rates. We encountered two main difficulties in the general case with unbounded rates.
\begin{itemize}
\item[-] First, the deterministic flow describing the evolution of individual energies and the resource (see again Section~\ref{subssub:wouah}) may only be locally well-defined. This is because with a possibly unbounded growth rate $g$, this flow can explode or reach 0 in finite time. Hence, we cannot control pathwisely our random trajectories for any time $t \geq 0$, by a straightforward comparison with this deterministic flow. We only have a result in expectation in Proposition~\ref{lemme:controlenp}.
\item[-] Then, the usual technique to obtain uniqueness is to pick two solutions $\mu_{t}$ and $\mu'_{t}$ to \eqref{eq:resslim}-\eqref{eq:indivlim} with the same initial condition, and to show that $| \langle \mu'_{t} - \mu_{t}, \varphi \rangle |= 0$ for any $t \geq 0$ and $\varphi$ in a broad enough set of test functions (see again the proof of Proposition II.5.7. in \cite{broduthesis}). A technical step is to provide an upper bound for $| \langle \mu'_{t} - \mu_{t}, \varphi \rangle |$ with integral terms, where integrands are functions that verify the same bounds as $\varphi$ up to a multiplicative constant, in order to use Gronwall lemma. The main difficulty here is precisely to check that the integrands verify the same constraints as $\varphi$, because they depend themselves on $\varphi$ and the previously mentioned deterministic flow. It is still an open question to know if one can find appropriate conditions on $\varphi$ and/or its derivatives, that we are able to recover for the previously mentioned integrands.
\end{itemize}

For the previously mentioned reasons, our feeling is that the classical proof of uniqueness in the case of bounded rates (the original argument comes from Step 3. in the proof of Theorem 5.3. in \cite{fournier2004microscopic}) can hardly be extended to more general cases. Hopefully, there are other ways to proceed, and we still conjecture that this uniqueness result should hold true in our setting. For example, we could certainly adapt the work of \cite{lauwal07} to obtain a uniqueness result in the case where birth and death rates are power functions with non-negative exponents, and the growth rate is a power function with an exponent $0 \leq \tau \leq 1$ (see Theorem 1.2. in \cite{lauwal07}). Finally, for biological reasons depicted in \cite{brodu2026}, we are particularly interested in a death rate which is a power function with a negative exponent. Obtaining a uniqueness result in that setting is left for future work.

\subsubsection{Extension of Theorem~\ref{theo:convergence} with additional assumptions}
\label{subsec:determinsiticattempt}

We aim for a stronger conclusion where we do not use Assumption~\ref{ass:finallejd}, and the tightness holds true in the broader Skorokhod space $\mathbb{D}([0,T], (\mathcal{M}_{1+ \mathrm{Id} + \omega}(\mathbb{R}^{*}_{+}),\mathrm{w}) \times [0,R_{\max}])$. This replaces $\omega$ with $1+ \mathrm{Id} + \omega$, thus is an amelioration of Theorem~\ref{theo:convergence}, only if $\omega$ is dominated by $1+ \mathrm{Id}$ in a neighborhood of 0 or $+ \infty$. This will be the case in the allometric example presented in Section~\ref{subsec:examples} (see also Figure~\ref{fig:exampeomega}).

\begin{conj}
We work under the assumptions of Theorem~\ref{theo:convergence} without Assumption~\ref{ass:finallejd}, but assume in addition that the sequence $(\mu_{0}^{K})_{K \in \mathbb{N}^{*}}$ converges in law towards $\mu_{0}^{*}$ in $(\mathcal{M}_{1+ \mathrm{Id}+\omega}(\mathbb{R}^{*}_{+}),\mathrm{w})$. Then, for all $T \geq 0$, the sequence $\bigg(\left(\mu^{K}_{t},R^{K}_{t}\right)_{t \in [0,T]}\bigg)_{K \in \mathbb{N}^{*}}$ is tight in $\mathbb{D}([0,T], (\mathcal{M}_{1+ \mathrm{Id} + \omega}(\mathbb{R}^{*}_{+}),\mathrm{w}) \times [0,R_{\max}])$.
\label{theo:final}
\end{conj}

\textbf{Remark:} We can show that if the death rate $d$ is bounded, then Conjecture~\ref{theo:final} holds true (see Lemma II.5.11. in \cite{broduthesis}). In particular, for any accumulation point $\mu^{*}$, Equation~\eqref{eq:indivlim} is valid for $\varphi : (t,x) \mapsto x$, \textit{i.e.} we have an explicit expression of $\langle \mu^{*}_{t}, \mathrm{Id} \rangle$ for $t \in [0,T]$ (in the case of a mass-structured model, this represents the total biomass for the limiting system described by $\mu^{*}$ at time $t$). If we adapt our setting to recover existing individual-based models with bounded rates, such as \cite{chi_08} or \cite{FRITSCH20151}, we thus extend their results. Indeed, previous papers were only able to compute numerically $\langle \mu^{*}_{t}, \mathrm{Id} \rangle$ for any $t \geq 0$, as the limit of quantities of the form $\langle \mu^{*}_{t}, \varphi \rangle$ with compactly supported functions $\varphi$.
\\\\
Now if the death rate $d$ is unbounded, we propose an approach to prove Conjecture~\ref{theo:final}. If we fix $T \geq 0$ and an accumulation point $(\mu^{*}_{t},R^{*}_{t})_{t \in [0,T]}$ of the sequence of renormalized processes of Section~\ref{subsec:algorenom}, we write for $t \in [0,T]$,
$$\beta_{t} := \displaystyle{\int_{0}^{t}} \langle \mu^{*}_{s}, d \times (1 + \mathrm{Id}) \rangle \mathrm{d}s.$$ 
Note that this quantity is possibly finite or infinite. 

\begin{prop}
We work under the assumptions of Theorem~\ref{theo:convergence} without Assumption~\ref{ass:finallejd}, but assume in addition that
\begin{itemize}
\item[$(i)$] the sequence $(\mu_{0}^{K})_{K \in \mathbb{N}^{*}}$ converges in law towards $\mu_{0}^{*}$ in $(\mathcal{M}_{1+ \mathrm{Id}+\omega}(\mathbb{R}^{*}_{+}),\mathrm{w})$,
\item[$(ii)$] for any accumulation point $\mu^{*}$, we have: $\forall t \in [0,T], \quad \beta_{t} < + \infty$,
\item[$(iii)$] any accumulation point $(\mu^{*}_{t})_{t \in [0,T]}$ is in $\mathcal{C}([0,T], (\mathcal{M}_{1+ \mathrm{Id} + \omega}(\mathbb{R}^{*}_{+}),\mathrm{w}))$.
\end{itemize} 
Then, the conclusion of Conjecture~\ref{theo:final} holds true.
\label{prop:ladernieredelathese}
\end{prop}
The proof of Proposition~\ref{prop:ladernieredelathese} can be found in Section II.5.3 of \cite{broduthesis} (see Theorem II.5.3.). Note that conditionally to the initial condition $\mu^{*}_{0}$, the additional assumptions we introduce in Proposition~\ref{prop:ladernieredelathese} depend only on the limit $\mu^{*}$, which is characterized by \eqref{eq:resslim}-\eqref{eq:indivlim}. Hence, we transposed our probabilistic questioning into the study of the deterministic solutions to the weak formulation of a PDE system. To show Conjecture~\ref{theo:final}, it suffices to show that solutions to \eqref{eq:resslim}-\eqref{eq:indivlim} verify the conclusions of points $(ii)$ and $(iii)$. This could be settled with a deterministic approach, which is not our area of expertise.

\section{Application to allometric functional parameters}
\label{subsec:examples}

In this section, to illustrate the general results of Section~\ref{sec:theorem}, we introduce a specific setting with allometric functional parameters (\textit{i.e.} the jump rates are power functions). For more details about the biological motivation behind this example, see \cite{brodu2026}. We first show in Section~\ref{ex3} that this allometric case falls within the framework of the \hyperlink{gen}{general setting} depicted in Section~\ref{sec:defass} (in particular, we prove that there exists a weight function $\omega$ verifying Assumptions~\ref{hyp:poidsomega} and \ref{ass:finallejd} when the functional parameters of our model are allometric). Thus, the tightness result of Theorem~\ref{theo:convergence} is valid in this allometric case, and we then provide numerical illustrations in Section~\ref{subsec:comparisonibmpde}.

\subsection{Allometric setting}
\label{ex3}

To the best of our knowledge, the following setting has not been studied with an individual-based approach so far, except for the model with constant resources in \cite{brodu2026}. In the following, we will refer to it as the `allometric setting'. For every $x>0$ and $R \geq 0$, we set:
\begin{enumerate}
\item $\ell(x) :=C_{\alpha}x^{\alpha}$,
\item $b(x):= \mathbb{1}_{x > x_{0}}C_{\beta}x^{\beta}$,
\item $f(x,R) := \phi(R)C_{\gamma}x^{\gamma}$ (\textit{i.e.} $\psi(x) = C_{\gamma}x^{\gamma}$),
\item $d(x):= C_{\delta}x^{\delta}$,
\end{enumerate}
with $\alpha, \beta, \gamma, \delta \in \mathbb{R}$ and $(C_{\alpha}, C_{\beta}, C_{\gamma}, C_{\delta}) \in \mathbb{R}^{*}_{+}$.
If $f$ and $g$ are two functions on $\mathbb{R}^{*}_{+}$ with $g$ positive, we write $f \circeq_{0} g$, respectively $f \circeq_{\infty} g$, if $f(x)= g(x)$ on a neighborhood of 0, respectively a neighborhood of $+ \infty$.
\begin{defi}[\textbf{Allometric form}]
We say that $\omega \in \mathcal{C}^{1}(\mathbb{R}^{*}_{+})$ has an \hypertarget{alloform}{allometric form} if there exists $0 \leq \kappa_{1} \leq \kappa_{2}$ with
\begin{itemize}
\item[-] $ \exists C_{1}>0, \quad \omega(x) \circeq_{0} C_{1} x^{\kappa_{1}}$,
\item[-]$ \exists C_{2}>0, \quad \omega(x) \circeq_{\infty} C_{2}x^{\kappa_{2}}$.
\end{itemize}
Remark that in that case, $\omega$ is non-decreasing.
\end{defi}
We verify easily that we can construct such functions for any $0 \leq \kappa_{1} \leq \kappa_{2}$. In the upcoming lemmas, we investigate under which conditions on the allometric coefficients we verify Assumptions~\ref{hyp:gainenergy}, \ref{hyp:probamortel} and \ref{ass:finallejd} with a weight function $\omega$ with an \hyperlink{alloform}{allometric form}.
\begin{lemme}
Under the allometric setting, we have
\begin{center}
$($Assumptions~\ref{hyp:gainenergy} and \ref{hyp:probamortel} $)\Leftrightarrow$ $(\delta \leq \alpha -1$, $\gamma = \alpha$ and $C_{\gamma} > C_{\alpha})$.
\end{center} 
\label{lemme:allomcond}
\end{lemme}
\begin{proof}
Under the allometric setting, we have
$$  \forall x>0 , \quad \psi(x)-\ell(x) = C_{\gamma}x^{\gamma}-C_{\alpha}x^{\alpha},  $$
so Lemma~\ref{lemme:gainenergy} implies that $\gamma = \alpha $ and $C_{\gamma}> C_{\alpha}$, and we verify that conversely, this implies Assumption~\ref{hyp:gainenergy}. Assumption~\ref{hyp:probamortel} is also immediately equivalent to $\delta \leq \alpha-1$.
\end{proof}
\begin{lemme}
Under the allometric setting, under Assumptions~\ref{hyp:gainenergy} and \ref{hyp:probamortel}, suppose that the weight function $\omega$ has an \hyperlink{alloform}{allometric form} with $0 \leq \kappa_{1} \leq \kappa_{2}$.
\begin{itemize}
\item[-] If $ \delta < -1$, then Assumption~\ref{hyp:poidsomega} is equivalent to 
$$\kappa_{1}=\kappa_{2}= -\delta, \quad \alpha \in [0,1]\quad \mathrm{and} \quad \beta \leq 2+ \delta;$$
\item[-] If $-1 \leq \delta \leq 0 $, then Assumption~\ref{hyp:poidsomega} is equivalent to 
$$ - \delta \leq \kappa_{1} \leq \kappa_{2} \leq \dfrac{1-\delta}{2} \leq 1, \quad \alpha \in [0,1], \quad \mathrm{and} \quad \beta \leq 1.$$
\item[-] If $\delta > 0 $, then Assumption~\ref{hyp:poidsomega} cannot be verified.
\end{itemize}
These constraints on the allometic coefficients are illustrated on Figure~\ref{fig:contraitnesallomdeux} and Figure~\ref{fig:contraitnesallomtrois}.
\label{lemme:allomqutre}  
\end{lemme}
\begin{proof}
By Lemma~\ref{lemme:allomcond}, we work with $\delta \leq \alpha -1$, $\gamma = \alpha$ and $C_{\gamma} > C_{\alpha}$, and an \hyperlink{alloform}{allometric form} for $\omega$ with $0 \leq \kappa_{1} \leq \kappa_{2}$. First, we suppose that Assumption~\ref{hyp:poidsomega} holds true. If we consider the different points of Assumption~\ref{hyp:poidsomega}, it suffices to study the associated inequalities on a neighborhood of 0 and on a neighborhood of $+ \infty$, since all the considered functions are continuous. Under the allometric setting, the third point of Assumption~\ref{hyp:poidsomega} gives
\begin{align}
\forall x>0, \quad C_{\delta}x^{\delta}(\omega(x)+ \omega^{2}(x)) \leq C_{d}\left( 1 + x + \omega(x) \right).
\label{eq:allomthird}
\end{align} 
Considering $x \rightarrow 0$, respectively $x \rightarrow + \infty$, and an \hyperlink{alloform}{allometric form} for $\omega$, we obtain that \eqref{eq:allomthird} implies that $- \delta \leq \kappa_{1}$, respectively that $2\kappa_{2}+ \delta \leq \max(1,\kappa_{2})$. First if $\kappa_{2}>1$, then we necessarily have $1 < -\delta = \kappa_{1} =\kappa_{2}$. Else, we verify that we have $-\delta \leq \kappa_{1} \leq \kappa_{2} \leq \dfrac{1-\delta}{2} \leq 1$. Then, the first point of Assumption~\ref{hyp:poidsomega} gives 
\begin{align}
\forall x>0, \quad  \max(\phi(R_{\max})C_{\gamma}-C_{\alpha},C_{\alpha})x^{\alpha}(1+\omega'(x)) \leq C_{g}\left( 1 + x + \omega(x) \right).
\label{eq:allompsi}
\end{align} 
Considering $x \rightarrow 0$ and $x \rightarrow + \infty$, we verify that \eqref{eq:allompsi} implies that $0 \leq \alpha \leq \max(1,\kappa_{2})$. In particular, we obtain that Assumption~\ref{hyp:poidsomega} cannot be verified if $\delta >0$ (because in that case, we would have $0 > -\delta \geq 1- \alpha \geq 0$). Also, if we suppose by contradiction that $\alpha >1$ (so $\kappa_{2}>1$), then considering \eqref{eq:allompsi} when $x \rightarrow + \infty$, we would have $\alpha+ \kappa_{2}-1 \leq \max(1,\kappa_{2}) = \kappa_{2}$ so $\alpha \leq 1$ which is a contradiction. Hence, we always have $\alpha \in [0,1]$. Finally, we consider the second point of Assumption~\ref{hyp:poidsomega}. First if $-1 \leq \delta \leq 0$, then $\omega$ is Lipschitz continuous on $(1,+\infty)$ by the previous work (because $\kappa_{2} \leq 1$), so the second point of Assumption~\ref{hyp:poidsomega} is equivalent to \eqref{eq:blipschitz}, which entails $\beta \leq 1$. Else if $\delta < -1$, we verify that the left-hand side in the second point of Assumption~\ref{hyp:poidsomega} is of order $x^{\beta-2\delta-2}$ in a neighborhood of $+ \infty$, and the right-hand side is of order $x^{-\delta}$, hence we necessarily have $\beta \leq 2+\delta$. The converse implications in Lemma~\ref{lemme:allomcond} are straightforward verifications.
\end{proof}
\textbf{Remark:} Note that if we want to pick a weight function that has an \hyperlink{alloform}{allometric form} with $0 \leq \kappa_{1} \leq \kappa_{2}$, we restrict ourselves to the case $\alpha \in [0,1]$. If we want to consider for example the case $\beta+1=\delta+1=\alpha= \gamma=3/4$ supported by the Metabolic Theory of Ecology \cite{savdee_08}, we can pick $\kappa_{1}= 1/4$ and $\kappa_{2}= 5/8$ according to Lemma~\ref{lemme:allomqutre} (and this is the best choice in the sense that we cannot choose another weight function that has an \hyperlink{alloform}{allometric form} and dominates this particular weight near 0 or $+ \infty$). Obviously, we investigated here only a precise form of the weight function $\omega$, other choices may be possible to be less restrictive in the case $\beta+1=\delta+1=\alpha=3/4$, or to study for example the case $\alpha <0$. We leave this for future work. 

\begin{figure}[h!]
\centering
\begin{tikzpicture}
\draw[pattern=north east lines, pattern color=green] (0,1) -- (-1,1) -- (-3,-1)--(-4,-2)--(0,-2);
\draw[color=green,thick,-] (0,-2) -- (0,1);
\draw[color=green,thick,-] (0,1) -- (-1,1);
\draw[color=green,thick,-] (-1,1) -- (-4,-2);
\draw[->] (-4,0) -- (1,0);
\draw[->] (0,-2) -- (0,2);
\draw (-0.1,0.1) -- (0.1,-0.1) node[below right]{0};
\node at (1,-0.25) {$\delta$};
\node at (-0.25,2) {$\beta$};

\draw (-1,0.1) -- (-1,-0.1) node[below]{$-1$};
\draw (-0.1,1) -- (0.1,1) node[right]{1};
\draw (-0.1,-1) -- (0.1,-1) node[right]{$-1$};

\draw[color=white,thick,-] (-4,-2) -- (0,-2);

\draw[fill=green] (-1,-1) circle (0.1cm);
\end{tikzpicture}
\caption{Visual representation of the constraints on $\delta$ and $\beta$ in Lemma~\ref{lemme:allomqutre}, with $\gamma=\alpha \in [0,1]$ (we took $\alpha=0$ on this figure). The admissible coefficients $(\delta, \beta)$ are those in the green hatched area, and verify one of the two following conditions: $(i)$ ($ \delta < -1$ and $\beta \leq 2+\delta$); $(ii)$ ($-1 \leq \delta \leq 0$ and $\beta \leq 1$). Remark that the green area always contains the particular case $\beta=\delta=\alpha-1=\gamma-1$ highlighted by the Metabolic Theory of Ecology, represented by a green dot on the figure.}
\label{fig:contraitnesallomdeux}
\end{figure}
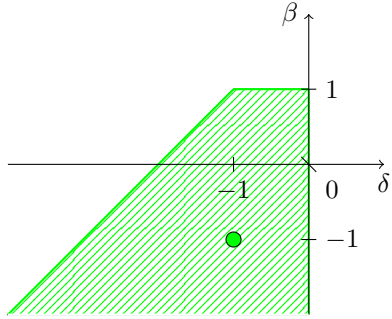

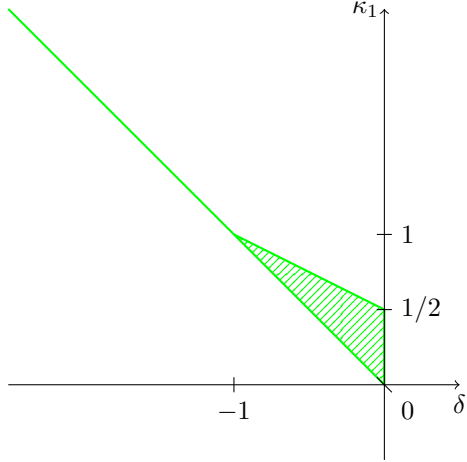
\begin{figure}[h!]
\centering
\begin{tikzpicture}
\draw[pattern=north east lines, pattern color=green] (-2,2) -- (0,0) -- (0,1);
\draw[color=green,thick,-] (-5,5) -- (-2,2);
\draw[color=green,thick,-] (-2,2) -- (0,0);
\draw[color=green,thick,-] (0,0) -- (0,1);
\draw[color=green,thick,-] (-2,2) -- (0,1);
\draw[->] (-5,0) -- (1,0);
\draw[->] (0,-1) -- (0,5);
\draw (-0.1,0.1) -- (0.1,-0.1) node[below right]{0};
\node at (1,-0.25) {$\delta$};
\node at (-0.25,5) {$\kappa_{1}$};

\draw (-2,0.1) -- (-2,-0.1) node[below]{$-1$};
\draw (-0.1,2) -- (0.1,2) node[right]{1};
\draw (-0.1,1) -- (0.1,1) node[right]{$1/2$};


\end{tikzpicture}
\caption{Visual representation of the constraints on $\kappa_{1}$ in Lemma~\ref{lemme:allomqutre}, depending on the value of $\delta$ (we took $\alpha=1$ on this figure). This represents the fact that if $\delta < -1$, then $\kappa_{1}=-\delta$, and if $-1 \leq \delta \leq 0$, then $-\delta \leq \kappa_{1} \leq (1-\delta)/2$. Note that $\kappa_{2}$ verifies the same conditions in Lemma~\ref{lemme:allomqutre}, so that this graph is also valid to visualize the constraints on $\kappa_{2}$. To sum up, if we want to pick an admissible triplet $(\delta,\kappa_{1},\kappa_{2})$, we first choose $(\delta,\kappa_{1})$ on the green line or in the green hatched area. Then, the remaining possible values for $\kappa_{2}$ are such that the two following conditions hold true: $\kappa_{2} \geq \kappa_{1}$, and $(\delta,\kappa_{2})$ is also on the green line or in the green hatched area.}
\label{fig:contraitnesallomtrois}
\end{figure}

\newpage

\begin{lemme}
Under the allometric setting, suppose that the weight function $\omega$ has an \hyperlink{alloform}{allometric form} with $0 \leq \kappa_{1} \leq \kappa_{2}$. Then, Assumption~\ref{ass:finallejd} is equivalent to 
\begin{center}
$ \exists \eta \in (0,1), \quad \kappa_{1} \leq \alpha \leq \max(\kappa_{2},1-\eta)$ and $\beta \leq \max(\kappa_{2},1-\eta)$.
\end{center} 
\label{eq:encoredesrestrictiobs}
\end{lemme}
\begin{proof}
This is an immediate verification (in particular, if $\omega$ has an \hyperlink{alloform}{allometric form}, the fact that $1/\omega$ is bounded near $+ \infty$ is automatically verified).
\end{proof}
Eventually, if we want to apply Theorem~\ref{theo:convergence} under the allometric setting, we thus have to gather all the restrictive assumptions of Lemmas~\ref{lemme:allomcond}, \ref{lemme:allomqutre} and \ref{eq:encoredesrestrictiobs}. For example, if we fix $\beta=\delta=\alpha-1$ (this is motivated by Theorem 2. in Section 3.1. of \cite{brodu2026}), this gives the following constraints on $\kappa_{1}$ and $\kappa_{2}$:
\begin{itemize}
\item[-] $0 \leq 1- \alpha \leq \kappa_{1} \leq \kappa_{2} \leq 1- \dfrac{\alpha}{2} < 1$,
\item[-] $\kappa_{1} \leq \alpha < 1$, which implies with the previous point that $\alpha \in [1/2,1)$.
\end{itemize}

\subsection{Numerical illustration of Theorem~\ref{theo:convergence} under the allometric setting}
\label{subsec:comparisonibmpde}

In this section, we illustrate numerically the tightness result of Theorem~\ref{theo:convergence} under the allometric setting. We consider a deterministic solution $(\mu^{*}_{t},R^{*}_{t})_{t \in [0,T]}$ to the system \eqref{eq:resslim}-\eqref{eq:indivlim} with initial condition $\mu^{*}_{0}$, and make the following assumption.

\begin{hyp}
For every $t \in [0,T]$, $\mu^{*}_{t}$ admits an integrable density $u_{t} \in \mathcal{C}^{1}(\mathbb{R}^{*}_{+} \setminus \{x_{0}\})$ with respect to Lebesgue measure, such that
\begin{itemize}
\item[-] for all $t \in [0,T]$, the function $ x \in \mathbb{R}^{*}_{+} \setminus \{ x_{0} \} \mapsto u_{t}(x) =: u(t,x)$ is $\mathcal{C}^{1}$,
\item[-] for all $t \in [0,T]$, the limits $u_{t}(x_{0}+)$ and $u_{t}(x_{0}-)$ exist and are finite,
\item[-] for all $x \in \mathbb{R}^{*}_{+}$, the function $t \in [0,T] \mapsto u_{t}(x)$ is $\mathcal{C}^{1}$,
\item[-] there exists a locally integrable function $F$ on $\mathbb{R}^{*}_{+}$, such that for all $t \in [0,T]$, for all $x >0$, $|\partial_{1} u(t,x)| \leq F(x).$
\end{itemize}
\label{hyp:reguldens}
\end{hyp}
We conjecture that Assumption~\ref{hyp:reguldens} holds true under our setting, and refer to Proposition III.1.3. in \cite{broduthesis} for a line of research. With the notations of Section~\ref{subsec:dyn} and Assumption~\ref{hyp:reguldens}, we let the reader check that we can rewrite the weak formulation \eqref{eq:resslim}-\eqref{eq:indivlim} into a classical PDE system with function solutions. Namely, $(u_{t},R^{*}_{t})_{t}$ should verify, for every $t \in (0,T]$ and $x \in \mathbb{R}^{*}_{+} \setminus \{ x_{0}\}$,
\begin{align}
\partial_{t}u_{t}(x) + \partial_{x}\bigg(g(x,R^{*}_{t})u_{t}(x)\bigg) & = b(x+x_{0})u_{t}(x+x_{0})-(b(x)+d(x))u_{t}(x),
\label{eq:indivedp}
\end{align}
where $\partial_{x}$, respectively $\partial_{t}$ is the partial derivative with respect to the variable $x$, respectively $t$, and
\begin{align}
\dfrac{\mathrm{d}R^{*}_{t}}{\mathrm{d}t}  = \varsigma(R^{*}_{t}) & - \chi \displaystyle{\int_{\mathbb{R}^{*}_{+}}} f(x,R^{*}_{t}) u_{t}(x) \mathrm{d}x,
\label{eq:ressedp}
\end{align}
so that in particular $t \mapsto R^{*}_{t}$ is $\mathcal{C}^{1}$ on $[0,T]$. We also have the boundary condition
\begin{align}
\displaystyle{\int_{\mathbb{R}^{*}_{+}}} b(y) u_{t}(y) \mathrm{d}y &= \bigg( u_{t}(x_{0}+)-u_{t}(x_{0}-) \bigg)g(x_{0},R^{*}_{t}),
\label{eq:boundaryedp}
\end{align}
and the initial condition $u_{0}$ at time $t=0$. 
\\\\
In the following, we simulate the stochastic process $(\mu^{K}_{t},R^{K}_{t})_{t \geq 0}$ of Section~\ref{subsec:construcdeux} for different values of $K \geq 1$, and denote it as the \textit{individual-based model}, or simply IBM. We will also simulate the PDE system \eqref{eq:indivedp}, \eqref{eq:ressedp}, \eqref{eq:boundaryedp}, and denote this deterministic model as the \textit{PDE model} in the following. Both IBM and PDE model are implemented with \verb+Python+, under the allometric setting of Section~\ref{ex3}, and for the renewal of the resource, we place ourselves in a chemostat setting (see \cite{FRITSCH20151}). We specify our simulation parameters in Appendix~\ref{app:simupara} and describe our algorithms in Appendix~\ref{app:algo}. As in Section 6 of \cite{FRITSCH20151}, we compare simulations of the IBM and the PDE model in three different regimes.
\begin{enumerate}
\item Small population size, with $K=100$;
\item Medium population size, with $K=1000$;
\item Large population size, with $K=10000$.
\end{enumerate}
For each of these regimes, we start from the same initial condition $u_{0}$ depicted on Figure~\ref{fig:intiialdensitth} in Appendix~\ref{app:simupara} and simulate 100 independent runs of the IBM. The convergence of the IBM towards the PDE model is illustrated on Figure~\ref{fig:convergence}, where we present the evolution of the population size, the total energy of the population, and the amount of resources over time for $t \in [0,200]$. The fact that the limit is apparently unique motivates the discussion of Section~\ref{section:uniqueness}. We also represent a phase portrait energy/resource on this time window. Remark that we illustrate the convergence of $\langle \mu^{K}_{t}, 1 \rangle$ (population size) and $\langle \mu^{K}_{t}, \mathrm{Id} \rangle$ (total energy), where $1+\mathrm{Id}$ is not dominated by $\omega$, and this is motivated by Conjecture~\ref{theo:final}. These simulation results are very similar to those obtained in Section 6 in \cite{FRITSCH20151}, but the main difference is the deviations of the IBM from the PDE in terms of total energy, that we observe on the second line of Figure~\ref{fig:convergence}. As time increases, it seems that the variability of the IBM trajectories around the PDE also increases, even if this phenomenon has less impact as $K$ goes to $+ \infty$ by our tightness result. This variability comes precisely from the main new contribution of our work compared to existing literature, which is the fact that the individual growth rate are not bounded, so individual energies can increase very fast.
\\\\
Then on Figure~\ref{fig:densities}, at times $t=0$, 20 and 160, we show a numerical approximation of the renormalized energy distribution $ \tilde{u}_{t} : x>0 \mapsto \dfrac{u_{t}(x)}{\int_{0}^{+ \infty} u_{t}(y) \mathrm{d}y}$, where $(t,x) \mapsto u_{t}(x)$ is solution to the PDE system \eqref{eq:indivedp}, \eqref{eq:ressedp}, \eqref{eq:boundaryedp} with initial condition $u_{0}$. On this curve, we superimpose renormalized histograms of the empirical energy distribution in the population for 100 independent IBM simulations, taken at the same times $t$ to illustrate again the convergence result of Theorem~\ref{theo:convergence}. We observe numerically that for every $t >0$, the density $ x \mapsto u_{t}(x)$ is discontinuous at $x_{0}$, as expressed in the boundary condition \eqref{eq:boundaryedp}. Furthermore, the density is rapidly (we can observe this phenomenon from $t=5$) concentrated on a precise energy window and has a bimodal shape, with a peak near $x_{0}$ and another one near 0, which seems natural with our birth rule in mind. The density decreases very fast to 0 after $x_{0}$ and seems to stabilize after time $t=150$. Finally, our main observation is that the system seems to reach a non-trivial equilibrium, different from $(R_{\mathrm{in}},0)$ (extinction of the population).

\begin{figure}[h!]
\centering
\begin{subfigure}{0.32\textwidth}
    \includegraphics[width=\textwidth]{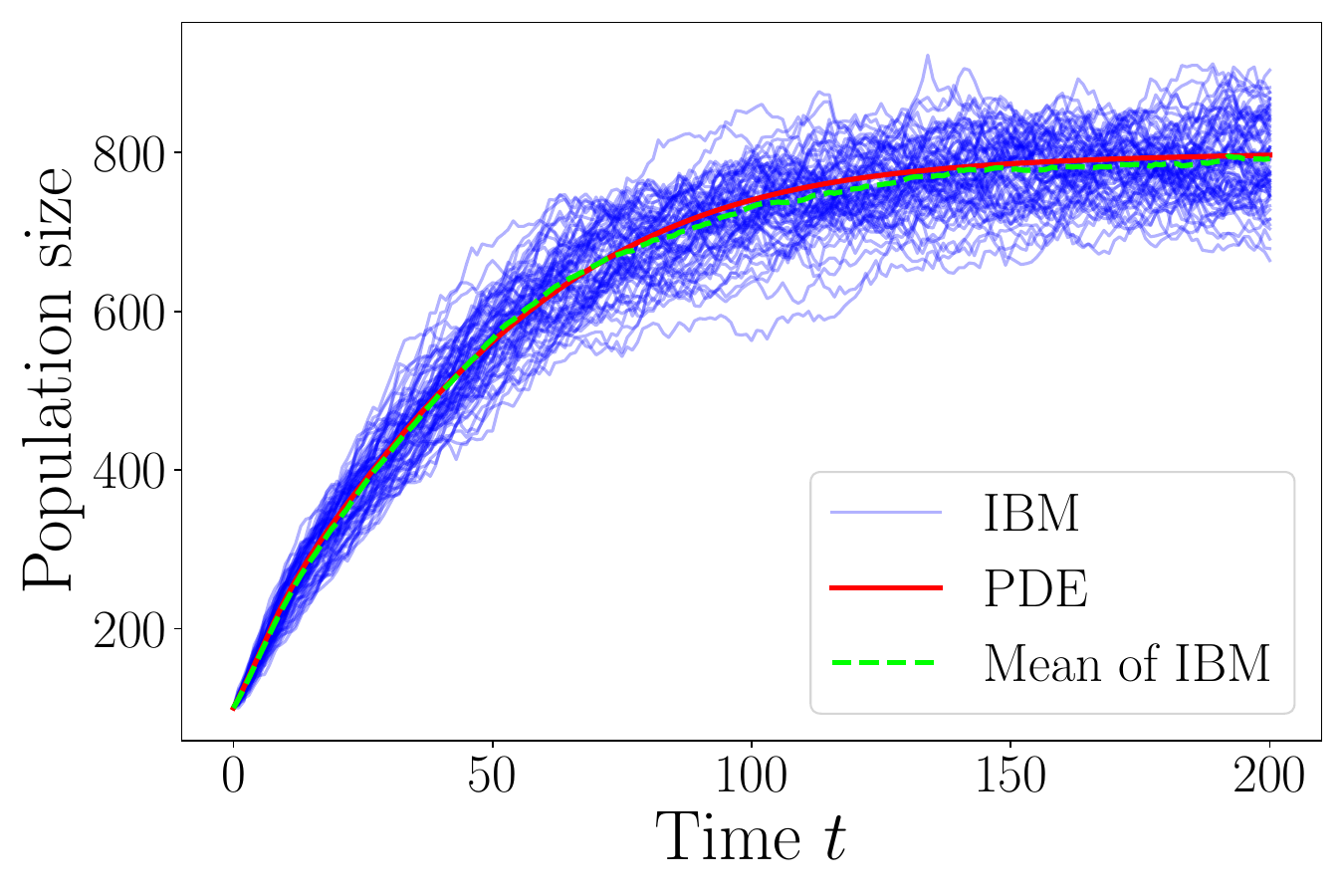} 
    \label{fig:sizecent}
\end{subfigure}
\hspace{0.01 cm}
\begin{subfigure}{0.32\textwidth}
    \includegraphics[width=\textwidth]{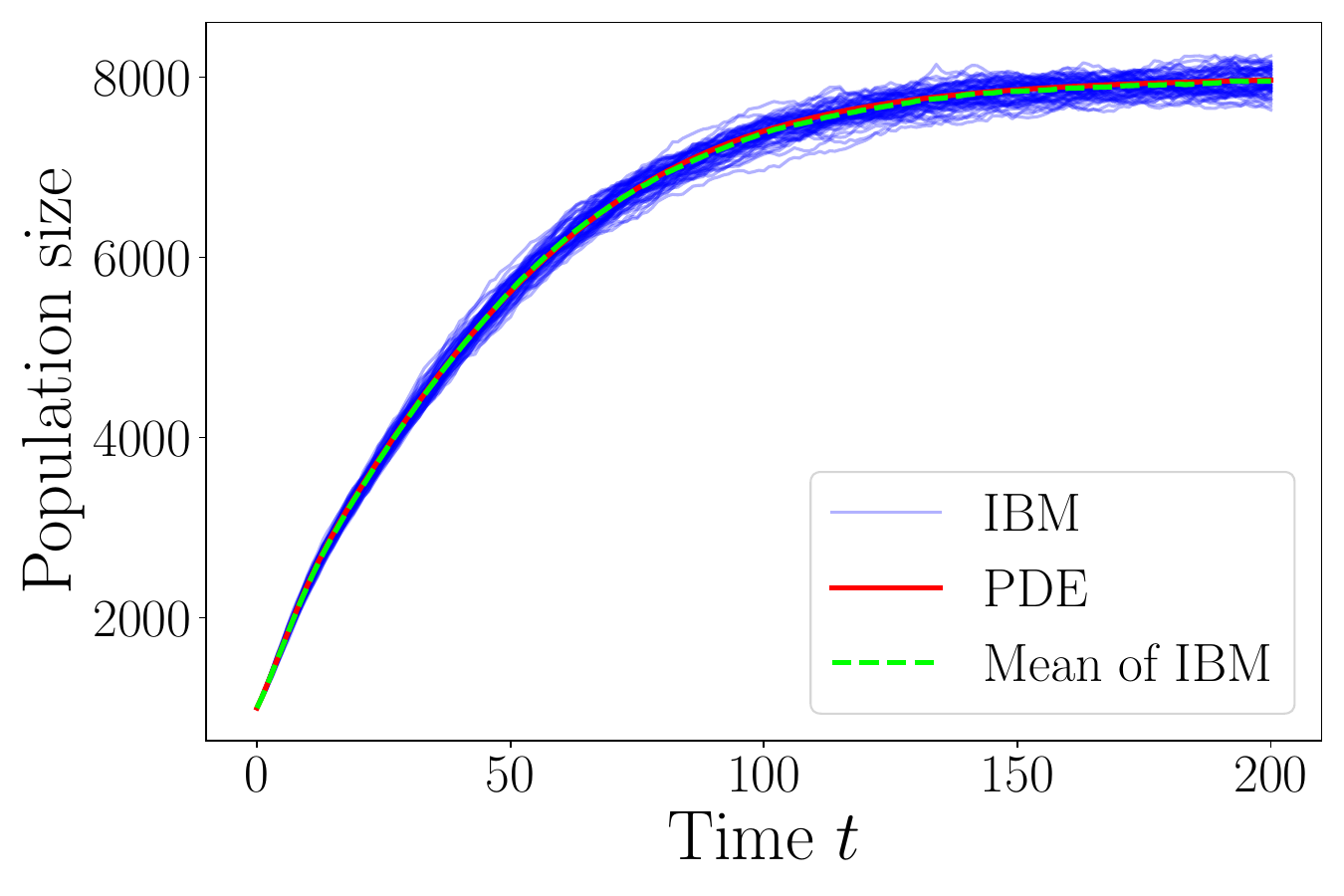} 
    \label{fig:sizemille}
\end{subfigure}
\hspace{0.01 cm}
\begin{subfigure}{0.32\textwidth}
    \includegraphics[width=\textwidth]{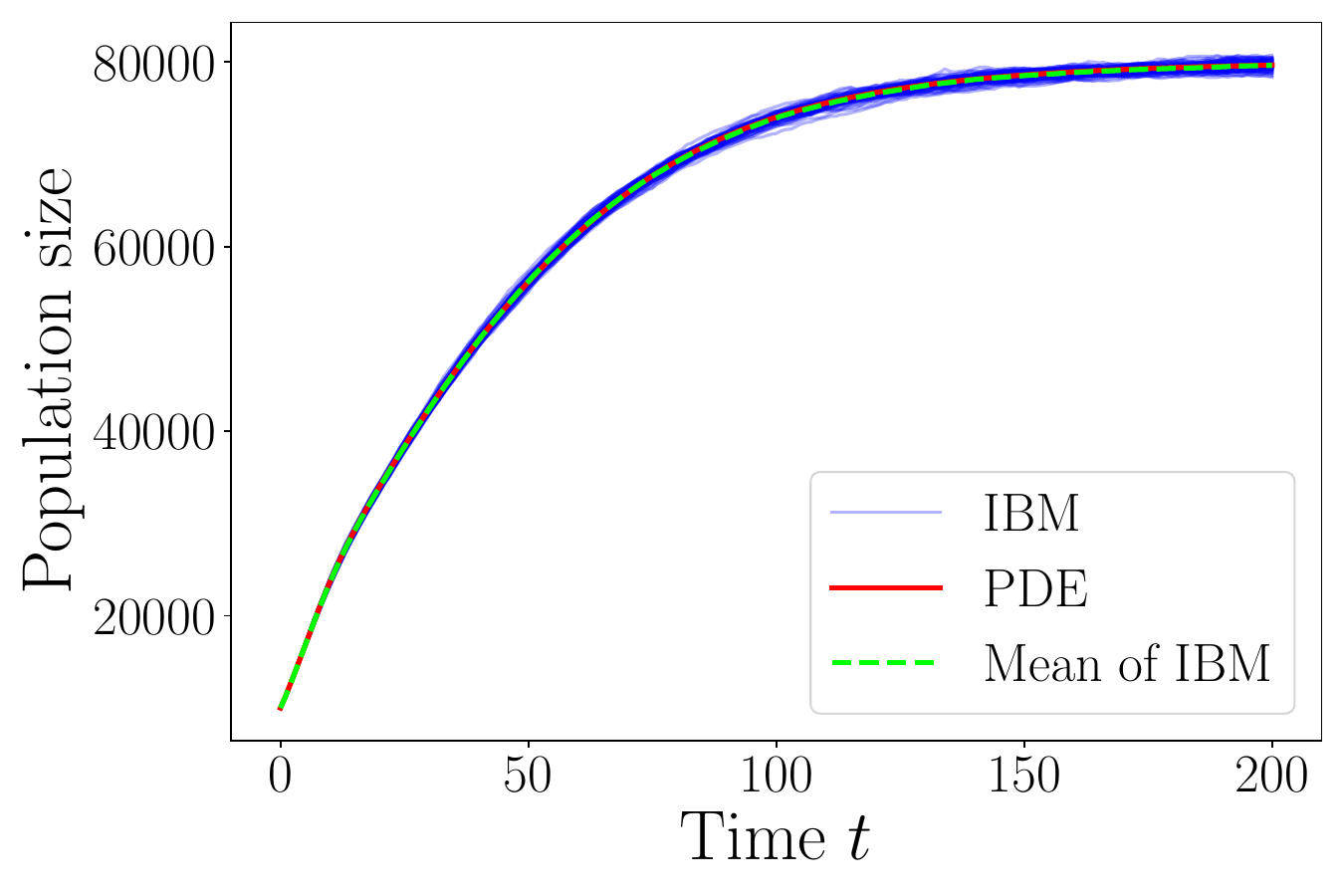} 
    \label{fig:size10000}
\end{subfigure}
\hspace{0.01 cm}
\begin{subfigure}{0.32\textwidth}
   \includegraphics[width=\textwidth]{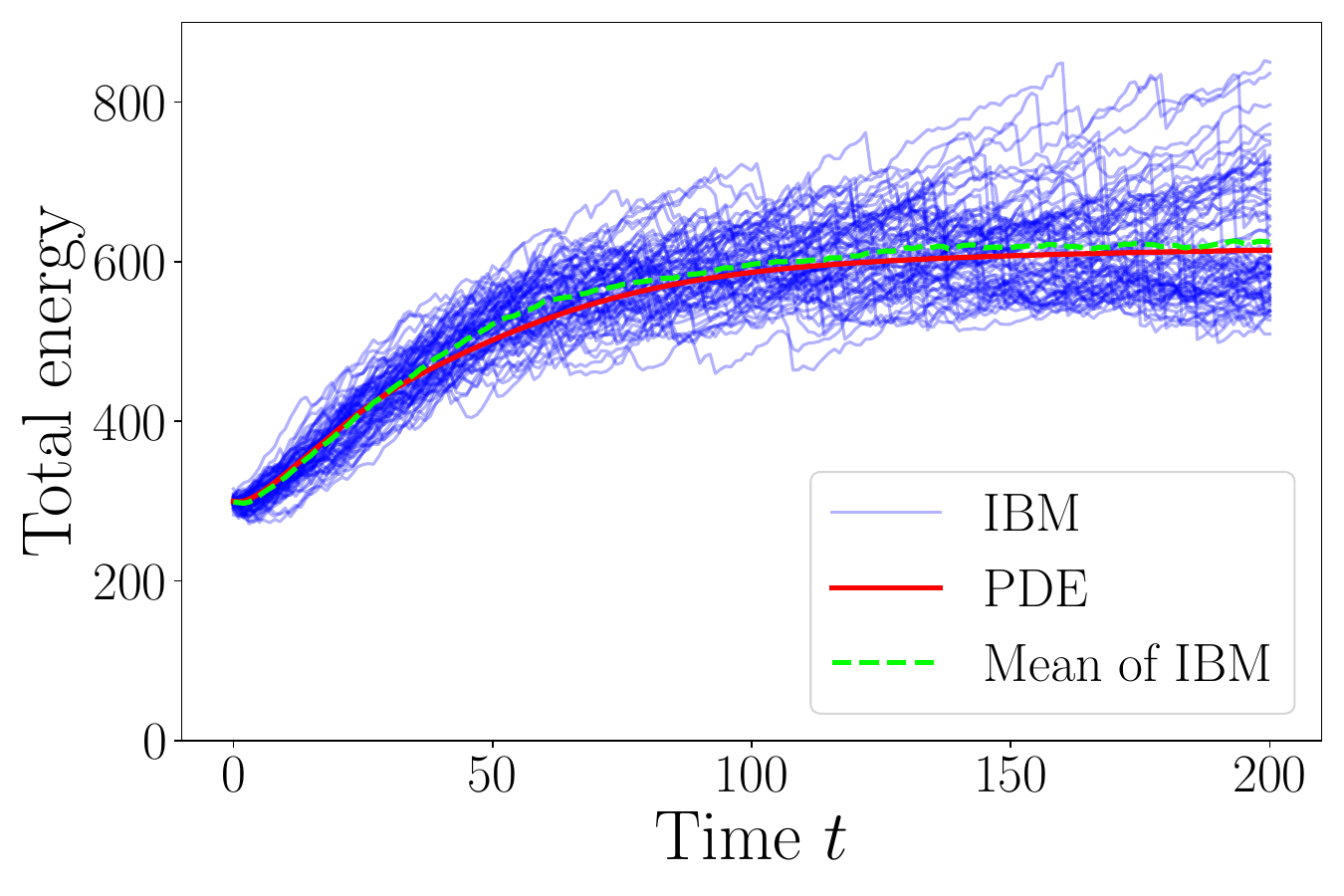}  
    \label{fig:energycent}
\end{subfigure}
\hspace{0.01 cm}
\begin{subfigure}{0.32\textwidth}
   \includegraphics[width=\textwidth]{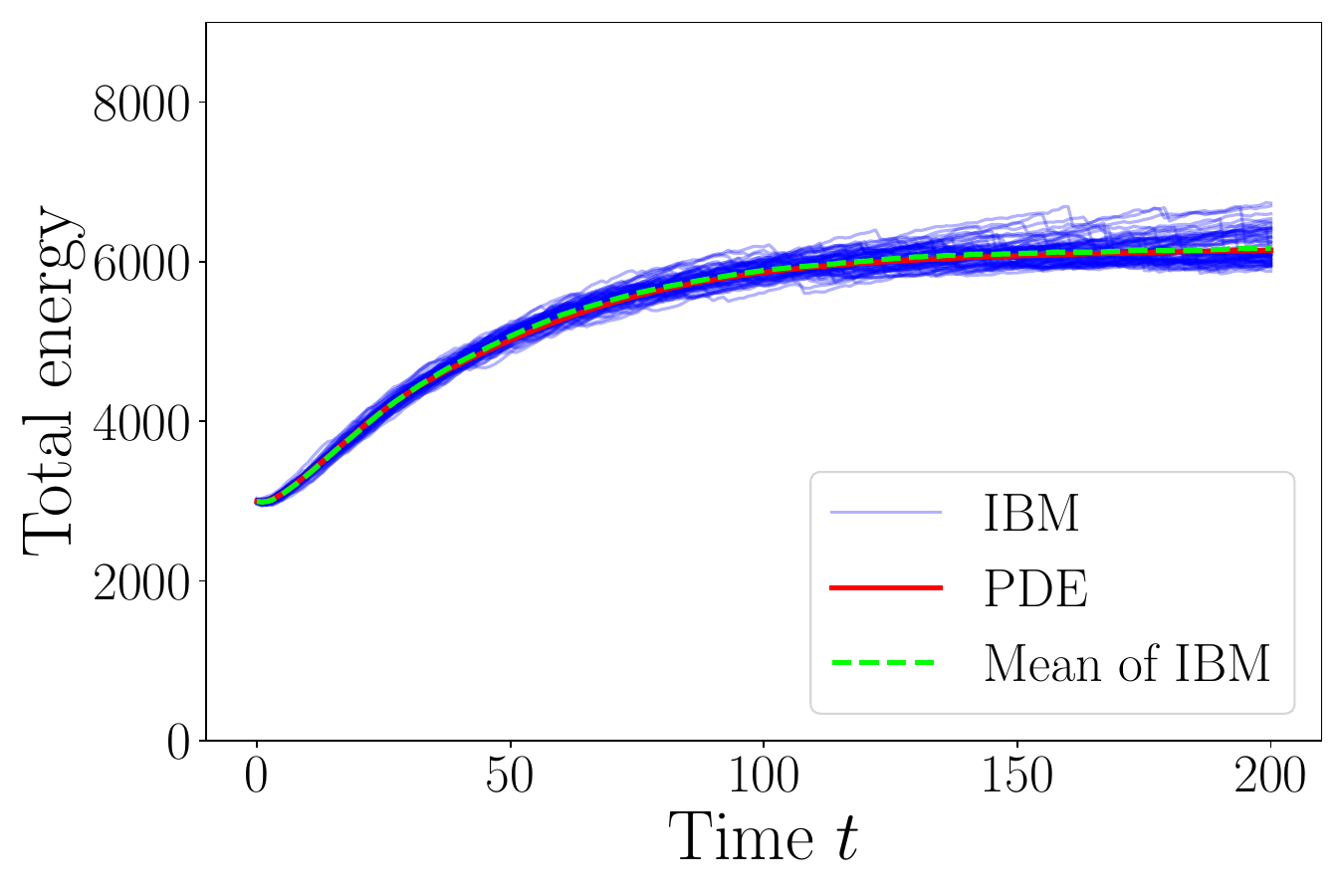}  
    \label{fig:energymille}
\end{subfigure}
\hspace{0.01 cm}
\begin{subfigure}{0.32\textwidth}
   \includegraphics[width=\textwidth]{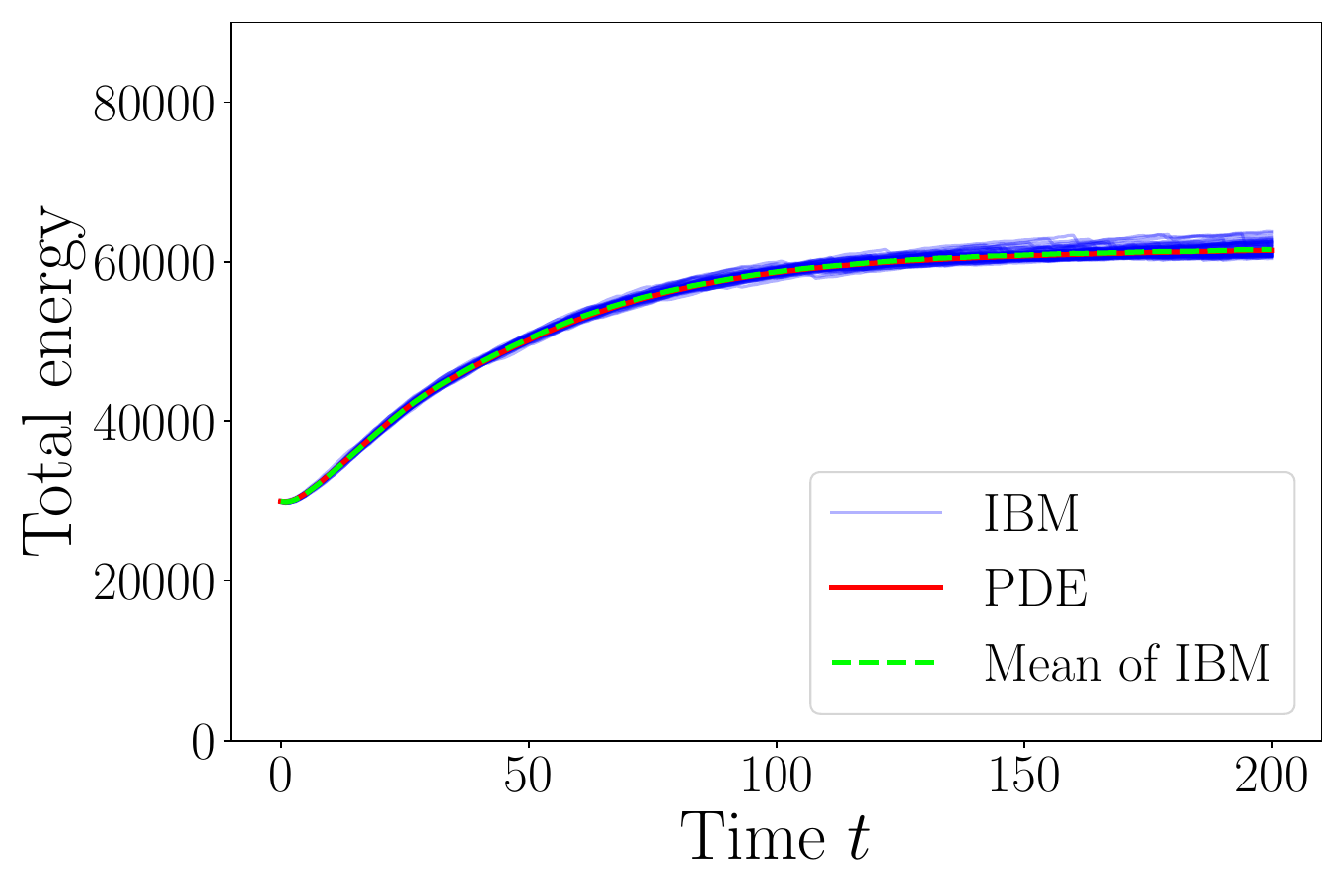}  
    \label{fig:energy10000}
\end{subfigure}
\hspace{0.01 cm}
\begin{subfigure}{0.32\textwidth}
   \includegraphics[width=\textwidth]{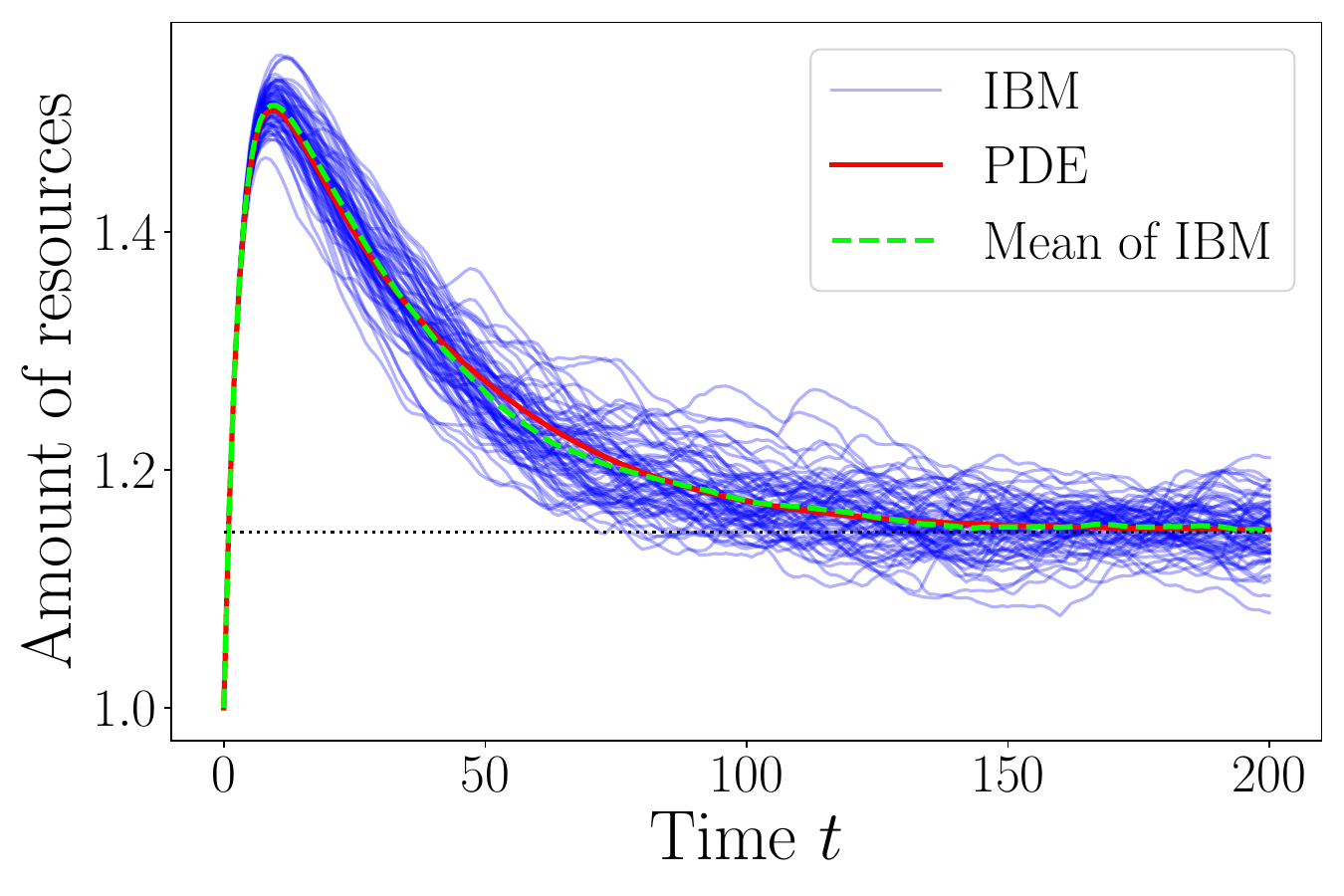}
    \label{fig:resourcecent}
\end{subfigure}
\hspace{0.01cm}
\begin{subfigure}{0.32\textwidth}
   \includegraphics[width=\textwidth]{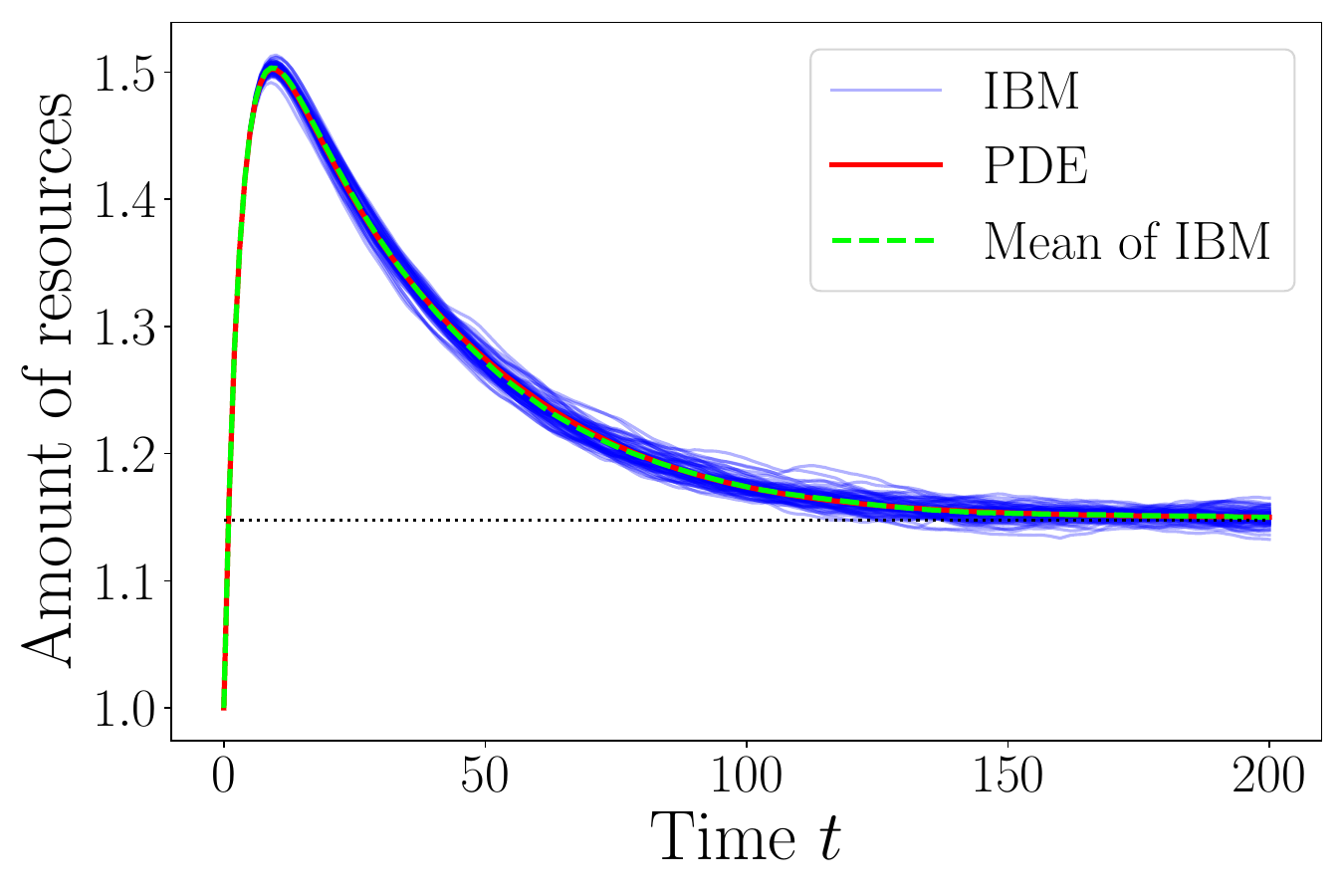}
    \label{fig:resourcemille}
\end{subfigure}
\hspace{0.01cm}
\begin{subfigure}{0.32\textwidth}
   \includegraphics[width=\textwidth]{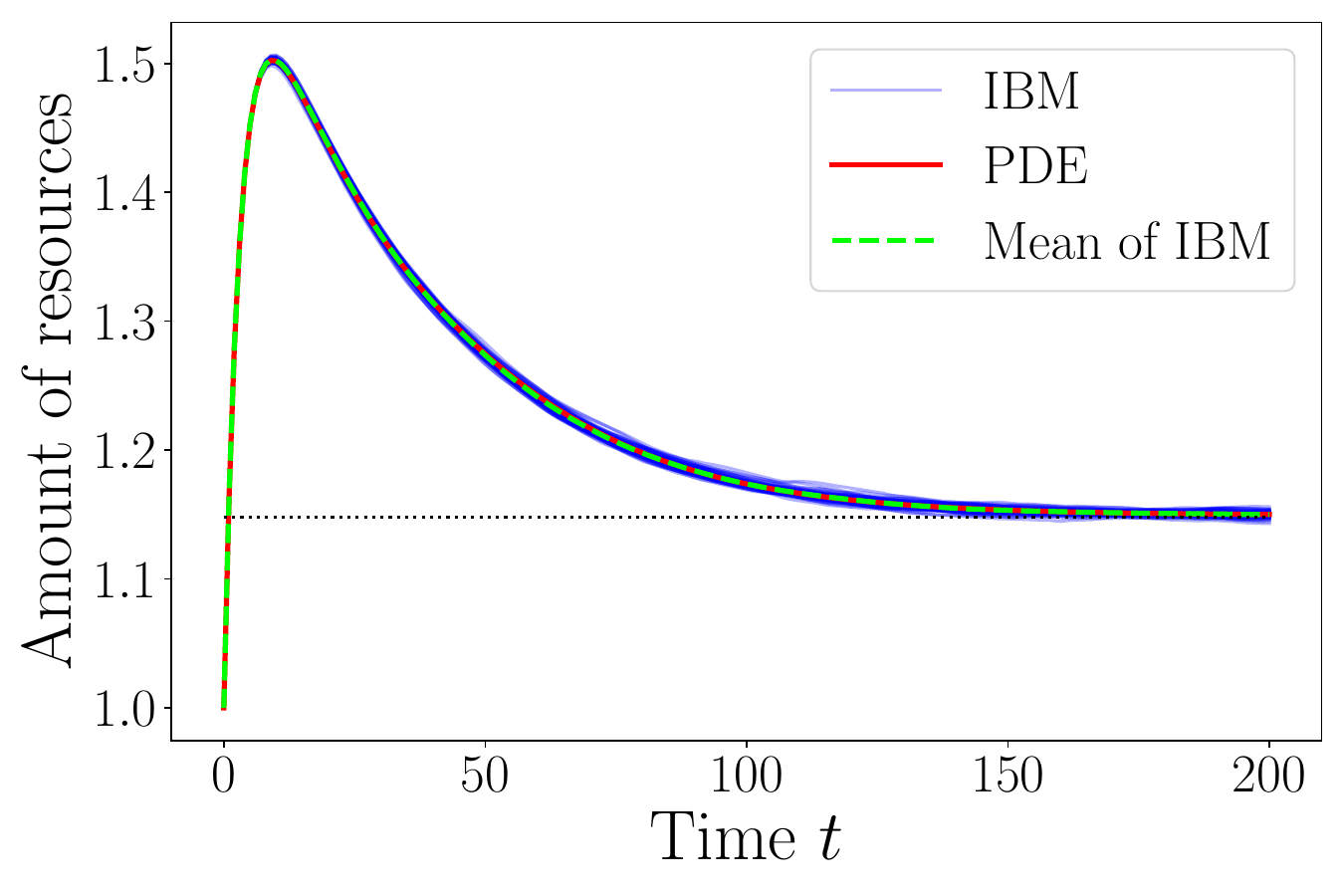}
    \label{fig:resource10000}
\end{subfigure}
\hspace{0.01cm}
\begin{subfigure}{0.32\textwidth}
  \includegraphics[width=\textwidth]{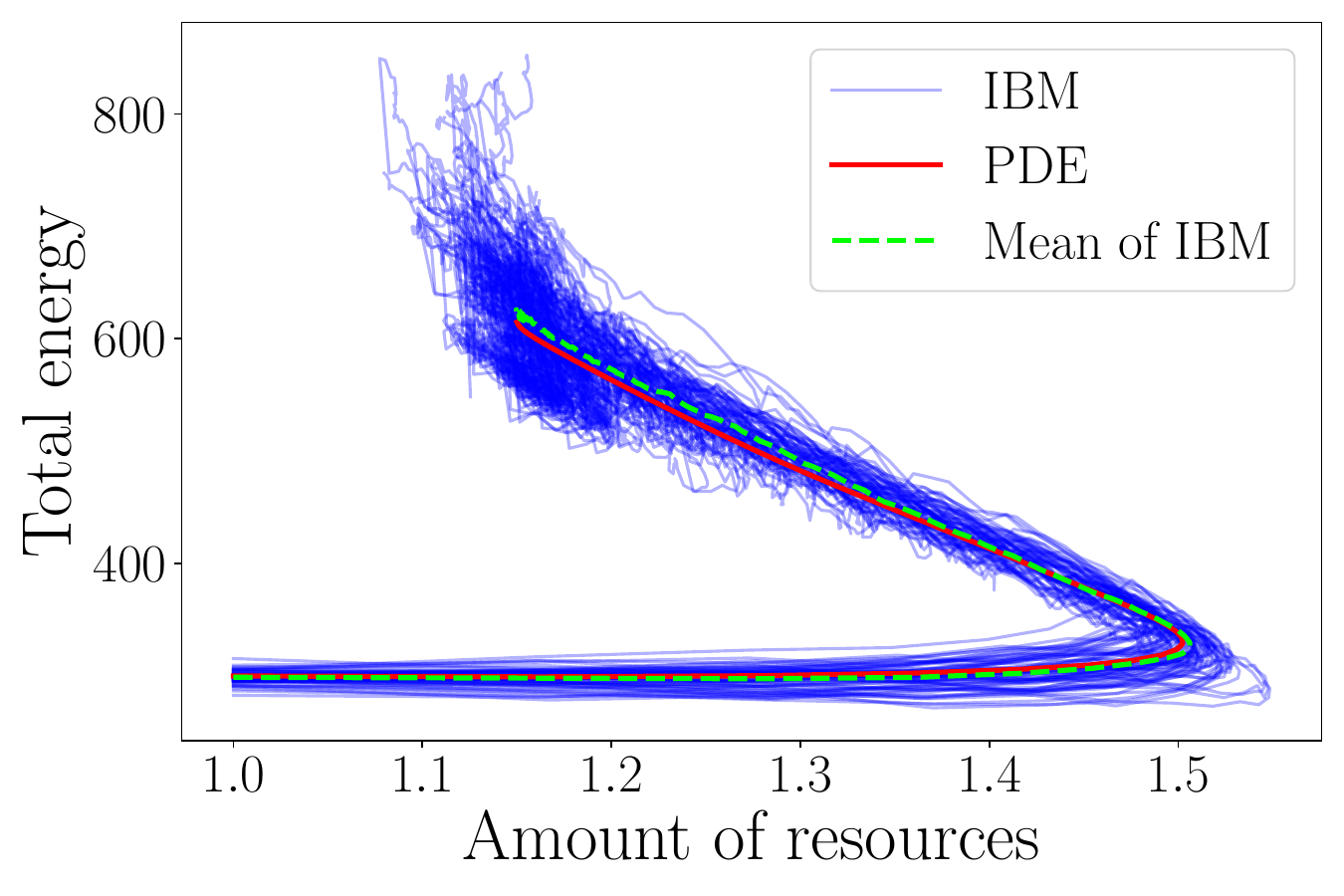}
  \caption{$K=100$}
    \label{fig:phasepcent}
 \end{subfigure}
    \begin{subfigure}{0.32\textwidth}
  \includegraphics[width=\textwidth]{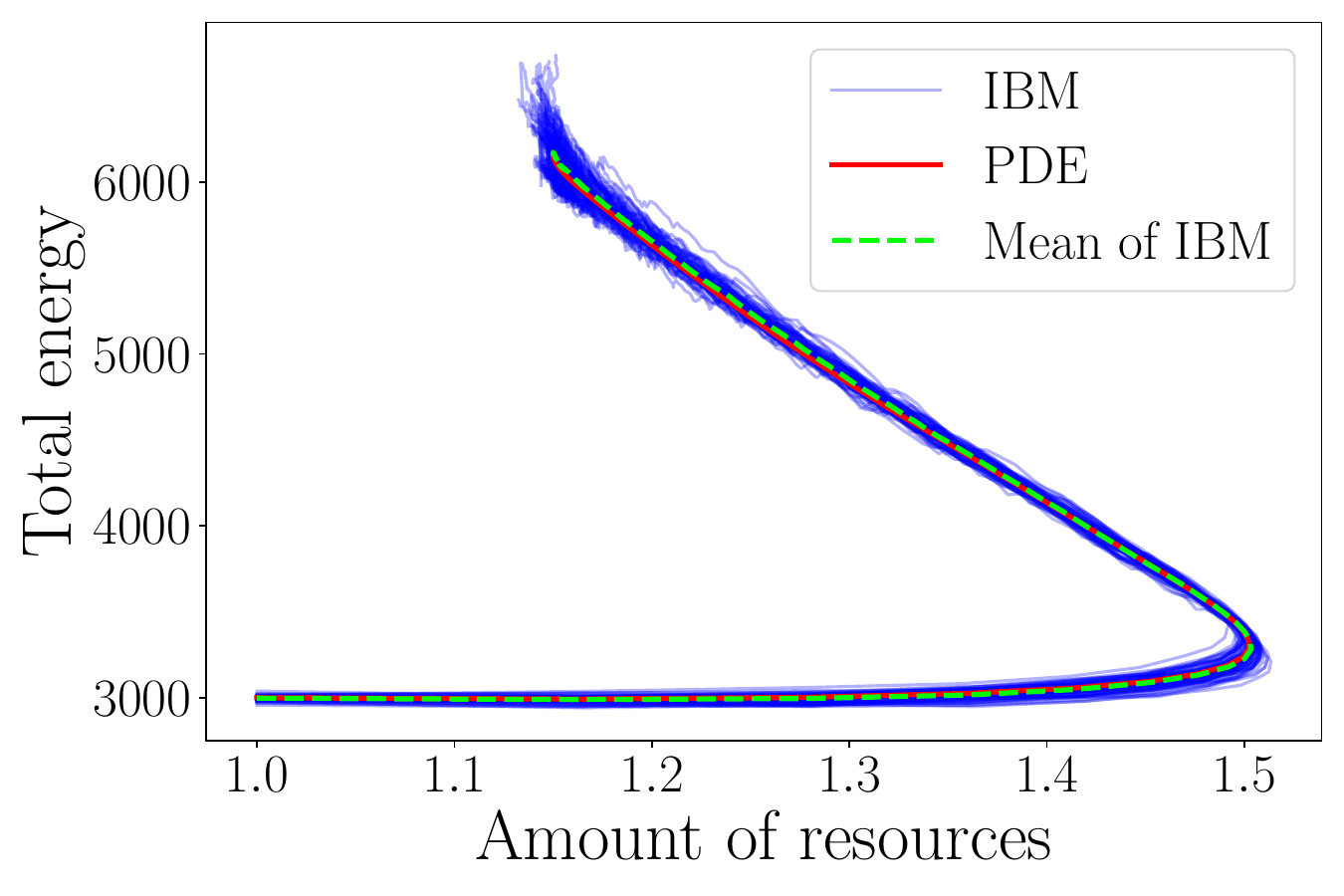}
  \caption{$K=1000$}
    \label{fig:phasepmille}
 \end{subfigure}
    \begin{subfigure}{0.32\textwidth}
  \includegraphics[width=\textwidth]{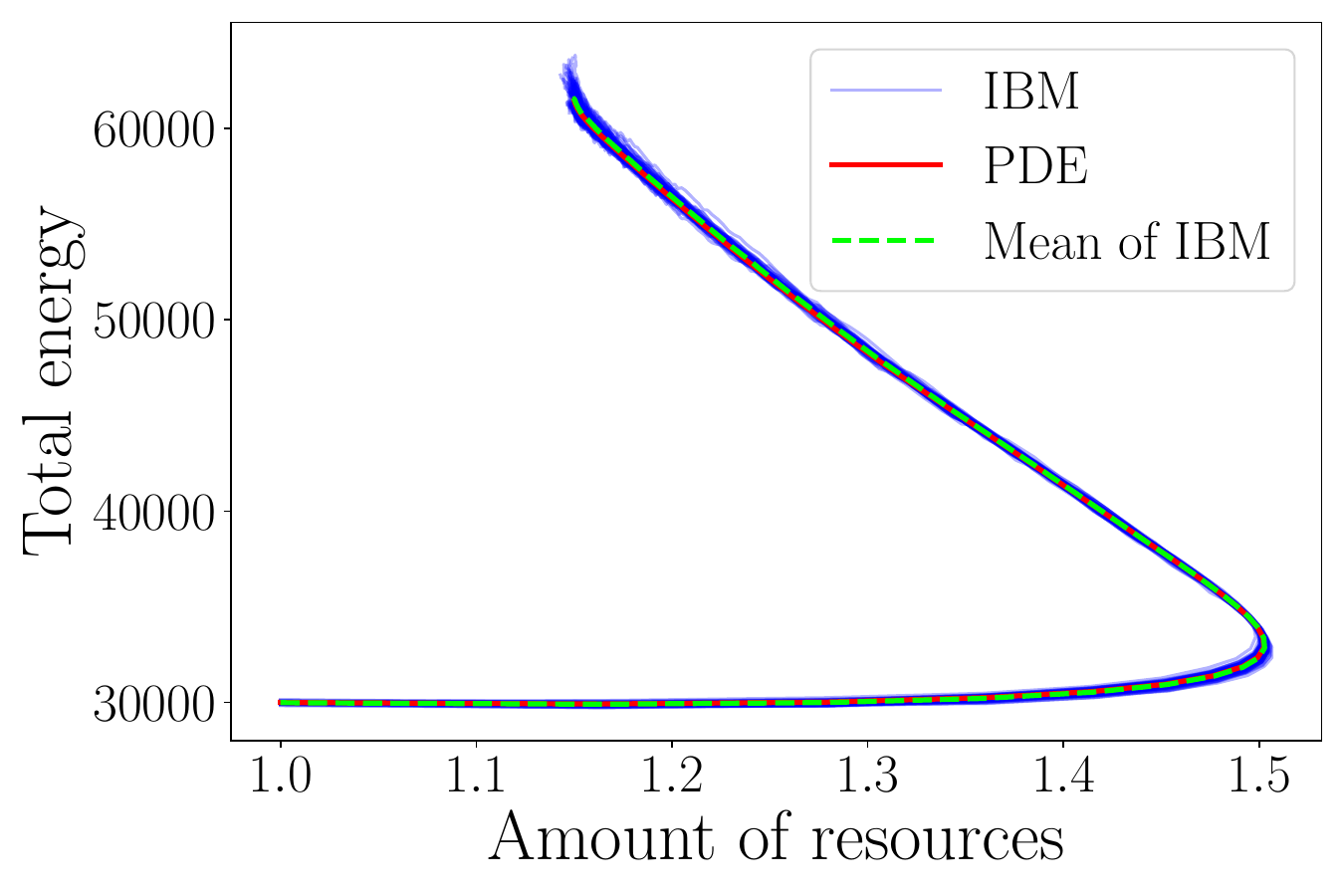}
  \caption{$K=10000$}
    \label{fig:phasep10000}
\end{subfigure}
\caption{From the first row to the third row, time evolutions of $N^{K}_{t}$, $E^{K}_{t}$ and $R^{K}_{t}$ (in blue), respectively $N^{*}_{t} := \langle \mu^{*}_{t}, 1 \rangle$, $E^{*}_{t} := \langle \mu^{*}_{t}, \mathrm{Id} \rangle$ and $R^{*}_{t}$ (in red), representing the population size, the total energy of the population and the amount of resources, and associated respectively with the trajectories of 100 independent IBM simulations (in blue) and the numerical resolution of the PDE system \eqref{eq:indivedp}, \eqref{eq:ressedp}, \eqref{eq:boundaryedp} with initial condition $u_{0}$ (in red). The fourth row presents the energy/resource phase portrait. The green dotted curve is the mean value of the stochastic simulations in blue. These graphs are presented for small (left), medium (middle) and large (right) initial population sizes. On the third row, the dotted black line locates the value $R_{\mathrm{eq}}$ assumed to be an equilibrium for the amount of resource.}
\label{fig:convergence}
\end{figure}

\begin{figure}[h!]
\centering
\begin{subfigure}{0.28\textwidth}
  \includegraphics[width=\textwidth]{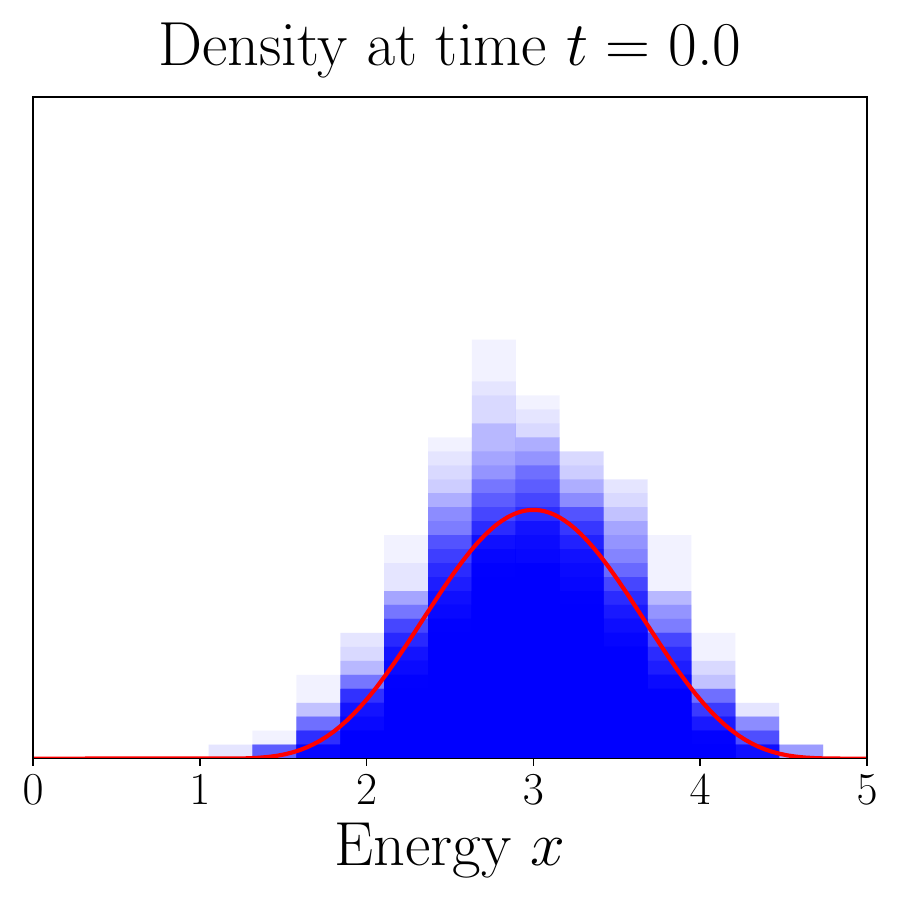}
    \label{fig:d0100}
 \end{subfigure}
 \begin{subfigure}{0.28\textwidth}
  \includegraphics[width=\textwidth]{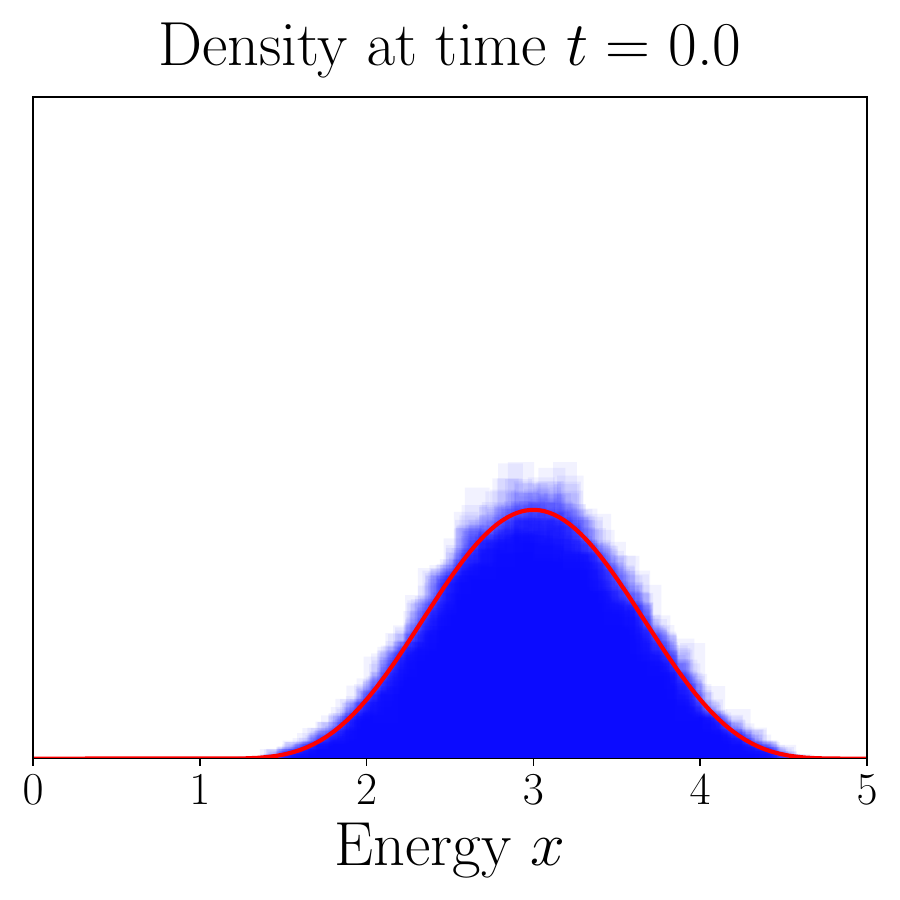}
    \label{fig:d01000}
 \end{subfigure}
 \begin{subfigure}{0.28\textwidth}
  \includegraphics[width=\textwidth]{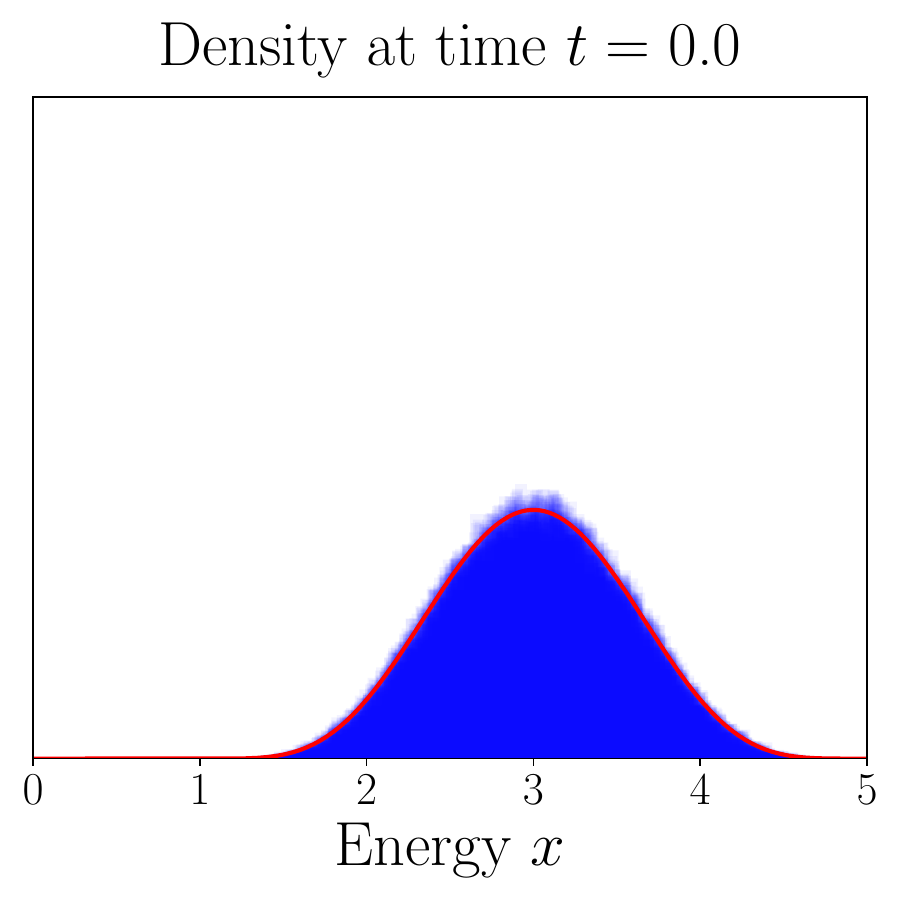}
    \label{fig:d0der}
 \end{subfigure}
 \begin{subfigure}{0.28\textwidth}
  \includegraphics[width=\textwidth]{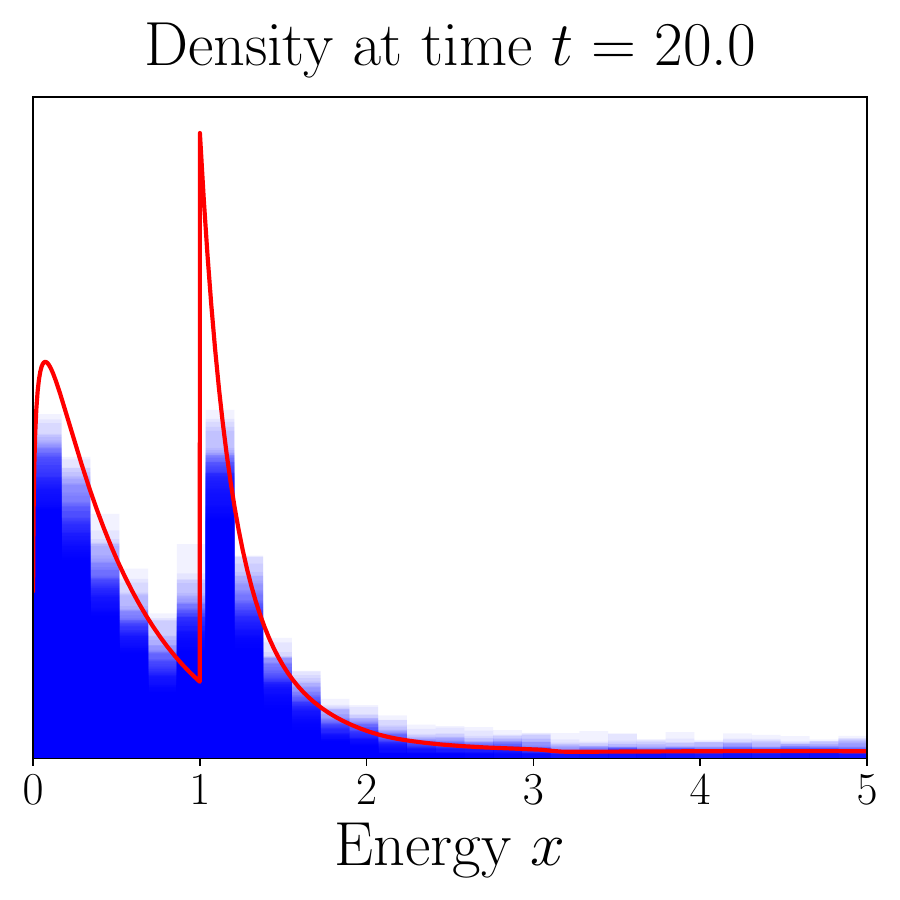}
    \label{fig:d20100}
 \end{subfigure}
    \begin{subfigure}{0.28\textwidth}
    \includegraphics[width=\textwidth]{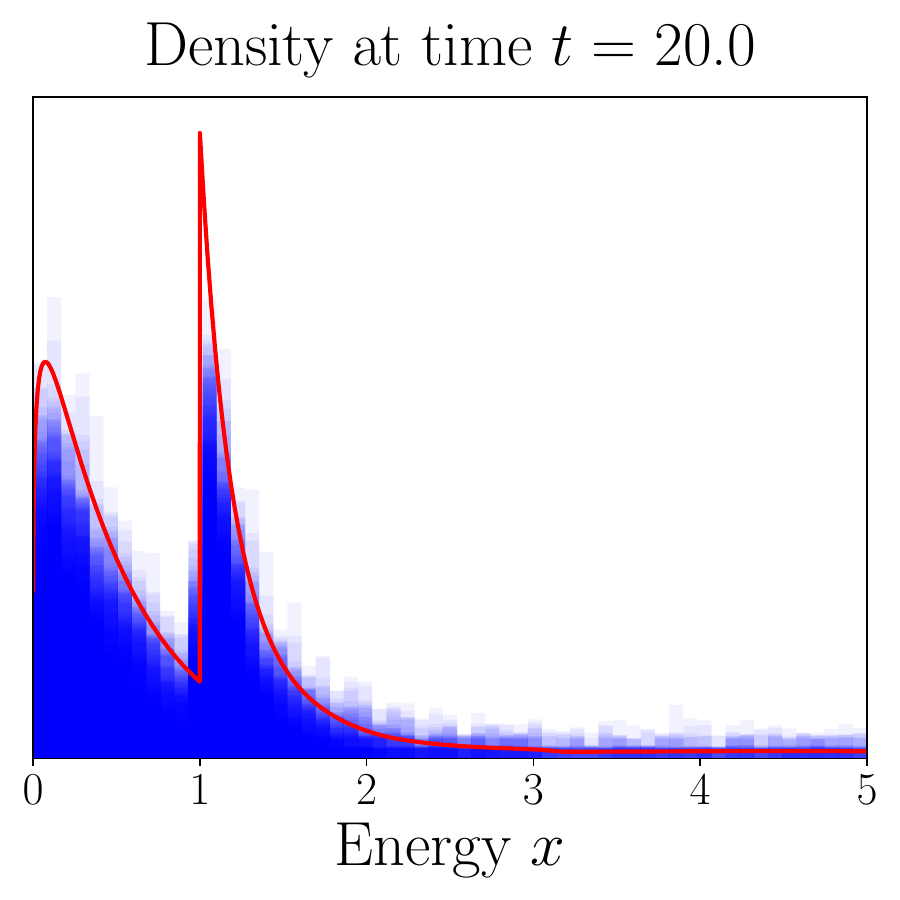}
    \label{fig:d201000}
 \end{subfigure}
 \begin{subfigure}{0.28\textwidth}
  \includegraphics[width=\textwidth]{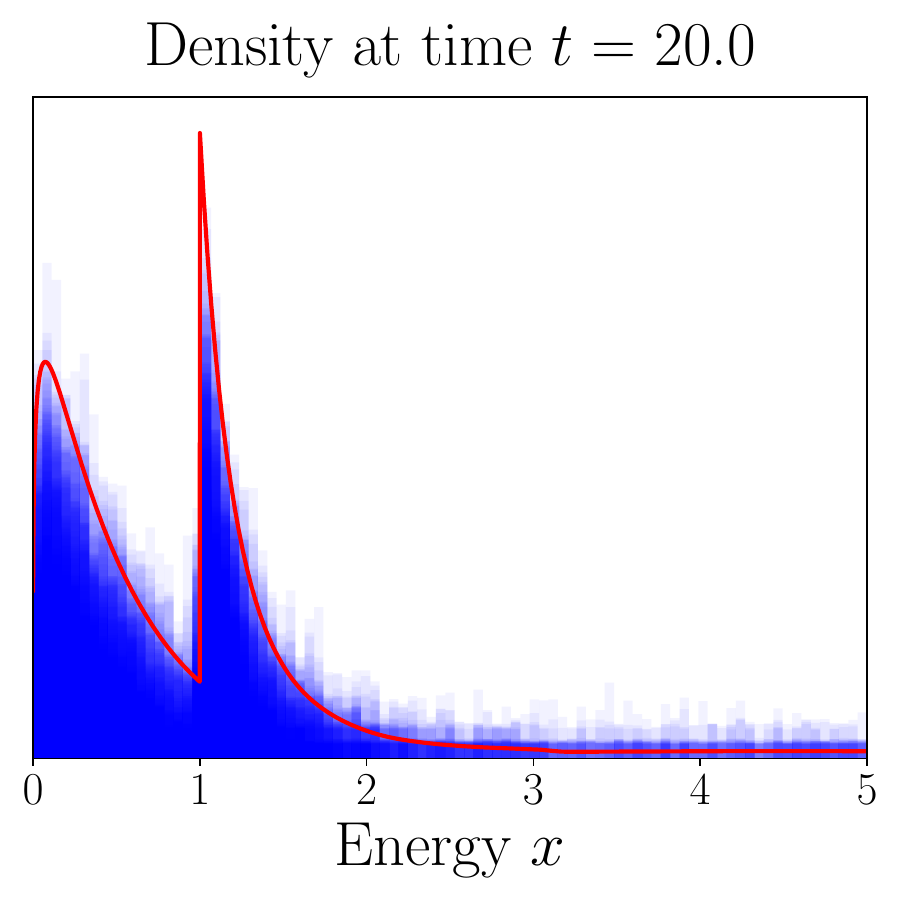}
    \label{fig:d20der}
 \end{subfigure}
    \begin{subfigure}{0.28\textwidth}
  \includegraphics[width=\textwidth]{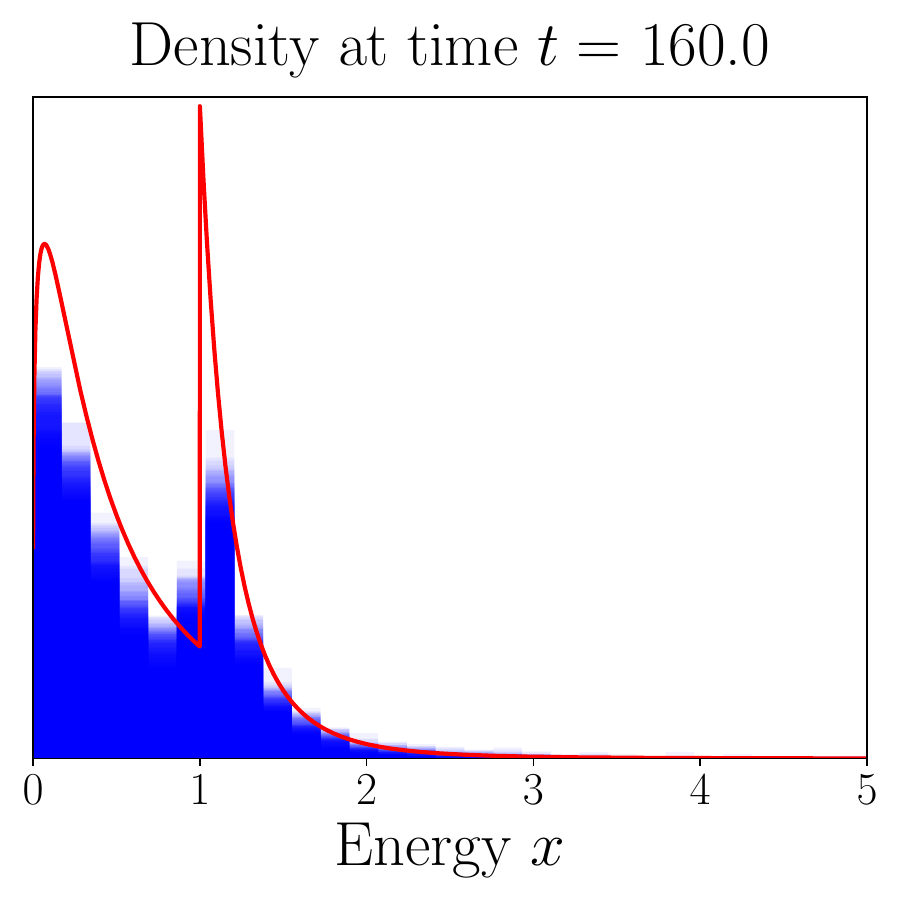}
  \caption{$K=100$}
    \label{fig:d160100}
\end{subfigure}
\begin{subfigure}{0.28\textwidth}
  \includegraphics[width=\textwidth]{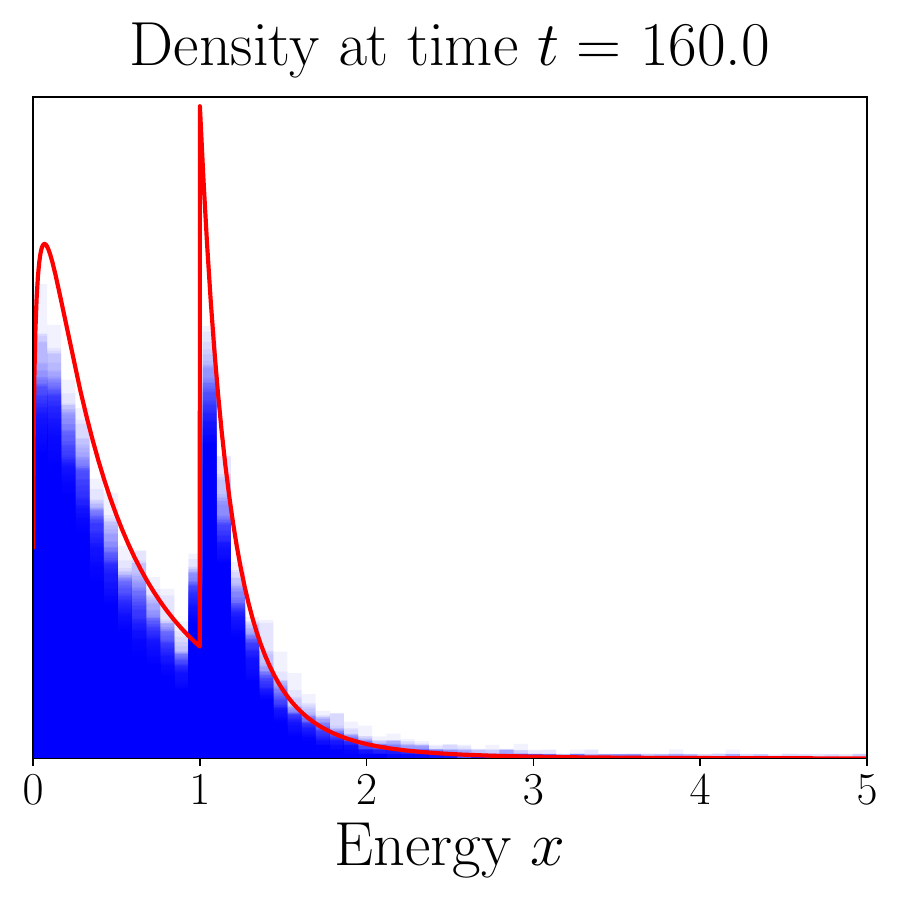}
  \caption{$K=1000$}
    \label{fig:d1601000}
\end{subfigure}
\begin{subfigure}{0.28\textwidth}
  \includegraphics[width=\textwidth]{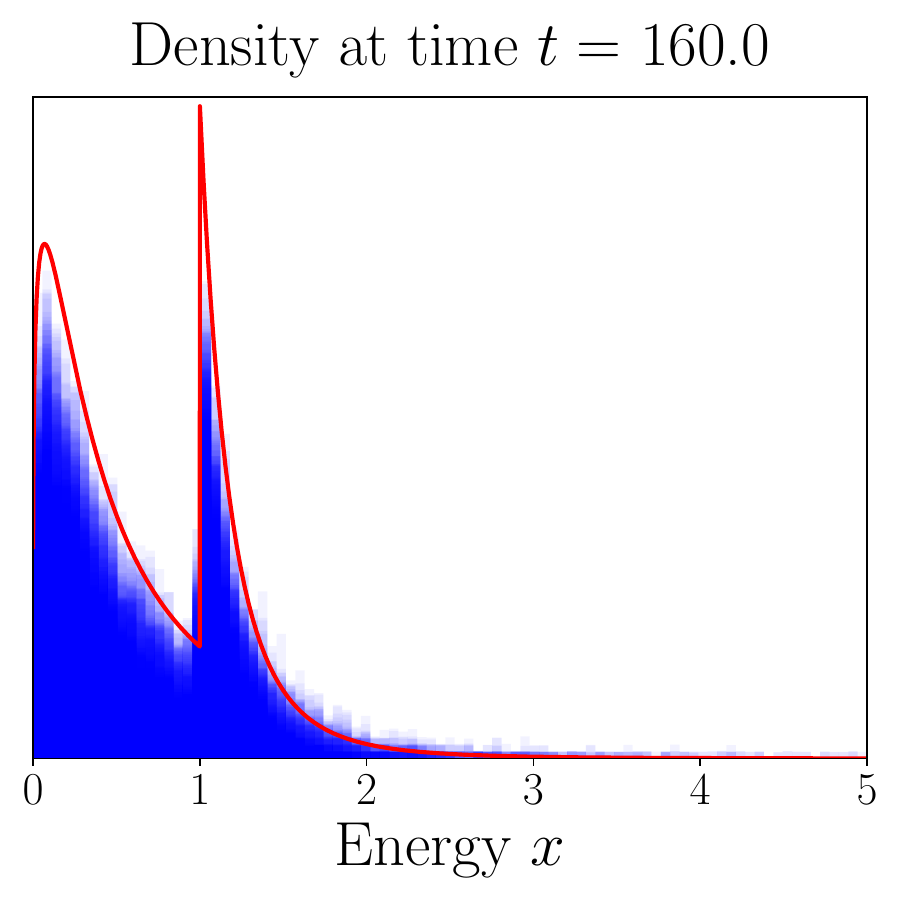}
  \caption{$K=10000$}
    \label{fig:d160der}
\end{subfigure}
\caption{Energy distribution on the energy window $[0,5]$ at time $t=0$ (above), 20 (middle) and 160 (bottom) for small (left), medium (middle) and large (right) initial population sizes. The red curve represents the renormalized energy distribution $\tilde{u}_{t}(.)$ for the corresponding value of $t$. The blue histogram represent the empirical energy distribution of individuals for 100 independent runs of IBM. The number of bins for the histograms is adapted to population sizes on each subfigure.}
\label{fig:densities}
\end{figure}

\section{Proof of Theorem~\ref{theo:convergence}}
\label{sec:proof}

We follow the \hyperlink{sketch}{sketch of the proof} highlighted in Section~\ref{sec:theorem}.

\subsection{Proof of the tightness of $\left(\mathscr{L}^{K}\right)_{K \in \mathbb{N}^{*}}$ in $\mathbb{D}([0,T], (\mathcal{M}_{\omega}(\mathbb{R}^{*}_{+}),v) \times [0,R_{\max}])$}
\label{subsec:tension}

First, we prove the tightness of $\left(\mathscr{L}^{K}_{\mu}\right)_{K \in \mathbb{N}^{*}}$, where for $K \in \mathbb{N}^{*}$, $\mathscr{L}^{K}_{\mu}$ is the law of $(\mu^{K}_{t})_{t \in [0,T]}$. We give the following criterion of tightness in $\mathbb{D}([0,T], (\mathcal{M}_{\omega}(\mathbb{R}^{*}_{+}),v))$, which is an extension of Theorem 2.1 in \cite{roel_86} to weighted spaces of measures. For every $f \in \mathfrak{B}_{\omega}(\mathbb{R}^{*}_{+})$, we define the projection
$$\begin{array}{lccc}
\pi_{f} : &  \mathbb{D}([0,T],(\mathcal{M}_{\omega}(\mathbb{R}^{*}_{+}),v)) & \longrightarrow & \mathbb{D}([0,T],\mathbb{R}) \\
& (\mu_{t})_{t \in [0,T]} & \longmapsto & (\langle \mu_{t}, f \rangle)_{t \in [0,T]}.
\end{array} $$
Also, we write $f \in \mathcal{C}_{0}(\mathbb{R}^{*}_{+})$, if $f$ is continous, $f(x) \xrightarrow[x \rightarrow 0]{} 0$ and $f(x) \xrightarrow[x \rightarrow + \infty]{} 0$.
\begin{theorem}
Let $T \geq 0$, $\left(P^{K}\right)_{K \in \mathbb{N}}$ be a sequence of probability measures on $\mathbb{D}([0,T],(\mathcal{M}_{\omega}(\mathbb{R}^{*}_{+}),v))$, and $D$ be a dense countable subset of $\mathcal{C}_{0}(\mathbb{R}^{*}_{+})$ for the topology of uniform convergence. Assume that for all $f \in D \cup \{ \omega\}$, $\left(\pi_{f}*P^{K}\right)_{K \in \mathbb{N}}$ is a tight sequence of probability measures on $\mathbb{D}([0,T],\mathbb{R})$, where $\pi_{f} * P^{K}$ is the usual pushforward of $P^{K}$ by $\pi_{f}$. Then $\left(P^{K}\right)_{K \in \mathbb{N}}$ is tight on $\mathbb{D}([0,T],(\mathcal{M}_{\omega}(\mathbb{R}^{*}_{+}),v))$.
\label{theo:roel}
\end{theorem}
We prove Theorem~\ref{theo:roel} in Appendix~\ref{appendix:roel}. It is well-known that there exists a countable set $D \subseteq \mathcal{C}^{\infty}_{c}(\mathbb{R}^{*}_{+})$, such that $D$ is dense in $\mathcal{C}_{0}(\mathbb{R}^{*}_{+})$ for the topology of uniform convergence (see Lemma II.4.1. in \cite{broduthesis}). Thus, thanks to Theorem~\ref{theo:roel}, it suffices to show that for all $\varphi \in \mathcal{C}^{\infty}_{c}(\mathbb{R}^{*}_{+})  \cup \{ \omega\}$, $\left(\pi_{\varphi}*\mathscr{L}_{\mu}^{K}\right)_{K \in \mathbb{N}}$ is a tight sequence of probability measures on $\mathbb{D}([0,T],\mathbb{R})$. Let $\varphi$ be such a function, and note that in particular, $\varphi \in \mathfrak{B}_{\omega}(\mathbb{R}^{*}_{+})$ and $\varphi$ is bounded. Then by Proposition~\ref{prop:martingalerenormalisee},  applied to $(t,x) \mapsto \varphi(x)$, for $t \in [0,T]$, we have the semi-martingale decomposition
        \[ \langle \mu^{K}_{t}, \varphi \rangle = V^{K}_{\varphi,t} + \heartsuit^{K}_{\varphi,t}, \]
where $(\heartsuit^{K}_{\varphi,t})_{t \geq 0}$ is a square-integrable martingale. The tightness of $\left(\pi_{\varphi}*\mathscr{L}_{\mu}^{K}\right)_{K \in \mathbb{N}}$ in $\mathbb{D}([0,T],\mathbb{R})$ is proven thanks to a criterion from Aldous and Rebolledo (\cite{ald_86}, Corollary 2.3.3). It suffices to show that
        \begin{enumerate}
            \item For every $t \in [0,T]$, the sequence of laws of $(\langle \mu^{K}_{t}, \varphi \rangle)_{K \geq 1}$ is tight in $\mathbb{R}$.
            \item For every $t \in [0,T]$, for every $\varepsilon >0$, for every $\eta > 0$, there exists $\delta >0$ and $K_{0} \geq 1$, such that for every sequence of stopping times $(S_{K},T_{K})_{K \in \mathbb{N}^{*}}$ such that $S_{K} \leq T_{K} \leq t$ for all $K \in \mathbb{N}^{*}$, we have
            \begin{align}           
 \underset{K \geq K_{0}}{\mathrm{sup}} \mathbb{P}\bigg( |\langle \heartsuit^{K}_{\varphi} \rangle_{T_{K}} - \langle \heartsuit^{K}_{\varphi} \rangle_{S_{K}}| \geq \eta, T_{K} \leq S_{K} + \delta \bigg) \leq \varepsilon, \label{eq:crochetcontrole} \\
                \underset{K \geq K_{0}}{\mathrm{sup}} \mathbb{P}\bigg( | V^{K}_{\varphi,T_{K}} - V^{K}_{\varphi,S_{K}}| \geq \eta, T_{K} \leq S_{K} + \delta \bigg) \leq \varepsilon.
               \label{eq:finitevarcontrol}
         \end{align}
            \end{enumerate}
First, we use Markov inequality to obtain, for any $t \in [0,T]$, $M >0$, $K \in \mathbb{N}^{*}$,
    \begin{align*}
        \mathbb{P}(| \langle \mu^{K}_{t}, \varphi \rangle | \geq M) & \leq \dfrac{1}{M} \mathbb{E}(| \langle \mu^{K}_{t}, \varphi \rangle |) \\
        & \leq \dfrac{1}{M} \sup_{K \in \mathbb{N}^{*}} \mathbb{E}( \sup_{t \in [0,T]} |\langle \mu^{K}_{t}, \varphi \rangle |).
    \end{align*} 
Remark that $\varphi \in \mathfrak{B}_{\omega}(\mathbb{R}^{*}_{+})$, so Proposition~\ref{lemme:controlenp} with $p=1$ entails that the sequence of laws of $(\langle \mu^{K}_{t}, \varphi \rangle)_{K \geq 1}$ is tight in $\mathbb{R}$. Then, we fix $t \in [0,T]$, $\delta >0$ and $(S_{K},T_{K})_{K \in \mathbb{N}^{*}}$ a sequence of stopping times such that $S_{K} \leq T_{K} \leq t$ for all $K \in \mathbb{N}^{*}$. By Lemma~\ref{lemme:equivalence}, Proposition~\ref{prop:martingalerenormalisee}, and using the fact that $\varphi \in \mathcal{C}^{\infty}_{c}(\mathbb{R}^{*}_{+})  \cup \{ \omega\}$, so $\varphi$ is bounded or $\varphi=\omega$, there exists a constant $C>0$ such that for $K \geq 1$,
\begin{align*}
    \mathbb{E}\left(|\langle \heartsuit^{K}_{\varphi} \rangle_{T_{K}} - \langle \heartsuit^{K}_{\varphi} \rangle_{S_{K}}| \mathbb{1}_{\{ | T_{K}-S_{K}| \leq \delta \}}\right) & \\
&  \hspace{-3 cm}  =  \dfrac{1}{K} \mathbb{E} \bigg(\mathbb{1}_{\{ | T_{K}-S_{K}| \leq \delta \}} \displaystyle{\int_{S_{K}}^{T_{K}}\int_{\mathbb{R}^{*}_{+}}}\bigg[b(x)\bigg( \varphi(x_{0})+\varphi(x-x_{0})-\varphi(x) \bigg)^{2} \\
    & \hspace{3 cm} + d(x)\varphi^{2}(x) \bigg] \mu^{K}_{s}(\mathrm{d}x) \mathrm{d}s \bigg) \\
& \hspace{-3 cm}  \leq \dfrac{C}{K} \mathbb{E}\left( \mathbb{1}_{\{ | T_{K}-S_{K}| \leq \delta \}} \displaystyle{\int_{S_{K}}^{T_{K}}}\left\langle \mu^{K}_{s}, 1+ \mathrm{Id}+\omega \right\rangle \mathrm{d}s \right) \\ 
 & \hspace{-3cm}   \leq \dfrac{C}{K}\delta \sup_{K \in \mathbb{N}^{*}} \mathbb{E}\left( \sup_{t \in [0,T]} \left(E^{K}_{t} + N^{K}_{t}+ \Omega_{t}^{K}\right) \right). 
\end{align*}
By Proposition~\ref{lemme:controlenp} with $p=1$ and using Markov inequality, for every $\varepsilon>0$ and $\eta >0$, we can thus find $\delta$ and $K_{0}$ such that \eqref{eq:crochetcontrole} holds true uniformly on the choice of the sequence $(S_{K},T_{K})_{K \in \mathbb{N}^{*}}$. Similarly, we use Lemma~\ref{lemme:equivalence}, Proposition~\ref{prop:martingalerenormalisee}, the fact that $\varphi \in \mathcal{C}^{\infty}_{c}(\mathbb{R}^{*}_{+})  \cup \{ \omega\}$, so $\varphi$ and $\varphi'$ are bounded and $\varphi \in \mathfrak{B}_{\omega}(\mathbb{R}^{*}_{+})$, or $\varphi \equiv \omega$, to assess that
    \begin{align*}
        \mathbb{E}\left(| V^{K}_{\varphi,T_{K}} - V^{K}_{\varphi,S_{K}}| \mathbb{1}_{\{ | T_{K}-S_{K}| \leq \delta \}} \right) \\
& \hspace{-3cm} \leq \mathbb{E}\bigg( \mathbb{1}_{\{ | T_{K}-S_{K}| \leq \delta \}} \displaystyle{\int_{S_{K}}^{T_{K}} \int_{\mathbb{R}^{*}_{+}} }  \bigg\{ \overline{g}(x)|\varphi'(x)| +  d(x)|\varphi(x)|\\
        & \hspace{1cm}+ b(x)\bigg| \varphi(x_{0})+\varphi(x-x_{0})-\varphi(x)\bigg| \bigg\} \mu^{K}_{s}(\mathrm{d}x) \mathrm{d}s \bigg) \\
 & \hspace{-3cm} \leq \omega_{1} \mathbb{E}\bigg( \mathbb{1}_{\{ | T_{K}-S_{K}| \leq \delta \}} \displaystyle{\int_{S_{K}}^{T_{K}} \int_{\mathbb{R}^{*}_{+}} }  \bigg( 1+x+\omega(x) \bigg) \mu^{K}_{s}(\mathrm{d}x) \mathrm{d}s \bigg) \\
 & \hspace{-3 cm} \leq \omega_{1}\delta \sup_{K \in \mathbb{N}^{*}} \mathbb{E}\left( \sup_{t \in [0,T]} \left(E^{K}_{t} + N^{K}_{t}+ \Omega_{t}^{K}\right) \right),
    \end{align*}
and we conclude in the same manner. Now, let us show the tightness of $\left(\mathscr{L}^{K}_{R}\right)_{K \in \mathbb{N}^{*}}$, where for $K \in \mathbb{N}^{*}$, $\mathscr{L}^{K}_{R}$ is the law of $(R^{K}_{t})_{t \in [0,T]}$. We use a simpler criterion of Aldous, without decomposing $R^{K}_{t}$ into a finite variation part and a martingale part. First, for every $t \geq 0$, for every $K \in \mathbb{N}^{*}$, then $R^{K}_{t} \in [0,R_{\max}]$, so for every $t \in [0,T]$, $(R^{K}_{t})_{K \geq 1}$ is tight in $\mathbb{R}$. Then from Theorem 16.10. in \cite{bill_99}, it suffices to show that for every $t \in [0,T]$, for every $\varepsilon>0$, for every $\eta >0$, there exists $\delta >0$ and $K_{0} \geq 1$, such that for every sequence of stopping times $(S_{K},T_{K})_{K \in \mathbb{N}^{*}}$ with $S_{K} \leq T_{K} \leq t$ for all $K \in \mathbb{N}^{*}$, we have
\begin{align*} 
\sup\limits_{K \geq K_{0}} \mathbb{P}\left(|R^{K}_{T_{K}}-R^{K}_{S_{K}}| \geq \eta, T_{K} \leq S_{K} + \delta\right) \leq \varepsilon.
\end{align*}
Let us fix $t \in  [0,T]$, $\delta>0$, $K \geq 1$ and a sequence of stopping times as defined previously, then we have by \eqref{eq:eqress} that
\begin{align*}
    \mathbb{E}\left(|R^{K}_{T_{K}}-R^{K}_{S_{K}}| \mathbb{1}_{\{T_{K} \leq S_{K} + \delta\}} \right) & \\
 & \hspace{-1 cm} = \mathbb{E}\bigg(\bigg| \displaystyle{\int_{S_{K}}^{T_{K}}} \varsigma(R^{K}_{s})  - \chi\langle \mu^{K}_{s}, f(.,R^{K}_{s}) \rangle  \mathrm{d}s \bigg| \mathbb{1}_{\{T_{K} \leq S_{K} + \delta\}} \bigg) \\
 & \hspace{-1 cm}   \leq  \delta ||\varsigma||_{\infty,[0,R_{\max}]} + \mathbb{E}\left( \chi \displaystyle{\int_{S_{K}}^{T_{K}}} \langle \mu^{K}_{s}, \overline{g} \rangle  \mathrm{d}s \; \mathbb{1}_{\{T_{K} \leq S_{K} + \delta\}} \right) \\
 & \hspace{-1 cm}   \leq  \delta \left( ||\varsigma||_{\infty,[0,R_{\max}]} + \chi C_{g} \sup\limits_{K \in \mathbb{N}^{*}} \mathbb{E}\left(  \sup\limits_{t \in [0,T]} \left( E^{K}_{t}+ N^{K}_{t}+ \Omega^{K}_{t} \right) \right) \right),   
\end{align*}
where we used the fact that for every $R \geq 0$ and $x>0$, $|f(x,R)| \leq \overline{g}(x)$ and Assumption~\ref{hyp:poidsomega}. This concludes by Markov inequality and Proposition~\ref{lemme:controlenp} with $p=1$. At this step, we have shown the tightness of $\left(\mathscr{L}^{K}\right)_{K \in \mathbb{N}^{*}}$ in $\mathbb{D}([0,T], (\mathcal{M}_{\omega}(\mathbb{R}^{*}_{+}),v) \times [0,R_{\max}])$. Then, we can use Prokhorov theorem, because $\mathbb{D}([0,T], (\mathcal{M}_{\omega}(\mathbb{R}^{*}_{+}),v) \times [0,R_{\max}])$ is metrizable (see Theorem 5.1. in \cite{bill_99}). This theorem states that we can extract a subsequence, still denoted as $\bigg(\left(\mu^{K}_{t},R^{K}_{t}\right)_{t \in [0,T]}\bigg)_{K \in \mathbb{N}^{*}}$ for the sake of simplicity, that converges in law towards some $(\mu^{*}_{t},R^{*}_{t})_{t \in [0,T]}$ in $\mathbb{D}([0,T], (\mathcal{M}_{\omega}(\mathbb{R}^{*}_{+}),v) \times [0,R_{\max}])$.

\subsection{Continuity of accumulation points}
\label{subsec:continuitylimit}

In this section, we show the continuity of the limit $(\mu^{*}_{t},R^{*}_{t})_{t \in [0,T]}$ highlighted at the end of Section~\ref{subsec:tension}, which is essential in the use of Theorem~\ref{theo:melroel} in Section~\ref{subsec:topotricks}, and in the characterization of the limit in Section~\ref{subsec:identification}. We follow Step 2 in Section 5 of \cite{jourdain_11} and begin with the continuity with respect to the topology of vague convergence.

\begin{lemme}
We work under the assumptions of Theorem~\ref{theo:convergence}. Then, any limit $(\mu^{*}_{t},R^{*}_{t})_{t \in [0,T]}$ of a subsequence of $\bigg(\left(\mu^{K}_{t},R^{K}_{t}\right)_{t \in [0,T]}\bigg)_{K \in \mathbb{N}^{*}}$ converging in law in $\mathbb{D}([0,T], (\mathcal{M}_{\omega}(\mathbb{R}^{*}_{+}),v) \times [0,R_{\max}])$ is in $\mathcal{C}([0,T], (\mathcal{M}_{\omega}(\mathbb{R}^{*}_{+}),v) \times [0,R_{\max}])$. 
\label{lemme:continuiti}
\end{lemme}
\begin{proof}
For the sake of simplicity, we write again $\bigg((\mu^{K}_{t},R^{K}_{t})_{t \in [0,T]}\bigg)_{K \in \mathbb{N}^{*}}$ for the converging subsequence. We begin with the continuity of $(\mu^{*}_{t})_{t \in [0,T]}$. By a simple adaptation of Theorem 10.2 p.148 in \cite{ek_2005} to our weighted context, it suffices to show that almost surely,
$$\sup\limits_{t \in [0,T]} \mathrm{d}_{P}^{\omega}(\mu^{K}_{t},\mu^{K}_{t-}) \xrightarrow[K \rightarrow + \infty]{} 0, $$
where $\mathrm{d}_{P}^{\omega}$ is the $\omega$-Prokhorov distance defined in Definition~\ref{defi:prok}. By Lemma B.2.5 in \cite{broduthesis}, we can replace the $\omega$-Prokhorov distance in the previous convergence by the $\omega$-Fortet-Mourier distance $d_{\mathrm{FM}}^{\omega}$, defined for every $\mu, \nu \in \mathcal{M}_{\omega}(\mathbb{R}^{*}_{+})$ by
$$ d_{\mathrm{FM}}^{\omega}(\mu,\nu) := \displaystyle{\sup_{\varphi \in \mathcal{D}}} |\langle \mu-\nu, \varphi \rangle|,$$ 
where $\mathcal{D} := (\varphi_{n})_{n \in \mathbb{N}}$ is a countable and dense subset of the set $\left\{ \varphi \in \mathcal{C}^{\infty}_{c}(\mathbb{R}^{*}_{+}), \; ||\varphi/ \omega||_{\infty} \leq 1, \; ||\varphi'||_{\infty} \leq 1 \right\}$ for the topology of uniform convergence (which exists by Lemma II.4.1 in \cite{broduthesis}). Thus, to conclude, it suffices to show that almost surely 
$$\sup\limits_{t \in [0,T]} \mathrm{d}^{\omega}_{\mathrm{FM}}(\mu^{K}_{t},\mu^{K}_{t-}) = \sup\limits_{t \in [0,T]} \sup\limits_{n \in \mathbb{N}} |\langle \mu^{K}_{t}-\mu^{K}_{t-}, \varphi_{n} \rangle| = \sup\limits_{n \in \mathbb{N}} \sup\limits_{t \in [0,T]} |\langle \mu^{K}_{t}-\mu^{K}_{t-}, \varphi_{n} \rangle| \xrightarrow[K \rightarrow + \infty]{} 0. $$
Without loss of generality, we can prove that this convergence holds true in $L^{1}$, because this implies almost sure convergence up to extraction, and our argument using Theorem 10.2 p.148 in \cite{ek_2005} remains true up to extracting a new subsequence that still converges in law towards $(\mu^{*}_{t})_{t \in [0,T]}$ in $\mathbb{D}([0,T], (\mathcal{M}_{\omega}(\mathbb{R}^{*}_{+}),v) \times [0,R_{\max}])$. 
Furthermore, for $K \geq 1$ and $n \geq 0$, we have the domination
$$\sup\limits_{t \in [0,T]} |\langle \mu^{K}_{t}-\mu^{K}_{t-}, \varphi_{n} \rangle| \leq 2 \sup\limits_{t \in [0,T]} \Omega^{K}_{t}.$$
The right-hand side above does not depend on $n$, and its expectation is bounded uniformly on $K$ by Proposition~\ref{lemme:controlenp} with $p=1$. Hence, by this domination argument, it suffices to show that
$$ \mathbb{E}\left( \sup\limits_{n \in \mathbb{N}} \sup\limits_{t \in [0,T]} |\langle \mu^{K}_{t}-\mu^{K}_{t-}, \varphi_{n} \rangle|\right) = \sup\limits_{n \in \mathbb{N}} \mathbb{E}\left( \sup\limits_{t \in [0,T]} \left|\langle \mu^{K}_{t}-\mu^{K}_{t-}, \varphi_{n} \rangle\right|\right) \xrightarrow[K \rightarrow + \infty]{} 0.$$
Let us consider any $\varphi \in \mathcal{C}^{\infty}_{c}(\mathbb{R}^{*}_{+})$ with $||\varphi/ \omega||_{\infty} \leq 1$ and $||\varphi'||_{\infty} \leq 1$. For $t \in [0,T]$ and $K \geq 1$, we have from the decomposition of Proposition~\ref{prop:martingalerenormalisee} that
$$ \langle \mu^{K}_{t}-\mu^{K}_{t-}, \varphi \rangle = V^{K}_{\varphi,t}-V^{K}_{\varphi,t-} + \heartsuit^{K}_{\varphi,t} - \heartsuit^{K}_{\varphi,t-} = \heartsuit^{K}_{\varphi,t} - \heartsuit^{K}_{\varphi,t-},$$
because $t \in [0,T] \mapsto V^{K}_{\varphi,t} $ is continuous. Hence, we obtain
$$ \sup\limits_{t \in [0,T]} |\langle \mu^{K}_{t}-\mu^{K}_{t-}, \varphi \rangle| \leq 2 \sup\limits_{t \in [0,T]} \left|\heartsuit^{K}_{\varphi,t} \right|.$$
Then, we use Doob maximal inequality for square-integrable martingales to obtain
$$ \mathbb{E}\left( \sup\limits_{t \in [0,T]} \left|\heartsuit^{K}_{\varphi,t} \right| \right) \leq 4 \mathbb{E}\left( \left\langle \heartsuit^{K}_{\varphi} \right\rangle_{T} \right) \leq \dfrac{4(C_{b}(\omega(x_{0})+x_{0})^{2}+C_{d})}{K} \mathbb{E}\left( E^{K}_{T} + N^{K}_{T} + \Omega^{K}_{T} \right), $$
by Proposition~\ref{prop:martingalerenormalisee}, the fact that $||\varphi/\omega||_{\infty} \leq 1$ and $||\varphi'||_{\infty} \leq 1$, and Assumption~\ref{hyp:poidsomega} (in particular, we use \eqref{eq:conditionomegadeux} in Lemma~\ref{lemme:equivalence}). This upper bound is uniform on $\varphi$ and converges to 0 when $K \rightarrow + \infty$ thanks to Proposition~\ref{lemme:controlenp} with $p=1$, which concludes for the continuity of $(\mu^{*}_{t})_{t \in [0,T]}$. Finally, we use again Theorem 10.2 p.148 in \cite{ek_2005} for the continuity of $(R^{*}_{t})_{t \in [0,T]}$, and it suffices to show that
$$ \sup\limits_{t \in [0,T]} \left| R^{K}_{t} - R^{K}_{t-} \right| \xrightarrow[K \rightarrow + \infty]{} 0, $$
which is immediate because every $(R^{K}_{t})_{t \in [0,T]}$ is continuous by construction.
\end{proof}

\textbf{Remark:} In the previous proof, note that we need to use the $\omega$-Fortet-Mourier distance instead of the total variation distance, which is classically used in similar contexts (see Lemma 5.7 in \cite{FRITSCH20151}). This is essentially because we need to control quantities of the form $|\langle \mu^{K}_{t}-\mu^{K}_{t-}, \varphi \rangle| $, where $\varphi$ is not necessarily bounded, but only dominated by $\omega$.
\\\\
Now, we want to show the continuity of any accumulation point with respect to the $\omega$-weak topology. We begin with a preliminary result.

\begin{lemme}
We work under the assumptions of Theorem~\ref{theo:convergence} and consider a subsequence of the sequence $\left(\left(\mu^{K}_{t},R^{K}_{t}\right)_{t \in [0,T]}\right)_{K \in \mathbb{N}^{*}}$ that converges in law towards $(\mu^{*}_{t},R^{*}_{t})_{t \in [0,T]}$ in $\mathbb{D}([0,T], (\mathcal{M}_{\omega}(\mathbb{R}^{*}_{+}),v) \times [0,R_{\max}])$. Let $\varphi : (t,x) \in \mathbb{R}^{+} \times \mathbb{R}^{*}_{+} \mapsto \varphi_{t}(x)$ be continuous and such that
\begin{align*}
\exists C>0, \forall x>0, \quad \sup_{t \in [0,T]} |\varphi_{t}(x)| \leq C(1+x+ \omega(x)).
\end{align*}
Then, we have
\begin{align*}
\mathbb{E}\left( \sup_{t \in [0,T]} \left| \langle \mu^{*}_{t}, \varphi_{t} \rangle \right|\right) <+ \infty.
\end{align*}
\label{corr:plusdefonctions}
\end{lemme}
\begin{proof}
Let $T \geq 0$, for the sake of simplicity, we still denote as $\left(\left(\mu^{K}_{t},R^{K}_{t}\right)_{t \in [0,T]}\right)_{K \in \mathbb{N}^{*}}$ the converging subsequence in the assumptions of Lemma~\ref{corr:plusdefonctions}. Without loss of generality, we will show that 
\begin{align*}
\mathbb{E}\left( \sup_{t \in [0,T]}  \langle \mu^{*}_{t}, 1+ \mathrm{Id} + \omega \rangle \right) <+ \infty.
\end{align*}
There exists an increasing sequence $(\varphi_{n})_{n \in \mathbb{N}}$ of $\mathcal{C}^{\infty}_{c}(\mathbb{R}^{*}_{+})$ non-negative functions that converges simply towards $1 + \mathrm{Id}+\omega$, so that by monotone convergence, for every $t \in [0,T]$,
$$ \langle \mu^{*}_{t}, 1+\mathrm{Id}+\omega \rangle = \lim_{n \rightarrow + \infty} \langle \mu^{*}_{t}, \varphi_{n} \rangle \quad \mathrm{and} \quad \forall x >0, \quad \sup_{n \in \mathbb{N}} \varphi_{n}(x) \leq 1+x+\omega(x).$$
Thanks to Lemma~\ref{lemme:continuiti} and Proposition 2.4 p.303 in \cite{Jacod1987LimitTF}, the mapping $\mu \mapsto \sup_{t \in [0,T]}  \langle \mu_{t}, \varphi_{n} \rangle $ is continuous on $\mathbb{D}([0,T], (\mathcal{M}_{\omega}(\mathbb{R}^{*}_{+}),v))$. Hence, for every $n \in \mathbb{N}$ and $t \in [0,T]$, we have almost surely $ \sup_{t \in [0,T]} \langle \mu^{*}_{t}, \varphi_{n} \rangle = \lim\limits_{K \rightarrow + \infty} \sup_{t \in [0,T]} \langle \mu^{K}_{t}, \varphi_{n} \rangle$, and thanks to Corollary~\ref{corr:uniformly}, the family of random variables $\left(\sup\limits_{t \in [0,T]} \langle \mu^{K}_{t}, \varphi_{n} \rangle \right)_{K \in \mathbb{N}^{*}}$ is uniformly integrable. Thus, we obtain
\begin{align*}
\mathbb{E} \left( \sup_{t \in [0,T]}\langle \mu^{*}_{t}, \varphi_{n} \rangle  \right) & =  \lim_{K \rightarrow + \infty} \mathbb{E}\left(\sup_{t \in [0,T]}  \langle \mu^{K}_{t}, \varphi_{n} \rangle \right) \leq  \sup\limits_{K \geq 1} \mathbb{E}\left( \sup_{t \in [0,T]} \left( E^{K}_{t} + N^{K}_{t} + \Omega^{K}_{t}\right)\right).
\end{align*}
We conclude by monotone convergence applied to the left-hand-side above, and Proposition~\ref{lemme:controlenp} with $p=1$ for the right-hand-side.
\end{proof}

\begin{corr}
We work under the assumptions of Theorem~\ref{theo:convergence}. Then, any limit $(\mu^{*}_{t},R^{*}_{t})_{t \in [0,T]}$ of a subsequence of $\bigg(\left(\mu^{K}_{t},R^{K}_{t}\right)_{t \in [0,T]}\bigg)_{K \in \mathbb{N}^{*}}$ converging in law in $\mathbb{D}([0,T], (\mathcal{M}_{\omega}(\mathbb{R}^{*}_{+}),v) \times [0,R_{\max}])$ is in $\mathcal{C}([0,T], (\mathcal{M}_{\omega}(\mathbb{R}^{*}_{+}),\mathrm{w}) \times [0,R_{\max}])$. 
\label{lemme:continuity}
\end{corr}

\begin{proof}
It suffices to show that for every continuous function $\varphi \in \mathfrak{B}_{\omega}(\mathbb{R}^{*}_{+})$ and $t \in [0,T]$, we have almost surely
$$\langle \mu^{*}_{t} - \mu^{*}_{t-}, \varphi \rangle = 0.$$
We fix such a function $\varphi$, and without loss of generality, we will show that
$$\mathbb{E}(|\langle \mu^{*}_{t} - \mu^{*}_{t-}, \varphi \rangle|) = 0.$$
We use an approximation argument: there exists a sequence $(\varphi_{n})_{n \in \mathbb{N}}$ of $\mathcal{C}^{\infty}_{c}(\mathbb{R}^{*}_{+})$ functions that converges simply towards $\varphi$. Thanks to Lemma~\ref{corr:plusdefonctions}, we can use dominated convergence to obtain
$$ \mathbb{E}(|\langle \mu^{*}_{t} - \mu^{*}_{t-}, \varphi \rangle|) = \lim\limits_{n \rightarrow + \infty} \mathbb{E}(|\langle \mu^{*}_{t} - \mu^{*}_{t-}, \varphi_{n} \rangle|) = 0,  $$
where the last equality comes from Lemma~\ref{lemme:continuiti}.
\end{proof}

\subsection{Proof of the tightness of $\left(\mathscr{L}^{K}\right)_{K \in \mathbb{N}^{*}}$ in $\mathbb{D}([0,T], (\mathcal{M}_{\omega}(\mathbb{R}^{*}_{+}),\mathrm{w}) \times [0,R_{\max}])$}
\label{subsec:topotricks}

Considering Section~\ref{subsec:tension}, it suffices to prove the tightness of $\left(\mathscr{L}^{K}_{\mu}\right)_{K \in \mathbb{N}^{*}}$ in $\mathbb{D}([0,T], (\mathcal{M}_{\omega}(\mathbb{R}^{*}_{+}),\mathrm{w}))$. To this end, we extend Théorème 3. in \cite{meleard1993convergences} to weighted spaces of measures.
\begin{theorem}
Let $w$ be any positive and continuous function on $\mathbb{R}^{*}_{+}$, $((\nu^{K}_{t})_{t \in [0,T]})_{K \in \mathbb{N}}$ be a sequence of processes in $\mathbb{D}([0,T], (\mathcal{M}_{w}(\mathbb{R}^{*}_{+}),\mathrm{w})$ and $(\nu^{*}_{t})_{t \in [0,T]}$ a process in the space $\mathcal{C}([0,T], (\mathcal{M}_{w}(\mathbb{R}^{*}_{+}),\mathrm{w}))$. Then, the following assertions are equivalent
\begin{itemize}
\item[$\mathrm{(i)}$] $((\nu^{K}_{t})_{t \in [0,T]})_{K \in \mathbb{N}}$ converges in law towards $(\nu^{*}_{t})_{t \in [0,T]}$ in $\mathbb{D}([0,T], (\mathcal{M}_{w}(\mathbb{R}^{*}_{+}),\mathrm{w}))$.
\item[$\mathrm{(ii)}$] $((\nu^{K}_{t})_{t \in [0,T]})_{K \in \mathbb{N}}$ converges in law towards $(\nu^{*}_{t})_{t \in [0,T]}$ in $\mathbb{D}([0,T], (\mathcal{M}_{w}(\mathbb{R}^{*}_{+}),v))$, and the sequence $((\langle \nu^{K}_{t}, w \rangle)_{t \in [0,T]})_{K \in \mathbb{N}}$ converges in law towards $(\langle \nu^{*}_{t}, w \rangle)_{t \in [0,T]}$ in $\mathbb{D}([0,T], \mathbb{R})$.
\end{itemize}
\label{theo:melroel}
\end{theorem}
We provide the proof of this result in Appendix~\ref{appendix:melroel}, and use it in the following with the converging subsequence still denoted as $(\mu^{K})_{K \geq 1}$ constructed in Section~\ref{subsec:tension}, its limit $\mu^{*}$ and $w \equiv \omega$. Note that we could also have formulated Theorem~\ref{theo:melroel} in terms of tightness of the sequences instead of convergence. We proved in Section~\ref{subsec:tension}, using Aldous and Rebolledo criterion, that $((\langle \mu^{K}_{t}, \omega \rangle)_{t \in [0,T]})_{K \in \mathbb{N}}$ is tight in $\mathbb{D}([0,T], \mathbb{R})$. Hence, also from Section~\ref{subsec:tension}, extracting again if necessary and using Prokhorov theorem, we can find a subsequence of $\bigg(\left(\mu^{K}_{t},R^{K}_{t}\right)_{t \in [0,T]}\bigg)_{K \in \mathbb{N}^{*}}$ that converges in law towards some $(\mu^{*}_{t},R^{*}_{t})_{t \in [0,T]}$ in $\mathbb{D}([0,T], (\mathcal{M}_{\omega}(\mathbb{R}^{*}_{+}),v) \times [0,R_{\max}])$, such that $((\langle \mu^{K}_{t}, \omega \rangle)_{t \in [0,T]})_{K \in \mathbb{N}}$ converges in law towards some $(x^{*}_{t})_{t \in [0,T]}$ in $\mathbb{D}([0,T], \mathbb{R})$. By Corollary~\ref{lemme:continuity}, we also know that these limits are continuous. Hence, to use Theorem~\ref{theo:melroel} and conclude this section, it remains to show that $(x^{*}_{t})_{t \in [0,T]}$ and $(\langle \mu^{*}_{t}, \omega \rangle)_{t \in [0,T]}$ have same law. Using Skorokhod representation theorem (Theorem 6.7. p.70 in \cite{bill_99}) if necessary, we assume that the previous convergences hold true almost surely. By a straightforward approximation argument by test functions as in the proof of Theorem~\ref{theo:melroel} in Appendix~\ref{appendix:melroel}, we then obtain almost surely
$$ \forall t \in [0,T], \quad x^{*}_{t} \geq \langle \mu^{*}_{t}, \omega \rangle.$$
Furthermore, if $(\varphi_{n})_{n \in \mathbb{N}}$ is a sequence of $\mathcal{C}^{\infty}_{c}(\mathbb{R}^{*}_{+})$ positive functions that converges pointwise to $\omega$ such that $\varphi_{n} \leq \omega$ for $n \in \mathbb{N}$, we obtain that for every $t \in [0,T]$
\begin{align*}
\mathbb{E}\left(x^{*}_{t} -\langle \mu^{*}_{t}, \omega \rangle\right) & = \mathbb{E}\left( \lim\limits_{K \rightarrow + \infty} \lim\limits_{n \rightarrow + \infty} \left(  \langle \mu^{K}_{t}, \varphi_{n} \rangle -\langle \mu^{*}_{t}, \varphi_{n} \rangle \right) \right) \\
& \leq \lim\limits_{K \rightarrow + \infty} \lim\limits_{n \rightarrow + \infty} \mathbb{E}\left(   \langle \mu^{K}_{t}, \varphi_{n} \rangle -\langle \mu^{*}_{t}, \varphi_{n} \rangle \right) \\
& = \lim\limits_{n \rightarrow + \infty} \lim\limits_{K \rightarrow + \infty} \mathbb{E}\left( \langle \mu^{K}_{t}, \varphi_{n} \rangle -\langle \mu^{*}_{t}, \varphi_{n} \rangle \right) = 0,
\end{align*}
where we first used the convergence of $((\langle \mu^{K}_{t}, \omega \rangle)_{t \in [0,T]})_{K \in \mathbb{N}}$ towards $(x^{*}_{t})_{t \in [0,T]}$ in $\mathbb{D}([0,T], \mathbb{R})$, then Fatou lemma, followed by a domination argument (because every $\varphi_{n}$ is dominated by $\omega$ and we have the uniform control in $K$ of Proposition~\ref{lemme:controlenp} with $p=1$), and finally used the convergence of $\bigg(\left(\mu^{K}_{t}\right)_{t \in [0,T]}\bigg)_{K \in \mathbb{N}^{*}}$ towards $(\mu^{*}_{t})_{t \in [0,T]}$ in $\mathbb{D}([0,T], (\mathcal{M}_{\omega}(\mathbb{R}^{*}_{+}),v))$ and domination by $\omega$ again. Hence, for every $t \in [0,T]$, we obtain almost surely $x^{*}_{t} = \langle \mu^{*}_{t}, \omega \rangle$, which concludes. 

\subsection{Characterization of accumulation points}
\label{subsec:identification}

We fix $T \geq 0$, and continue to write $\left((\mu^{K}_{t},R^{K}_{t})_{t \in [0,T]}\right)_{K \in \mathbb{N}^{*}}$ for a subsequence converging in law in $\mathbb{D}([0,T], (\mathcal{M}_{\omega}(\mathbb{R}^{*}_{+}),\mathrm{w}) \times [0,R_{\max}])$ towards some $(\mu^{*}_{t},R^{*}_{t})_{t \in [0,T]}$. We fix these notations in the following. The aim of this section is to prove the upcoming Proposition~\ref{prop:identification}, which characterizes the law of the limit $(\mu^{*},R^{*})$. We begin with preliminary results using Assumption~\ref{ass:finallejd}.

\begin{prop}
We work under the assumptions of Theorem~\ref{theo:convergence}. Then, the following convergence in law in the space $\mathbb{D}([0,T], \mathbb{R})$ holds true
\begin{align*} 
\left(\langle \mu^{K}_{t}, \varpi \rangle\right)_{t \in [0,T]} \xrightarrow[K \longrightarrow + \infty]{}  \left(\langle \mu^{*}_{t}, \varpi \rangle\right)_{t \in [0,T]}. 
\end{align*}
As a consequence, using the results of Section~\ref{subsec:tension} and Theorem~\ref{theo:roel} again, we obtain that the process $\left((\mu^{K}_{t},R^{K}_{t})_{t \in [0,T]}\right)_{K \in \mathbb{N}^{*}}$ converges in law in $\mathbb{D}([0,T], (\mathcal{M}_{\varpi+\omega}(\mathbb{R}^{*}_{+}),\mathrm{v}) \times [0,R_{\max}])$ towards $(\mu^{*}_{t},R^{*}_{t})_{t \in [0,T]}$.
\label{prop:holderproper}
\end{prop}
\begin{proof}
Using the Skorokhod representation theorem, we assume that the convergence of the sequence $\bigg(\left(\mu^{K}_{t},R^{K}_{t}\right)_{t \in [0,T]}\bigg)_{K \in \mathbb{N}^{*}}$ towards $(\mu^{*}_{t},R^{*}_{t})_{t \in [0,T]}$ in $\mathbb{D}([0,T], (\mathcal{M}_{\omega}(\mathbb{R}^{*}_{+}),\mathrm{w}) \times [0,R_{\max}])$ is almost sure in the following. We then aim to show that
$$\mathbb{E}\left(\sup\limits_{t \in [0,T]} \left|\langle \mu^{K}_{t}-\mu^{*}_{t}, \varpi \rangle \right| \right) \xrightarrow[K \longrightarrow + \infty]{} 0.$$
We consider $M>0$ and for $t \in [0,T]$, $K \geq 1$, we write
$$ \langle \mu^{K}_{t}, \varpi \rangle = \left\langle \mu^{K}_{t}, \varpi \mathbb{1}_{\left\{\varpi \leq M \omega\right\}} \right\rangle + \left\langle \mu^{K}_{t}, \varpi \mathbb{1}_{\left\{\varpi > M \omega\right\}} \right\rangle. $$
Hence, we have, as $\varpi$ is non-negative,
\begin{multline}
\mathbb{E}\left(\sup\limits_{t \in [0,T]} \left|\langle \mu^{K}_{t}-\mu^{*}_{t}, \varpi \rangle \right| \right) \leq \mathbb{E}\left(\sup\limits_{t \in [0,T]} |\left\langle \mu^{K}_{t}-\mu^{*}_{t}, \varpi \mathbb{1}_{\left\{\varpi \leq M \omega\right\}} \right\rangle | \right) \\
 + \mathbb{E}\left(\sup\limits_{t \in [0,T]} \left\langle \mu^{*}_{t}, \varpi \mathbb{1}_{\left\{\varpi > M \omega\right\}} \right\rangle  \right) + \sup\limits_{K \geq 1} \mathbb{E}\left(\sup\limits_{t \in [0,T]}\left\langle \mu^{K}_{t}, \varpi \mathbb{1}_{\left\{\varpi > M \omega\right\}} \right\rangle  \right).
\label{eq:hodleriscmierfb}
\end{multline}
Let us focus on the integrand of the right-most term for a fixed $t \in [0,T]$ and use Hölder inequality with $p:= \frac{1}{1-\eta}$ and $q:= 1/\eta$ to obtain
\begin{align*}
\left\langle \mu^{K}_{t}, \varpi \mathbb{1}_{\left\{\varpi > M \omega\right\}} \right\rangle & = \left\langle \mu^{K}_{t}, \dfrac{\varpi}{\omega^{1/q}} \omega^{1/q} \mathbb{1}_{\left\{\varpi > M \omega\right\}} \right\rangle \\
& \leq \left\langle \mu^{K}_{t}, \dfrac{\varpi^{p}}{\omega^{p/q}} \right\rangle^{1/p} \times \left\langle \mu^{K}_{t},  \omega \mathbb{1}_{\left\{\varpi > M \omega\right\}} \right\rangle^{1/q}
\end{align*} 
Then, by Assumption~\ref{ass:finallejd}, $1/\omega$ is bounded on a neighborhood of $+ \infty$. This associated to the facts that $p/q >0$, $\omega$ is continuous and $\varpi \equiv 0$ on $(0,x_{0})$ entails that there exists a constant $C>0$ with $\dfrac{\varpi^{p}(x)}{\omega^{p/q}(x)} \leq C x$ for every $x>0$ (we also use the definition of $\varpi$ in Assumption~\ref{ass:finallejd}). Also, there exists a constant $C'>0$ such that $\omega \mathbb{1}_{\left\{\varpi > M \omega\right\}} \leq \varpi/M \leq C' \times \mathrm{Id}/M$ and $1/p + 1/q = 1$, so we finally obtain
$$\left\langle \mu^{K}_{t}, \varpi \mathbb{1}_{\left\{\varpi > M \omega\right\}} \right\rangle \leq \dfrac{C' C^{1/p}}{M^{1/q}} \left\langle \mu^{K}_{t}, \mathrm{Id} \right\rangle.$$
Remark that we obtain the exact same bound, replacing $K$ with $*$, so that  
\begin{multline}
\mathbb{E}\left(\sup\limits_{t \in [0,T]} \left\langle \mu^{*}_{t}, \varpi \mathbb{1}_{\left\{\varpi > M \omega\right\}} \right\rangle  \right) + \sup\limits_{K \geq 1} \mathbb{E}\left(\sup\limits_{t \in [0,T]}\left\langle \mu^{K}_{t}, \varpi \mathbb{1}_{\left\{\varpi > M \omega\right\}} \right\rangle  \right)\\
\leq \dfrac{C' C^{1/p}}{M^{1/q}} \bigg[\mathbb{E}\left( \sup\limits_{t \in [0,T]} \left\langle \mu^{*}_{t}, \mathrm{Id} \right\rangle\right)
 + \sup\limits_{K \geq 1} \mathbb{E}\left( \sup\limits_{t \in [0,T]} \left\langle \mu^{K}_{t}, \mathrm{Id} \right\rangle \right)   \bigg],
\label{eq:majoholdzere}
\end{multline}
which converges to 0 when $M \longrightarrow + \infty$ thanks to Proposition~\ref{lemme:controlenp} with $p=1$ and Lemma~\ref{corr:plusdefonctions}. In the following, we thus fix $\varepsilon >0$ and $M_{0}>0$ such that \eqref{eq:majoholdzere}$<\varepsilon/2$. The function $x >0 \mapsto \varpi(x) \mathbb{1}_{\left\{\varpi \leq M_{0} \omega\right\}}$ is positive and dominated by $M_{0} \omega$. Hence, by the almost sure convergence of $\bigg(\left(\mu^{K}_{t},R^{K}_{t}\right)_{t \in [0,T]}\bigg)_{K \in \mathbb{N}^{*}}$ towards $(\mu^{*}_{t},R^{*}_{t})_{t \in [0,T]}$ in $\mathbb{D}([0,T], (\mathcal{M}_{\omega}(\mathbb{R}^{*}_{+}),\mathrm{w}) \times [0,R_{\max}])$, we obtain
$$ \mathbb{E}\left(\sup\limits_{t \in [0,T]} |\left\langle \mu^{K}_{t}-\mu^{*}_{t}, \varpi \mathbb{1}_{\left\{\varpi \leq M_{0} \omega\right\}} \right\rangle | \right) \xrightarrow[K \longrightarrow + \infty]{} 0,$$
so that there exists by \eqref{eq:hodleriscmierfb} some $K_{0} \geq 1$ such that for $K \geq K_{0}$,
$$\mathbb{E}\left(\sup\limits_{t \in [0,T]} \left|\langle \mu^{K}_{t}-\mu^{*}_{t}, \varpi \rangle \right| \right) \leq \varepsilon,$$
which ends the proof since this is valid for every $\varepsilon>0$.
\end{proof}

In the following, we write $\mathcal{C}_{\varpi+\omega}([0,T] \times \mathbb{R}^{*}_{+})$ for the set of continous functions on $[0,T] \times \mathbb{R}^{*}_{+}$ such that
$$\exists C>0, \forall s \in [0,T], \forall x>0 \quad |\varphi_{s}(x)| \leq C (\varpi(x) + \omega(x)).$$
The following lemma present classical results that follows from Proposition~\ref{prop:holderproper} and holds true in a general context (see Problem 26.~p.153 in \cite{ek_2005} or Proposition 2.4 p.303 in \cite{Jacod1987LimitTF}). The reader may also refer to Lemmas II.4.3 and II.4.4 in \cite{broduthesis} for a proof in a slightly different context.
 
\begin{lemme}
Let $t \in [0,T]$, $\varphi \in \mathcal{C}_{\varpi+\omega}([0,T] \times \mathbb{R}^{*}_{+})$ and $\phi$ continuous on $[0,R_{\max}]$, then
\begin{align*}
\sup_{t \in [0,T]} | \langle \mu^{K}_{t} - \mu^{*}_{t}, \varphi_{s} \rangle | \xrightarrow[K \longrightarrow + \infty]{}  0,
\end{align*}
and
\begin{align*}
\displaystyle{\int_{0}^{t}} \phi(R^{K}_{s}) \langle \mu^{K}_{s}, \varphi_{s} \rangle \mathrm{d}s \xrightarrow[K \longrightarrow + \infty]{} \displaystyle{\int_{0}^{t}} \phi(R^{*}_{s}) \langle \mu^{*}_{s}, \varphi_{s} \rangle \mathrm{d}s.
\end{align*}
\label{lemme:continuityphi}
\end{lemme}

In the following, for every $\varphi \in \mathcal{C}^{1,1}_{\omega,T}(\mathbb{R}^{+} \times \mathbb{R}^{*}_{+})$, for every $\nu \in \mathbb{D}([0,T],(\mathcal{M}_{\omega}(\mathbb{R}^{*}_{+}),\mathrm{w}))$, and for every $t \in [0,T]$, we define
     \begin{multline*}
       \Psi_{t}(\nu) := \langle \nu_{t}, \varphi_{t} \rangle - \langle \nu_{0}, \varphi_{0} \rangle - \displaystyle{\int_{0}^{t} \int_{\mathbb{R}^{*}_{+}} \bigg( \Phi_{s}(R_{s}^{\nu},x) }   \\
     + b(x)(\varphi_{s}(x_{0}) + \varphi_{s}(x-x_{0}) - \varphi_{s}(x)) - d(x)\varphi_{s}(x) \bigg) \nu_{s}(\mathrm{d}x) \mathrm{d}s,
     \end{multline*}
with $\Phi$ associated to $\varphi$ as in \eqref{eq:Phidef}, and $R^{\nu}$ defined by $R^{\nu}_{0} = R_{0}$ and for $t \in [0,T]$,
\begin{align}
\dfrac{\mathrm{d}R_{t}^{\nu}}{\mathrm{d}t} = \rho(R_{t}^{\nu},\nu_{t}).
\label{eq:ressourceùiu}
\end{align}  
We are now ready to state the main result of this section.

\begin{prop}
Under the \hyperlink{ren}{renormalized setting}, let $T \geq 0$ and $\left((\mu^{K}_{t},R^{K}_{t})_{t \in [0,T]}\right)_{K \in \mathbb{N}^{*}}$ be a subsequence of our renormalized sequence of processes, converging in law towards $(\mu^{*}_{t},R^{*}_{t})_{t \in [0,T]}$ in the space $\mathbb{D}([0,T], (\mathcal{M}_{\omega}(\mathbb{R}^{*}_{+}),\mathrm{w}) \times [0,R_{\max}])$. Then, almost surely, for every $t \in [0,T]$ and $\varphi \in \mathcal{C}^{1,1}_{\omega,T}(\mathbb{R}^{+} \times \mathbb{R}^{*}_{+})$, 
$$\Psi_{t}(\mu^{*}) = 0.$$
Furthermore, $R^{*}=R^{\mu^{*}}$ almost surely, with $R^{\mu^{*}}$ defined as in \eqref{eq:ressourceùiu}. We thus conclude that almost surely, $(\mu^{*}_{t},R^{*}_{t})_{t \in [0,T]}$ verifies Equations~\eqref{eq:resslim} and \eqref{eq:indivlim}.
\label{prop:identification}
\end{prop}
\begin{proof}
Let us fix $\varphi \in \mathcal{C}^{1,1}_{\omega,T}(\mathbb{R}^{+} \times \mathbb{R}^{*}_{+})$, $K\in \mathbb{N}^{*}$. We verify that we can always define a function $\tilde{\varphi} \in \mathcal{C}^{1,1}_{\omega}(\mathbb{R}^{+} \times \mathbb{R}^{*}_{+})$ with $\tilde{\varphi} \equiv \varphi$ on $[0,T] \times \mathbb{R}^{*}_{+}$. The results of Proposition~\ref{prop:martingalerenormalisee} are then valid for $\tilde{\varphi}_{t}$ for $t \geq 0$, hence for $\varphi_{t}$ for any $t \in [0,T]$. In the following, we fix $t \in [0,T]$ and to simplify the notations, we write $\varphi$ instead of $\tilde{\varphi}$. We divide the proof in three steps. First, we reduce the problem of showing that $\Psi_{t}(\mu^{*}) = 0$ almost surely to proving that $(\Psi_{t}(\mu^{K}))_{K \in \mathbb{N}^{*}}$ converges in law towards $\Psi_{t}(\mu^{*})$ in $\mathbb{R}$. Then, we prove the convergence in law of $(\Psi_{t}(\mu^{K}))_{K \in \mathbb{N}^{*}}$ towards a randow variable defined in the same manner as $\Psi_{t}(\mu^{*})$, but replacing $R^{\mu^{*}}$ with $R^{*}$. Finally, we prove that almost surely, $R^{*} = R^{\mu^{*}}$, which concludes the proof.
\bigbreak
\textbf{Step 1: Reduction of the problem of showing that $\Psi_{t}(\mu^{*}) = 0$ almost surely}
\\\\
We consider $\Psi_{t}(\mu^{K})=\heartsuit^{K}_{\varphi,t}$ defined in Proposition~\ref{prop:martingalerenormalisee}. By Proposition~\ref{prop:martingalerenormalisee} applied to $\varphi$, by assumptions on $\varphi$ and by Lemma~\ref{lemme:equivalence}, there exists a constant $C>0$ such that
\begin{align*}
    \mathbb{E}(|\Psi_{t}(\mu^{K})|^{2}) & = \mathbb{E}(|\heartsuit_{\varphi,t}^{K}|^{2})  = \mathbb{E}(\langle \heartsuit_{\varphi}^{K} \rangle_{t})  \leq \dfrac{C}{K} \sup_{K \in \mathbb{N}^{*}} \mathbb{E}\left(\sup_{t \in [0,T]} (E^{K}_{t}+ N^{K}_{t} + \Omega^{K}_{t}) \right),
\end{align*}
which entails with Proposition~\ref{lemme:controlenp} with $p=1$ that $\mathbb{E}(|\Psi_{t}(\mu^{K})|) \xrightarrow[K \rightarrow + \infty]{} 0$, because $\mathbb{E}(|\Psi_{t}(\mu^{K})|)^{2} \leq \mathbb{E}(|\Psi_{t}(\mu^{K})|^{2})$ by Jensen inequality. Hence, it suffices to show that $\mathbb{E}(|\Psi_{t}(\mu^{K})|) \xrightarrow[K \rightarrow + \infty]{} \mathbb{E}(|\Psi_{t}(\mu^{*})|)$ to conclude that $\Psi_{t}(\mu^{*}) = 0$ almost surely. By Corollary~\ref{corr:uniformly}, the family of square-integrable martingales $\left(\Psi_{t}(\mu^{K})\right)_{K \in \mathbb{N}^{*}}$ is uniformly integrable. Then, by Proposition 2.3 p.494 in \cite{ek_2005}, it suffices to show that $(\Psi_{t}(\mu^{K}))_{K \in \mathbb{N}^{*}}$ converges in law towards $\Psi_{t}(\mu^{*})$ in $\mathbb{R}$. 
%
\\\\
\textbf{Step 2: Convergence in law of $(\Psi_{t}(\mu^{K}))_{K \in \mathbb{N}^{*}}$ towards $\Psi_{t}(\mu^{*})$}
\\\\
By definition, $R^{\mu^{K}}$ coincide with $R^{K}$ and we write
\begin{multline*}
\Psi_{t}(\mu^{K}) =  \langle \mu^{K}_{t}, \varphi_{t} \rangle -  \langle \mu^{K}_{0}, \varphi_{0} \rangle -  \displaystyle{\int_{0}^{t}} \phi\left(R^{K}_{s}\right) \langle \mu^{K}_{s}, \psi(.) \partial_{2}\varphi(s,.) \rangle \mathrm{d}s \\
    + \displaystyle{\int_{0}^{t}}\bigg\langle \mu^{K}_{s}, d(x)\varphi_{s}(x)  - b(x)\bigg( \varphi_{s}(x_{0})+\varphi_{s}(x-x_{0})-\varphi_{s}(x)\bigg) -\partial_{1}\varphi(s,x) + \ell(x)\partial_{2}\varphi(s,x) \bigg\rangle  \mathrm{d}s,
\end{multline*}
with $\phi$, $\psi$ and $\ell$ defined in Section~\ref{subsec:dyn}, related to respectively the functional response, the growth rate and the metabolic rate. Then, for $s \in [0,T]$ and $x>0$, let us define
$$ \mho_{\varphi,s}(x) := d(x)\varphi_{s}(x)  - b(x)\bigg( \varphi_{s}(x_{0})+\varphi_{s}(x-x_{0})-\varphi_{s}(x)\bigg) -\partial_{1}\varphi(s,x) + \ell(x)\partial_{2}\varphi(s,x), $$
which is well-defined because $b \equiv 0$ on $(0,x_{0})$, and we obtain
\begin{align}
\Psi_{t}(\mu^{K}) = \langle \mu^{K}_{t}, \varphi_{t} \rangle -  \langle \mu^{K}_{0}, \varphi_{0} \rangle + \displaystyle{\int_{0}^{t}} \langle \mu^{K}_{s}, \mho_{\varphi,s} \rangle \mathrm{d}s - \displaystyle{\int_{0}^{t}}\phi\left(R^{K}_{s}\right) \langle \mu^{K}_{s}, \psi(.)\partial_{2}\varphi(s,.) \rangle \mathrm{d}s,
\label{eq:expresioosnpsi}
\end{align}
First, by Corollary~\ref{lemme:continuity}, the limit $(\mu^{*}_{t})_{t \in [0,T]}$ is in $\mathcal{C}([0,T], (\mathcal{M}_{\omega}(\mathbb{R}^{*}_{+}),\mathrm{w}))$. Hence, from Theorem 7.8 p.131 in \cite{ek_2005}, and because $\left((\mu^{K}_{t})_{t \in [0,T]}\right)_{K \in \mathbb{N}^{*}}$ converges in law towards $(\mu^{*}_{t})_{t \in [0,T]}$ in the space $\mathbb{D}([0,T], (\mathcal{M}_{\omega}(\mathbb{R}^{*}_{+}),\mathrm{w}))$, we have that for every $t \in [0,T]$, the marginal distribution $(\mu_{t}^{K})_{K \in \mathbb{N}^{*}}$ converges in law towards $\delta_{\mu^{*}_{t}}$ in $(\mathcal{M}_{\omega}(\mathbb{R}^{*}_{+}),\mathrm{w})$. By definition of the $\omega$-weak topology (which makes every $\mu \in \mathcal{M}_{\omega}(\mathbb{R}^{*}_{+}) \mapsto \langle \mu, \varphi \rangle$ continuous if $\varphi \in \mathfrak{B}_{\omega}(\mathbb{R}^{*}_{+})$ is continuous) and assumption on $\varphi$, we obtain that
\begin{center}
$\left(\langle \mu^{K}_{t}, \varphi_{t} \rangle -  \langle \mu^{K}_{0}, \varphi_{0} \rangle\right)_{K \in \mathbb{N}^{*}}$ converges in law towards $\langle \mu^{*}_{t}, \varphi_{t} \rangle -  \langle \mu^{*}_{0}, \varphi_{0} \rangle$ in $\mathbb{R}$. 
\end{center} 
Then, with the assumption $\varphi \in \mathcal{C}^{1,1}_{\omega,T}(\mathbb{R}^{+} \times \mathbb{R}^{*}_{+})$, Assumption~\ref{hyp:poidsomega} and the additional Assumption~\ref{ass:finallejd}, we verify that $\mho_{\varphi} \in \mathcal{C}_{\varpi+\omega}([0,T] \times \mathbb{R}^{*}_{+})$ (in particular, we use the fact that $\partial_{2}\varphi(s,.)$ is bounded, uniformly on $s \in [0,T]$, to control the term $\varphi_{s}(x-x_{0})- \varphi_{s}(x)$; and also the fact that $d \varphi_{s}/\omega$ is bounded, uniformly on $s \in [0,T]$). Hence, by Lemma~\ref{lemme:continuityphi} applied to $\mho_{\varphi}$ and $\phi \equiv 1$, we obtain that
\begin{center}
$\left(\displaystyle{\int_{0}^{t}} \langle \mu^{K}_{s}, \mho_{\varphi,s} \rangle \mathrm{d}s\right)_{K \in \mathbb{N}^{*}}$ converges in law towards $\displaystyle{\int_{0}^{t}} \langle \mu^{*}_{s}, \mho_{\varphi,s} \rangle \mathrm{d}s$ in $\mathbb{R}$.
\end{center} 
Finally, by assumption on $\varphi$ and Assumption~\ref{ass:finallejd}, we verify that $(t,x) \mapsto \psi(x) \partial_{2}\varphi(t,x) \in \mathcal{C}_{\varpi+\omega}([0,T] \times \mathbb{R}^{*}_{+})$, and the function $\phi$ is continuous on $[0,R_{\max}]$ by assumption (see again Section~\ref{subsec:dyn} for the definition of the functions $\psi$ and $\phi$). By Lemma~\ref{lemme:continuityphi}, we obtain that 
\begin{center}
$\displaystyle{\int_{0}^{t}} \phi\left(R^{K}_{s}\right) \langle \mu^{K}_{s}, \psi(.) \partial_{2}\varphi(s,.) \rangle \mathrm{d}s$ converges in law towards $\displaystyle{\int_{0}^{t}} \phi\left(R^{*}_{s}\right) \langle \mu^{*}_{s}, \psi(.) \partial_{2}\varphi(s,.) \rangle \mathrm{d}s$.
\end{center}
Hence, to conclude the whole proof, it suffices to show that $(R^{*}_{t})_{t \in [0,T]}$ and $(R^{\mu^{*}}_{t})_{t \in [0,T]}$ are equal almost surely.
\\\\
\textbf{Step 3: Proof of $R^{*} = R^{\mu^{*}}$ almost surely}
\\\\
For every $\nu \in \mathbb{D}([0,T], (\mathcal{M}_{\omega}(\mathbb{R}^{*}_{+}),\mathrm{w})$, $R^{\nu}$ is entirely determined by $\nu$ with \eqref{eq:ressourceùiu}. Indeed, there exists a function $\mathcal{R} : \mathbb{D}([0,T], (\mathcal{M}_{\omega}(\mathbb{R}^{*}_{+}),\mathrm{w}) \mapsto \mathcal{C}([0,T], [0,R_{\max}])$ such that $\mathcal{R}(\nu) = R^{\nu}$, and it is precisely given for $t \in [0,T]$ by the functional equation
$$\mathcal{R}(\nu)_{t} = R_{0} + \displaystyle{\int_{0}^{t}} \bigg( \varsigma(\mathcal{R}(\nu)_{s}) - \chi \langle  \nu_{s}, f(.,\mathcal{R}(\nu)_{s}) \rangle \bigg) \mathrm{d}s.$$
Let us first prove that $\mathcal{R}$ is continuous at every $\nu := (\nu_{s})_{s \in [0,T]} \in \mathcal{C}([0,T],(\mathcal{M}_{\omega}(\mathbb{R}^{*}_{+}),\mathrm{w}))$. We fix such $\nu$ and a sequence $((\nu^{K}_{s})_{s \in [0,T]})_{K \geq 1}$ that converges towards $\nu$ in the space $\mathbb{D}([0,T],(\mathcal{M}_{\omega}(\mathbb{R}^{*}_{+}),\mathrm{w}))$, and we aim to show that $ \sup\limits_{t \in [0,T]} \left| \mathcal{R}(\nu^{K})_{t} - \mathcal{R}(\nu)_{t} \right| \xrightarrow[K \rightarrow + \infty]{} 0.$ For $t \in [0,T]$ and $K \geq 1$, we compute
\begin{align*}
    & |\mathcal{R}(\nu^{K})_{t}-\mathcal{R}(\nu)_{t}|  \\ = & \hspace{0.1 cm} \bigg|\displaystyle{\int_{0}^{t}} \bigg[\varsigma(\mathcal{R}(\nu^{K})_{s}) -\varsigma(\mathcal{R}(\nu)_{s}) - \chi \left( \phi(\mathcal{R}(\nu^{K})_{s})\langle  \nu^{K}_{s}, \psi \rangle -\phi(\mathcal{R}(\nu)_{s})\langle  \nu_{s}, \psi \rangle \right) \bigg] \mathrm{d}s \bigg| \\
    \leq & \hspace{0.1 cm} ||\varsigma'||_{\infty,[0,R_{\max}]}\displaystyle{\int_{0}^{t}} |\mathcal{R}(\nu^{K})_{s} -\mathcal{R}(\nu)_{s}| \mathrm{d}s + \chi \displaystyle{\int_{0}^{t}} \left| \phi(\mathcal{R}(\nu^{K})_{s}) - \phi(\mathcal{R}(\nu)_{s}) \right| \langle  \nu^{K}_{s}, \psi \rangle \mathrm{d}s \\
    &  \hspace{5cm} + \chi ||\phi||_{\infty,[0,R_{\max}]} \displaystyle{\int_{0}^{t}} \left| \langle  \nu^{K}_{s}-\nu_{s}, \psi \rangle \right| \mathrm{d}s \\
   \leq & \hspace{0.1 cm} \bigg(||\varsigma'||_{\infty,[0,R_{\max}]} + \chi k T \sup\limits_{K \geq 1} \sup\limits_{u \in [0,T]} \left(E^{K}_{u} + N^{K}_{u} +\Omega_{u}^{K} \right) \bigg)\displaystyle{\int_{0}^{T}} \sup\limits_{\tau \in [0,s]} |\mathcal{R}(\nu^{K})_{\tau} -\mathcal{R}(\nu)_{\tau}| \mathrm{d}s \\
    &  \hspace{5cm} + \chi ||\phi||_{\infty,[0,R_{\max}]} T \sup_{s \in [0,T]} \left| \langle  \nu^{K}_{s}-\nu_{s}, \psi \rangle \right|,
\end{align*} 
where we used the fact that $\phi$ is Lipschitz continuous (see \eqref{eq:lipshitzphi}), that $\psi \leq \overline{g}$ and Assumption~\ref{hyp:poidsomega}. Then, by Proposition~\ref{lemme:controlenp} with $p=1$, the previous upper bound is almost surely finite and independent of $t$ so it is an upper bound for $ \sup\limits_{s \in [0,t]} \left| \mathcal{R}(\nu^{K})_{s} - \mathcal{R}(\nu)_{s} \right|$ and by Gronwall's lemma, there exists a constant $C'>0$ such that
\begin{align*}
\sup\limits_{t \in [0,T]} |\mathcal{R}(\nu^{K})_{t}-\mathcal{R}(\nu)_{t}| & \leq \chi ||\phi||_{\infty,[0,R_{\max}]} T \sup_{s \in [0,T]} \left| \langle  \nu^{K}_{s}-\nu_{s}, \psi \rangle \right| e^{C'T}.
\end{align*}
By Assumption~\ref{ass:finallejd} and because $\psi \leq \overline{g}$, then $(t,x) \mapsto \psi(x) \in \mathcal{C}_{\varpi+\omega}([0,T] \times \mathbb{R}^{*}_{+})$. By Lemma~\ref{lemme:continuityphi} (replacing $\mu^{K}$ with $\nu^{K}$ and $\mu^{*}$ with $\nu$ since they verify the same assumptions), this ends the proof of the fact that $\mathcal{R}$ is continuous at every $\nu$ continuous. Now, let us get back to the proof that $R^{*}=R^{\mu^{*}}$ almost surely. We consider $(\varphi_{n})_{n \in \mathbb{N}}$ a sequence of $\mathcal{C}^{\infty}_{c}(\mathbb{R})$ positive functions that converges pointwise towards $\mathbb{1}_{\{0\}}$ and such that $\varphi_{n}(0)=1$ for every $n \geq 0$. By the mapping theorem and convergence in law of $(\mu^{K},R^{K})_{K \geq 1}$ towards $(\mu^{*},R^{*})$, we have that for every $n \in \mathbb{N}$, for every $t \in [0,T]$,
\begin{align}
\lim\limits_{K \rightarrow + \infty} \mathbb{E}\left(\varphi_{n}\left(R^{K}_{t}-R^{\mu^{K}}_{t} \right)\right) = \lim\limits_{K \rightarrow + \infty} \mathbb{E}\left(\varphi_{n}\left(R^{K}_{t}-\mathcal{R}\left(\mu^{K}_{t}\right) \right)\right) =  \mathbb{E}\left(\varphi_{n}\left(R^{*}_{t}-R^{\mu^{*}}_{t} \right)\right).
\label{eq:aireaieairéa}
\end{align}
On the one hand, for every $K \geq 1$, we have by definition that $R^{K}= R^{\mu^{K}}$ so that the left-most term in \eqref{eq:aireaieairéa} is constant equal to 1. On the other hand, taking the limit $n \rightarrow + \infty$ in the right-most term  gives $\mathbb{E}\left(\mathbb{1}_{\{0\}}\left(R^{*}_{t}-R^{\mu^{*}}_{t} \right)\right)$ (dominated convergence holds trivially because everything is bounded), so we conclude that for every $t \in [0,T]$, we almost surely have $R^{*}_{t} = R^{\mu^{*}}_{t}$, which concludes thanks to the continuity of the considered functions.
\end{proof}


\textbf{Acknowledgments.}
The author would like to thank the members of the Inria team SIMBA (Statistical Inference and Modeling for Biological Applications) for valuable discussions around this work, and especially among them Nicolas Champagnat and Coralie Fritsch. Also, Sylvain Billiard should be mentioned for his active participation to the biological motivation of this work. 
\\\\
\textbf{Funding.}
This work was partially supported by the Chaire ``Mod\'elisation Math\'ematique et Biodiversit\'e'' of VEOLIA Environment, \'Ecole Polytechnique, Mus\'eum National d'Histoire Naturelle and Fondation X, and by the European Union (ERC, SINGER, 101054787). Views and opinions expressed are however those of the author(s) only and do not necessarily reflect those of the European Union or the European Research Council. Neither the European Union nor the granting authority can be held responsible for them.

\appendix
\appendixpage

\section{Proofs of intermediate results}
\label{app:inter}

\subsection{Proof of Proposition~\ref{prop:etapeun}}
\label{app:prop27}

In all the following proof, we suppose by contradiction that $\mathbb{P}(t_{\exp}(\mu_{0},R_{0}) \leq J_{1} < + \infty) >0$, and work under this event. Importantly, this implies that until time $t_{\exp}(\mu_{0},R_{0})$, there are no random birth or death jumps in the population, so the process $(\mu_{t})_{t}$ is well-defined and deterministic on $[0,t_{\exp}(\mu_{0},R_{0}))$. Thus, under the event $\{t_{\exp}(\mu_{0},R_{0}) \leq J_{1} < + \infty\}$, we have $V_{0} \neq \varnothing$ (otherwise $t_{\exp}(\mu_{0},R_{0}) = + \infty$) and for every $u \in V_{0}$ and $s < t_{\exp}(\mu_{0},R_{0})$, we have $\xi^{u}_{s} =X^{u}_{s}(\Xi_{0},R_{0})$ with the notations of Section~\ref{sec:construction}.
We assess that one of the two following situations occurs:
\\\\
$\begin{array}{cl}
\mathrm{(i)} & \exists u \in V_{0}, \quad \xi_{s}^{u} \xrightarrow[s \rightarrow t_{\exp}(\mu_{0},R_{0})]{} + \infty, \\
\mathrm{(ii)} & \exists u \in V_{0}, \quad \xi_{s}^{u} \xrightarrow[s \rightarrow t_{\exp}(\mu_{0},R_{0})]{} 0.
\end{array}$
\\\\
Indeed, if (i) and (ii) are not verified, then by definition, we would necessarily have $t_{\exp}(\mu_{0},R_{0}) = + \infty$.
\\\\
First, suppose that there exists $u \in V_{0}$ such that (i) is verified. We let the reader check that the same decomposition as in Lemma~\ref{lemme:decompocun} holds true for every $s \in [0,t_{\exp}(\mu_{0},R_{0}))$, so that we can use \eqref{eq:controlressource} for every such $s$. In that case, we would have
$$ \forall s < t_{\exp}(\mu_{0},R_{0}), \quad 
\xi_{s}^{u} \leq E_{s} \leq R_{0}+E_{0}+s||\varsigma||_{\infty,[0,R_{\max}]} \leq R_{0}+E_{0}+ t_{\exp}(\mu_{0},R_{0})||\varsigma||_{\infty,[0,R_{\max}]}  < + \infty, $$
which contradicts (i) by letting $s \rightarrow t_{\exp}(\mu_{0},R_{0}) $.
\\\\
Else, suppose that $u \in V_{0}$ is such that (ii) is verified. By construction of our process with Poisson point measures, under the event $\{t_{\exp}(\mu_{0},R_{0}) \leq J_{1} < + \infty \}$, we have on the one hand that
\begin{align}
\displaystyle{\int_{0}^{t_{\exp}(\mu_{0},R_{0})}} \left( \sum\limits_{u \in V_{0}} (b+d)(\xi^{u}_{s}) \right) \mathrm{d}s< + \infty.
\label{eq:sautpoisson}
\end{align}
Indeed, before time $t_{\exp}(\mu_{0},R_{0})$, individual energies $\xi^{u}_{s}$ follow the deterministic flows $X^{u}_{.}(\Xi_{0},R_{0})$, so the previous integral is a deterministic quantity. If the previous integral was infinite, the first jump time $J_{1}$ given by the Poisson point measures $\mathcal{N}$ and $\mathcal{N}'$ would follow an inhomogeneous exponential law, and we would have 
$$\mathbb{P}(J_{1}<t_{\exp}(\mu_{0},R_{0}))=1- \exp\left(\displaystyle{\int_{0}^{t_{\exp}(\mu_{0},R_{0})}} \left( \sum\limits_{u \in V_{0}} (b+d)(\xi^{u}_{s}) \right) \mathrm{d}s\right)=1.$$
On the other hand, we write
\begin{align*}
\displaystyle{\int_{0}^{t_{\exp}(\mu_{0},R_{0})}} d(\xi^{u}_{s}) \mathrm{d}s \geq  \displaystyle{\int_{0}^{t_{\exp}(\mu_{0},R_{0})}} d(\xi^{u}_{s})\mathbb{1}_{\{g(\xi^{u}_{s},R_{s})<0 \}} \mathrm{d}s.
\end{align*} 
By (ii) and considering \eqref{eq:indivenergy}, there exists a sequence $(O_{n})_{n \in \mathbb{N}}$ of disjoint open intervals such that 
\begin{align}
\{s \in (0,t_{\exp}(\mu_{0},R_{0})), \, g(\xi^{u}_{s},R_{s})<0\} = \bigsqcup\limits_{n\in \mathbb{N}} O_{n} \quad \mathrm{and} \quad 
(0, \xi^{u}_{0}) \subseteq \{ \xi^{u}_{s}, \; s \in \bigsqcup_{n\in \mathbb{N}} O_{n} \}.
\label{eq:minorassun}
\end{align}
For every $n\in \mathbb{N}$, from \eqref{eq:indivenergy}, $ s \mapsto \xi^{u}_{s}$ is a decreasing bijection from $O_{n}$ to $\xi^{u}(O_{n})$, and we write $\tilde{\xi}^{u}$ for its inverse function. We perform the change of variables $w=\xi^{u}_{s}$ on each interval $O_{n}$ to obtain
$$ \displaystyle{\int_{0}^{t_{\exp}(\mu_{0},R_{0})}} d(\xi^{u}_{s}) \mathrm{d}s \geq \sum\limits_{n \in \mathbb{N}}  \displaystyle{\int_{ O_{n}}} d(\xi^{u}_{s}) \mathrm{d}s = \sum\limits_{n \in \mathbb{N}} \displaystyle{\int_{ \xi^{u}(O_{n})}} \dfrac{d(w)}{-g(w,R_{\tilde{\xi}^{u}_{w}})} \mathrm{d}w \geq \displaystyle{\int_{0}^{\xi^{u}_{0}}} \dfrac{d}{\ell}(w) \mathrm{d}w, $$
because of \eqref{eq:minorassun}, and $R \mapsto g(x,R)$ is non-decreasing for every $x>0$. By \eqref{eq:sautpoisson}, the left-most integral is finite, which contradicts the fact that the right-most integral is infinite with Assumption~\ref{hyp:probamortel}, and this ends the proof.

\subsection{Proof of Proposition~\ref{prop:controlpop}}
\label{app:prop212}

Let $M>0$ and $T \geq 0$. We apply Lemma~\ref{lemme:decompocun} to $\varphi : (t,x) \mapsto 1+ \omega(x)$ and $F : (r,x) \mapsto x$, and use Proposition~\ref{prop:controlebiomasse} to obtain for every $0 \leq t < T \wedge \tau_{M} \wedge J_{\infty}$,
\begin{align*}
        E_{t}+ N_{t} + \Omega_{t} \leq & \hspace{0.1 cm} R_{0} + E_{0}+t||\varsigma||_{\infty,[0,R_{\max}]} +  N_{0} + \Omega_{0} +\displaystyle{\int_{0}^{t}} \langle
        \mu_{s}, g(.,R_{s})\omega' \rangle \mathrm{d}s
        \\ & +  \displaystyle{\int_{0}^{t}\int_{\mathcal{U} \times \mathbb{R}^{*}_{+} } \mathbb{1}_{\{u \in V_{s-} \}} \mathbb{1}_{\{h \leq b(\xi^{u}_{s-})\}}} \\
        & \hspace{3cm} \bigg(1+ \omega(x_{0})+\omega(\xi^{u}_{s-}-x_{0})-\omega(\xi^{u}_{s-}) \bigg)\mathcal{N}(\mathrm{d}s,\mathrm{d}u,\mathrm{d}h) \\
     & - 
    \displaystyle{\int_{0}^{t}\int_{\mathcal{U} \times \mathbb{R}^{*}_{+} } \mathbb{1}_{\{ u \in V_{s-}\}} \mathbb{1}_{\{h \leq d(\xi^{u}_{s-})\}}\bigg(1+ \omega(\xi^{u}_{s-}) \bigg)}  \mathcal{N}'(\mathrm{d}s,\mathrm{d}u,\mathrm{d}h) \\
    \leq & \hspace{0.1 cm} R_{\max} + E_{0} +  N_{0} + \Omega_{0}+ T||\varsigma||_{\infty,[0,R_{\max}]} +\displaystyle{\int_{0}^{T}} \langle
        \mu_{s}, \overline{g}\omega' \rangle \mathbb{1}_{\{s < \tau_{M} \wedge J_{\infty}\}} \mathrm{d}s
        \\ & +  \displaystyle{\int_{0}^{T}\int_{\mathcal{U} \times \mathbb{R}^{*}_{+} } \mathbb{1}_{\{s < \tau_{M} \wedge J_{\infty}\}} \mathbb{1}_{\{u \in V_{s-} \}} \mathbb{1}_{\{h \leq b(\xi^{u}_{s-})\}}} \\
     & \hspace{3 cm }\bigg(1+ |\omega(x_{0})+\omega(\xi^{u}_{s-}-x_{0})-\omega(\xi^{u}_{s-})| \bigg)\mathcal{N}(\mathrm{d}s,\mathrm{d}u,\mathrm{d}h).
    \end{align*}

We can take the supremum of the left-hand side over $t \in [0,T \wedge \tau_{M} \wedge J_{\infty})$. Then, we take expectations and apply Fubini theorem, which is valid since all integrands are positive. Also, all the expectations are finite thanks to the first and second points of Assumption~\ref{hyp:poidsomega} and by definition of $\tau_{M}$, so we have a true semi-martingale decomposition of the integrated term against the Poisson point measure $\mathcal{N}$. We obtain
\begin{align*}
\mathbb{E}\bigg(\sup_{t \in [0,T \wedge \tau_{M} \wedge J_{\infty})} (E_{t} + N_{t}+ \Omega_{t})\bigg)  \leq & \hspace{0.1 cm} R_{\max} + \mathbb{E}(E_{0} +  N_{0} + \Omega_{0})+ T||\varsigma||_{\infty,[0,R_{\max}]} \\
&  +(C_{g}+C_{b})\displaystyle{\int_{0}^{T}} \mathbb{E}\left((E_{s}+N_{s}+\Omega_{s}) \mathbb{1}_{\{s < \tau_{M} \wedge J_{\infty}\}}\right) \mathrm{d}s \\
\leq & \hspace{0.1 cm} R_{\max} + \mathbb{E}(E_{0} +  N_{0} + \Omega_{0})+ T||\varsigma||_{\infty,[0,R_{\max}]} \\
&  +(C_{g}+C_{b})\displaystyle{\int_{0}^{T}} \mathbb{E}\left(\sup_{\tau \in [0,s \wedge \tau_{M} \wedge J_{\infty})} (E_{\tau} + N_{\tau}+ \Omega_{\tau})\right) \mathrm{d}s
\end{align*}     

By Gronwall lemma, we then have
\begin{multline}
\mathbb{E}\left(\sup_{t \in [0,T \wedge \tau_{M} \wedge J_{\infty})} (E_{t}+N_{t}+\Omega_{t})\right)  \\
 \leq \bigg(R_{\max} + \mathbb{E}(E_{0} +  N_{0} + \Omega_{0})+ T||\varsigma||_{\infty,[0,R_{\max}]}\bigg)e^{(C_{g}+ C_{b})T}.
\label{eq:gronwalloset}
\end{multline} 
Now, as $(\tau_{M})_{M \in \mathbb{N}^{*}}$ is a non-decreasing sequence, to prove \eqref{eq:taum}, it suffices to show that for all $T \geq 0$, $\mathbb{P}(\tau_{M} \leq T) \xrightarrow[M \rightarrow + \infty]{} 0$. Because $\mathbb{E}(E_{0} +  N_{0} + \Omega_{0})<+\infty$, this follows from \eqref{eq:gronwalloset} and
\begin{align*}
\mathbb{P}(\tau_{M} \leq T) & \leq \mathbb{P}\left(\sup_{t \in [0,T \wedge \tau_{M} \wedge J_{\infty})} (E_{t} + N_{t}+\Omega_{t}) \geq M\right) \\
& \leq \dfrac{1}{M}\mathbb{E}\left(\sup_{t \in [0,T \wedge \tau_{M} \wedge J_{\infty})} (E_{t} + N_{t}+\Omega_{t})\right). 
\end{align*}
Finally, to obtain \eqref{eq:controleomega}, we use \eqref{eq:taum} and apply Fatou lemma to obtain
\begin{align*}
\mathbb{E}\left(\sup_{t \in [0,T \wedge J_{\infty})} (E_{t}+N_{t}+\Omega_{t})\right) \leq \underset{M \rightarrow + \infty}{\mathrm{lim}} \mathbb{E}\left(\sup_{t \in [0,T \wedge \tau_{M}\wedge J_{\infty})} (E_{t}+N_{t}+\Omega_{t})\right),
\end{align*}  
which ends the proof by \eqref{eq:gronwalloset}.

\subsection{Proof of Proposition~\ref{lemme:controlenp}}
\label{app:lemmcontr}

\begin{lemme}
Let $\varphi \in \mathcal{C}^{1,1}(\mathbb{R}^{+} \times \mathbb{R}^{*}_{+})$, and $F \in \mathcal{C}^{1,1}([0,R_\mathrm{in}] \times \mathbb{R})$. Under Assumptions~\ref{hyp:probamortel} and \ref{hyp:poidsomega}, we have almost surely, for all $K \in \mathbb{N}^{*}$, $t \geq 0$,
\begin{align*}
F(R^{K}_{t},\langle \mu^{K}_{t}, \varphi_{t} \rangle) = & \hspace{0.1 cm} F(R_{0},\langle \mu^{K}_{0}, \varphi_{0} \rangle) \\
    & + \displaystyle{\int_{0}^{t}} \rho(\mu^{K}_{s},R^{K}_{s}) \partial_{1}F(R^{K}_{s},\langle \mu^{K}_{s}, \varphi_{s} \rangle) \\
    & \hspace{5cm} + \langle \mu^{K}_{s}, \Phi_{s}(R^{K}_{s},.) \rangle \partial_{2} F(R^{K}_{s},\langle \mu^{K}_{s}, \varphi_{s} \rangle) \mathrm{d}s \\
    & + \displaystyle{\int_{0}^{t}}\int_{\mathcal{U} \times \mathbb{R}^{*}_{+} } \mathbb{1}_{\{u \in V^{K}_{s-} \}} \mathbb{1}_{\{h \leq b\left(\xi^{u,K}_{s-}\right)\}} \\
    & \hspace{3 cm}  \bigg( F(R^{K}_{s},\langle \mu^{K}_{s-}+\frac{1}{K} \left( \delta_{x_{0}} + \delta_{\xi^{u,K}_{s-}-x_{0}}-\delta_{\xi^{u,K}_{s-}} \right), \varphi_{s} \rangle) \\
    & \hspace{6 cm}- F(R^{K}_{s}, \langle \mu^{K}_{s-}, \varphi_{s} \rangle) \bigg) \mathcal{N}(\mathrm{d}s,\mathrm{d}u,\mathrm{d}h) &
\end{align*}
\begin{align*}
     \hspace{3cm}& +
    \displaystyle{\int_{0}^{t}\int_{\mathcal{U} \times \mathbb{R}^{*}_{+} } \mathbb{1}_{\{u \in V^{K}_{s-} \}} \mathbb{1}_{\{h \leq d(\xi^{u,K}_{s-})\}} }\bigg( F(R^{K}_{s},\langle \mu^{K}_{s-}-\frac{1}{K}\delta_{\xi^{u,K}_{s-}}, \varphi_{s} \rangle) \\
    & \hspace{6cm}- F(R^{K}_{s}, \langle \mu^{K}_{s-}, \varphi_{s} \rangle) \bigg) \mathcal{N}'(\mathrm{d}s,\mathrm{d}u,\mathrm{d}h),
\end{align*}
with $\Phi$ associated to $\varphi$ as in \eqref{eq:Phidef}.
\label{lemme:decompok}
\end{lemme}

\begin{proof}
For $K \in \mathbb{N}^{*}$ and $t \geq 0$, it suffices to remark that $$F(R^{K}_{t},\langle
\mu^{K}_{t}, \varphi_{t} \rangle) = F\left(R^{K}_{t},\langle
\nu^{K}_{t}, \dfrac{1}{K}\varphi_{t} \rangle \right),$$
and to apply Lemma~\ref{lemme:decompocun}. Note in particular that we recover the function $\rho$ because we pick the appropriate renormalization of the parameter $\chi$ in \eqref{eq:regroupindiv}. The decomposition is almost surely valid for every $t \geq 0$ thanks to Corollary~\ref{corr:jinfini}.
\end{proof}

Now for the proof of Proposition~\ref{lemme:controlenp}, we fix $T \geq 0$, $K \geq 1$, and apply Lemma~\ref{lemme:decompok} to $F : (r,x) \mapsto x^{p}\mathbb{1}_{\{x > 0 \}}$ (which is non-negative, non-decreasing on the variable $x$ and $\mathcal{C}^{1,1}([0,R_{\max}] \times \mathbb{R})$ because $p \geq 1$) and $\varphi : (t,x) \mapsto 1+x+ \omega(x) $, to obtain for every $t \in [0, T]$,
\begin{align*}
\left(E_{t}^{K} +N_{t}^{K} + \Omega^{K}_{t}\right)^{p} \leq & \hspace{0.1 cm} \left(E_{0}^{K} +N_{0}^{K} + \Omega^{K}_{0}\right)^{p} + pC_{g} \displaystyle{\int_{0}^{T} \left(E_{s}^{K} +N_{s}^{K} + \Omega^{K}_{s}\right)^{p} \mathrm{d}s} \\
    & + \displaystyle{\int_{0}^{T}\int_{\mathcal{U} \times \mathbb{R}^{*}_{+} } \mathbb{1}_{\{u \in V^{K}_{s-} \}} \mathbb{1}_{\{h \leq b\left(\xi^{u,K}_{s-}\right)\}}} \\  
    & \hspace{2 cm}\bigg[ \bigg(E_{s-}^{K} +N_{s-}^{K} + \Omega^{K}_{s-} +\dfrac{1+\omega(x_{0})}{K}  \bigg)^{p} \\
    & \hspace{4.5cm} - \left(E_{s-}^{K} +N_{s-}^{K} + \Omega^{K}_{s-}\right)^{p} \bigg] \mathcal{N}(\mathrm{d}s,\mathrm{d}u,\mathrm{d}h).
\end{align*}
We used the fact that $\omega$ is positive and non-decreasing, the first point of Assumption~\ref{hyp:poidsomega} and the fact that $F$ is non-decreasing on the variable $x$ (so that the integral against $\mathcal{N}'$ is non-positive). For $M>0$, as in Proposition~\ref{prop:controlpop}, we define the stopping times
$$\tau^{K}_{M} := \inf \left\{ t \geq 0, \; E^{K}_{t} + N^{K}_{t} + \Omega^{K}_{t} \geq M \right\},$$
with the convention $\inf(\varnothing) = + \infty$. We consider the supremum over $[0,T \wedge \tau^{K}_{M}]$ of $\left(E_{t}^{K} +N_{t}^{K} + \Omega^{K}_{t}\right)^{p}$ and then take expectations. In particular, we verify that integrals against $\mathcal{N}$ are true semi-martingales thanks to the definition of the stopping time $\tau^{K}_{M}$, and all the expectations are well-defined and finite until this stopping time. We also use the fact that there exists $C_{p,x_{0}}>0$ such that for every $y > 0$, \[ \left(y+\dfrac{1+\omega(x_{0})}{K}\right)^{p}-y^{p} \leq \dfrac{C_{p,x_{0}}}{K}(1+y^{p-1}). \] We apply Fubini theorem because all the integrands are positive and this leads to
\begin{align*}
    \mathbb{E}\bigg(\underset{t \in [0,T \wedge \tau^{K}_{M}]}{\mathrm{sup}} \left(E_{t}^{K} +N_{t}^{K} + \Omega^{K}_{t}\right)^{p}\bigg) & \\
  & \hspace{-3 cm} \leq  \mathbb{E}\left(\left(E_{0}^{K} +N_{0}^{K} + \Omega^{K}_{0}\right)^{p}\right) + pC_{g} \displaystyle{\int_{0}^{T \wedge \tau^{K}_{M}} \mathbb{E}\left( \left(E_{s}^{K} +N_{s}^{K} + \Omega^{K}_{s}\right)^{p} \right) \mathrm{d}s} \\
    & \hspace{-2 cm} + C_{p,x_{0}}C_{b} \displaystyle{\int_{0}^{T \wedge \tau^{K}_{M}}}
   \mathbb{E}\left(E_{s-}^{K} +N_{s-}^{K} + \Omega^{K}_{s-} + \left(E_{s-}^{K} +N_{s-}^{K} + \Omega^{K}_{s-}\right)^{p} \right)  \mathrm{d}s,
\end{align*}
where we used the second point of Assumption~\ref{hyp:poidsomega}, and in particular the fact that it implies $b \leq C_{b}(1+ \mathrm{Id} + \omega)$. Finally, we use the fact that for $x \geq 0$, $x+x^{p} \leq 2(1+x^{p})$ to obtain
\begin{align*}
    \mathbb{E}\bigg(\underset{t \in [0,T\wedge \tau^{K}_{M}]}{\mathrm{sup}} \left(E_{t}^{K} +N_{t}^{K} + \Omega^{K}_{t}\right)^{p}\bigg) &\\
&   \hspace{-2.2 cm} \leq \mathbb{E}\left(\left(E_{0}^{K} +N_{0}^{K} + \Omega^{K}_{0}\right)^{p}\right) + 2TC_{p,x_{0}}C_{b} \\
    & \hspace{-1.8cm} + (pC_{g}+2C_{p,x_{0}}C_{b}) \displaystyle{\int_{0}^{T} \mathbb{E} \left( \sup_{u \in [0,s \wedge \tau^{K}_{M}]} \left(E_{u}^{K} +N_{u}^{K} + \Omega^{K}_{u}\right)^{p} \right) \mathrm{d}s}.
\end{align*}
We conclude with Gronwall lemma that
\begin{align*}
    \mathbb{E}\bigg(\underset{t \in [0,T \wedge \tau^{K}_{M}]}{\mathrm{sup}}  \left(E_{t}^{K} +N_{t}^{K} + \Omega^{K}_{t}\right)^{p}\bigg) & \\
& \hspace{-2 cm} \leq \bigg(\mathbb{E}\left(\left(E_{0}^{K} +N_{0}^{K} + \Omega^{K}_{0}\right)^{p}\right) + 2TC_{p,x_{0}}C_{b} \bigg)e^{(pC_{g}+2C_{p,x_{0}}C_{b})T}.
\end{align*}
By \eqref{eq:lesmom}, this upper bound is finite, uniformly on $K$, and does not depend on $M$. By the same arguments as in the proof of Proposition~\ref{prop:controlpop}, the sequence $\tau^{K}_{M}$ goes to $+ \infty$ when $M \rightarrow + \infty$, and this ends the proof.

\subsection{Proof of Theorem~\ref{theo:roel}}
\label{appendix:roel}

We reorganize the structure of the proof of Theorem 2.1 in \cite{roel_86}. A slight difference lies in the characterization of relatively compact sets in $(\mathcal{M}_{\omega}(\mathbb{R}^{*}_{+}),v)$ compared to the case $(\mathcal{M}_{1}(\mathbb{R}^{*}_{+}),v)$ treated in \cite{roel_86}. Let $\varepsilon >0$, our goal is to show that there exists $C$ relatively compact in $\mathbb{D}([0,T],(\mathcal{M}_{\omega}(\mathbb{R}^{*}_{+}),v))$ such that for all $K \in \mathbb{N}$, $P^{K}(C) \geq 1-\varepsilon$.
\\\\
First, let us enumerate the elements of the countable set $D \cup \{ \omega \}$ as $(f_{k})_{k \in \mathbb{N}}$, with $f_{0} = \omega$. For every $k \in \mathbb{N}$, by assumption, $(\pi_{f_{k}}*P^{K})_{K \in \mathbb{N}}$ is a tight sequence of probability measures on $\mathbb{D}([0,T],\mathbb{R})$, so there exists a compact set $C_{k} \subseteq \mathbb{D}([0,T],\mathbb{R})$ such that for every $K \in \mathbb{N}$, $(\pi_{f_{k}}*P^{K})(C_{k}) \geq 1 - \varepsilon/2^{k+1}$. We define
$$ C := \bigcap\limits_{k \in \mathbb{N}} \pi^{-1}_{f_{k}}(C_{k}), $$
which immediately verifies by construction, for every $K \in \mathbb{N}$,
\begin{align*}
P^{K}(C) = 1 - P^{K}\left(\bigcup\limits_{k \in \mathbb{N}} \left(\pi^{-1}_{f_{k}}(C_{k})\right)^{c} \right) \geq 1 - \sum\limits_{k \in \mathbb{N}} \left(1 - (\pi_{f_{k}}*P^{K})(C_{k}) \right) \geq 1- \varepsilon.
\end{align*}
Let us show that $C$ is relatively compact in $\mathbb{D}([0,T],(\mathcal{M}_{\omega}(\mathbb{R}^{*}_{+}),v))$ to conclude. We begin with a preliminary definition.

\begin{defi}
    For $\delta >0$, we write $\Pi_{\delta}$ for the set of all finite partitions $0=t_{0}<...<t_{k}=T$ of $[0,T]$, of any size $k$ and verifying $\min_{1  \leq i \leq k} (t_{i}-t_{i-1}) > \delta$. Let $(X,d)$ be a metric space, then the $\delta$-càdlàg modulus of continuity of any function $f \in \mathbb{D}([0,T],(X,d))$ is
        $$w_{\delta,d}(f) := \inf_{(t_{i})_{0 \leq i \leq k} \in \Pi_{\delta}} \hspace{0.2 cm} \sup_{1 \leq i \leq k} \hspace{0.2 cm} \sup_{(s,t) \in [t_{i-1},t_{i})} \hspace{0.2 cm} d(f(s),f(t)). $$
\end{defi}
The space $\mathbb{D}([0,T],(\mathcal{M}_{\omega}(\mathbb{R}^{*}_{+}),v))$ is complete (see Proposition B.2.1. in \cite{broduthesis} and Theorem 5.6 p.121 in \cite{ek_2005}), so we can use Theorem 6.3 p.123 in \cite{ek_2005} to assess that $C$ is relatively compact, if and only if
\begin{center}
$\forall t \in [0,T] \cap \mathbb{Q}, \quad \{ \mu_{t}, \mu \in C\} $ is relatively compact in $(\mathcal{M}_{\omega}(\mathbb{R}^{*}_{+}),v)$,
\end{center}
and
\begin{align}
\lim\limits_{\delta \rightarrow 0} \sup\limits_{\mu \in C} w_{\delta,d_{v}}(\mu) = 0.
\label{eq:kurtzikurtz}
\end{align} 
First, we have the inclusion $C \subseteq \pi_{\omega}^{-1}(C_{0})$, which implies that there exists a constant $M>0$, such that for every $t \in [0,T]$, for every $\mu \in C$, $\langle \mu_{t}, \omega \rangle \leq M$. For every $t \in [0,T]$, the subset $\{ \mu_{t}, \mu \in C\}$ is thus relatively compact in $(\mathcal{M}_{\omega}(\mathbb{R}^{*}_{+}),v)$ (it is included in the set $\{\nu, \, \langle \nu, \omega \rangle \leq M \}$, which is compact by the Banach-Alaoglu theorem). Let us finally show \eqref{eq:kurtzikurtz}. The set $D$ is by assumption a dense countable subset of $\mathcal{C}_{0}(\mathbb{R}^{*}_{+})$ for the topology of uniform convergence. The set $\mathcal{C}_{0}(\mathbb{R}^{*}_{+})$ is composed of bounded continuous functions, and is an algebra that separates points, thus it is a separating class of functions for $(\mathcal{M}_{\omega}(\mathbb{R}^{*}_{+}),v)$ (Theorem 4.5 in \cite{ek_2005}). In particular, $D = (f_{k})_{k \in \mathbb{N}^{*}}$ is dense in a convergence determining set, so by Theorem 2.4.~p.9 in \cite{kurtz1981approximation}, to show \eqref{eq:kurtzikurtz}, it suffices to show that for every $k \in \mathbb{N}^{*}$,
\begin{align*}
\lim\limits_{\delta \rightarrow 0} \sup\limits_{\mu \in C} w_{\delta,d}(\langle \mu, f_{k} \rangle) = 0,
\end{align*} 
where $d$ is the usual distance on $\mathbb{R}$. For every $k \in \mathbb{N}$, $C_{k}$ is compact in $\mathbb{D}([0,T],\mathbb{R})$, hence by Theorem 6.3 p.123 in \cite{ek_2005} again,
\begin{align*}
\lim\limits_{\delta \rightarrow 0} \sup\limits_{\phi \in C_{k}} w_{\delta,d}(\phi) = 0.
\end{align*} 
For every $k \in \mathbb{N}$, we have the inclusion $C \subseteq \pi_{f_{k}}^{-1}(C_{k})$, so finally 
\begin{align*}
\lim\limits_{\delta \rightarrow 0} \sup\limits_{\mu \in C} w_{\delta,d}(\langle \mu, f_{k} \rangle) \leq \lim\limits_{\delta \rightarrow 0} \sup\limits_{\mu \in \pi_{f_{k}}^{-1}(C_{k})} w_{\delta,d}(\langle \mu, f_{k} \rangle) \leq \lim\limits_{\delta \rightarrow 0} \sup\limits_{\phi \in C_{k}} w_{\delta,d}(\phi) = 0,
\end{align*} 
which entails \eqref{eq:kurtzikurtz} and ends the proof.

\subsection{Proof of Theorem~\ref{theo:melroel}}
\label{appendix:melroel}

We follow the same structure of proof as in \cite{meleard1993convergences}.  First, we introduce some notations and show a deterministic result.

\begin{defi}
The Prokhorov distance is defined for every $\mu,\nu \in M_{1}(\mathbb{R}^{*}_{+})$ by
\begin{align*}
d_{P}^{1}(\mu,\nu) := \inf \{ \varepsilon >0, \forall A \in \mathcal{B}(\mathbb{R}^{*}_{+}), \hspace{0.1 cm}  \mu(A) \leq \nu(A^{\varepsilon}) + \varepsilon, \hspace{0.1 cm} \nu(A) \leq \mu(A^{\varepsilon}) + \varepsilon \},
\end{align*}
where $\mathcal{B}(\mathbb{R}^{*}_{+})$ is the usual Borel $\sigma$-algebra on $\mathbb{R}^{*}_{+}$, and $A^{\varepsilon} := \{ y \in \mathbb{R}^{*}_{+}, \inf_{x \in A} |x-y| < \varepsilon \}.$ For any positive function $w$, we extend this definition to the $w$-Prokhorov distance $d_{P}^{w}$, defined for every $\mu,\nu \in \mathcal{M}_{w}(\mathbb{R}^{*}_{+})$ by
\begin{align*}
d_{P}^{w}(\mu,\nu) := d_{P}^{1}(w*\mu,w*\nu),
\end{align*} 
where $w*\mu$ is the usual pushforward of $\mu$ by $w$ (\textit{i.e.} for every measurable function $f$, $\langle w*\mu, f \rangle := \langle \mu, w f \rangle$). 
\label{defi:prok}
\end{defi}

The $w$-Prokhorov distance metrizes the space $(\mathcal{M}_{w}(\mathbb{R}^{*}_{+}),\mathrm{w})$ (see Proposition B.2.2 in \cite{broduthesis}). Also, it is well-known that $(\mathcal{M}_{1}(\mathbb{R}^{*}_{+}),v)$ is metrizable by a distance $d_{v}^{1}$ (see Proposition B.2.1 in \cite{broduthesis}), and from this base case, we define the distance $d_{v}^{w}$, for every $\mu, \nu \in \mathcal{M}_{w}(\mathbb{R}^{*}_{+})$ by
\begin{align*}
d_{v}^{w}(\mu,\nu) := d_{v}^{1}(w*\mu,w*\nu).
\end{align*}
 
\begin{lemme}
Let $w$ be a positive function on $\mathbb{R}^{*}_{+}$, $(\nu_{n})_{n \in \mathbb{N}}$ and $\nu$ elements of $\mathcal{M}_{w}(\mathbb{R}^{*}_{+})$. Then,
\begin{center}
$d^{w}_{P}(\nu_{n},\nu) \xrightarrow[n \rightarrow + \infty]{} 0 $, if and only if $\bigg(d^{w}_{v}(\nu_{n},\nu) \xrightarrow[n \rightarrow + \infty]{} 0 $ and $|\langle \nu_{n}- \nu, w \rangle|\xrightarrow[n \rightarrow + \infty]{} 0. \bigg)$
\end{center} 
\label{lemme:deuxiemedistance}
\end{lemme}

\begin{proof}
The direct implication is straightforward, we show the converse implication. Let $\varepsilon>0$ and $f \in \mathcal{C}_{w}(\mathbb{R}^{*}_{+})$, so there exists a constant $C>0$ such that $|f| \leq Cw$. Also, we consider $(\zeta_{p})_{p \in \mathbb{N}}$ an increasing sequence of positive functions in $\mathcal{C}^{\infty}_{c}(\mathbb{R}^{*}_{+})$ that converges pointwise towards the constant function equal to $1$ on $\mathbb{R}^{*}_{+}$. First, we write, for any $p \in \mathbb{N}$,
\begin{align*}
| \langle \nu_{n}, f \rangle - \langle \nu, f \rangle | & \leq | \langle \nu_{n}, f \zeta_{p} \rangle - \langle \nu, f \zeta_{p} \rangle | + | \langle \nu_{n}, f (1 -\zeta_{p}) \rangle | + | \langle \nu, f (1 -\zeta_{p}) \rangle | \\
& \leq | \langle \nu_{n}, f \zeta_{p} \rangle - \langle \nu, f \zeta_{p} \rangle | + C\bigg(  \langle \nu_{n}, w (1 -\zeta_{p}) \rangle  + \langle \nu, w (1 -\zeta_{p}) \rangle \bigg).
\end{align*}  
By dominated convergence, $ \langle \nu, w (1 -\zeta_{p}) \rangle$ converges to 0 when $p$ goes to $+ \infty$. Let us fix $p_{0} \in \mathbb{N}$ such that $ \langle \nu, w (1 -\zeta_{p_{0}}) \rangle \leq \varepsilon$. Then, we can write for any $n \in \mathbb{N}$,
$$ \langle \nu_{n}, w (1 -\zeta_{p_{0}}) \rangle = \langle \nu_{n}, w \rangle  - \langle \nu_{n}, w\zeta_{p_{0}} \rangle, $$
which converges to $\langle \nu, w (1 -\zeta_{p_{0}}) \rangle$ when $n$ goes to $+ \infty$ by assumption. Thus, there exists $n_{0} \in \mathbb{N}$ such that for every $n \geq n_{0}$, $\langle \nu_{n}, w (1 -\zeta_{p_{0}}) \rangle  \leq 2 \varepsilon$. Finally, by vague convergence, the term $| \langle \nu_{n}, f \zeta_{p_{0}} \rangle - \langle \nu, f \zeta_{p_{0}} \rangle |$ converges to 0 when $n$ goes to $+ \infty$. To conclude, for every $\varepsilon>0$, we can find $n_{1}\in \mathbb{N}$ such that for every $n \geq n_{1}$, $| \langle \nu_{n}, f \rangle - \langle \nu, f \rangle | \leq (1+3C) \varepsilon$, so that $| \langle \nu_{n}, f \rangle - \langle \nu, f \rangle | \xrightarrow[n \rightarrow + \infty]{} 0 $. This is valid for any $f \in \mathcal{C}_{w}(\mathbb{R}^{*}_{+})$, which ends the proof.
\end{proof}

\begin{defi}
For $w$ positive function on $\mathbb{R}^{*}_{+}$, we define the distance $\eth_{w}$ on $\mathcal{M}_{w}(\mathbb{R}^{*}_{+})$ by
$$\forall (\mu,\nu) \in \mathcal{M}_{w}(\mathbb{R}^{*}_{+}), \quad \eth_{w}(\mu,\nu) := d^{w}_{v}(\mu,\nu) + | \langle \mu-\nu, w \rangle |.$$
\label{defi:eth}
\end{defi}

We verify immediately thanks to Lemma~\ref{lemme:deuxiemedistance} that the distances $\eth_{w}$ and $d_{P}^{w}$ are topologically equivalent on $\mathcal{M}_{w}(\mathbb{R}^{*}_{+})$. Let us continue with a probabilistic result. 

\begin{prop}
Let $w$ be a positive function on $\mathbb{R}^{*}_{+}$, $(\nu_{n})_{n \in \mathbb{N}}$ and $\nu$ random variables in $\mathcal{M}_{w}(\mathbb{R}^{*}_{+})$. Then, $(\nu_{n})_{n \in \mathbb{N}}$ converges in law towards $\nu$ in $(\mathcal{M}_{w}(\mathbb{R}^{*}_{+}),d_{P}^{w})$, if and only if $(\nu_{n})_{n \in \mathbb{N}}$ converges in law towards $\nu$ in $(\mathcal{M}_{w}(\mathbb{R}^{*}_{+}),d_{v}^{w})$ and $(\langle \nu_{n}, w \rangle)_{n \in \mathbb{N}}$ converges in law towards $\langle \nu, w \rangle$ in $\mathbb{R}$.
\label{prop:probaresukt}
\end{prop}

\begin{proof}
The direct implication is straightforward, we show the converse implication. First, by convergence of the marginal distributions, the sequence of laws of $(\nu_{n}, \langle \nu_{n}, w \rangle)_{n \in \mathbb{N}}$ is tight in $(\mathcal{M}_{w}(\mathbb{R}^{*}_{+}),d^{w}_{v}) \times \mathbb{R}$. By Prokhorov theorem, any subsequence of $(\nu_{n}, \langle \nu_{n}, w \rangle)_{n \in \mathbb{N}}$ admits a subsequence converging in law in $(\mathcal{M}_{w}(\mathbb{R}^{*}_{+}),d^{w}_{v}) \times \mathbb{R}$, towards a random variable that we denote as $(\nu',x')$. Furthermore, by convergence of the marginal distributions, $\nu'$, respectively $x'$, has same law as $\nu$, respectively $\langle \nu,w \rangle$. Our main goal is to show that for any such converging subsequence (still denoted as $(\nu_{n}, \langle \nu_{n}, w \rangle)_{n \in \mathbb{N}}$), the law of the limiting couple $(\nu',x')$ is unique and equal to the law of the couple $(\nu, \langle \nu, w \rangle)$. Note that the technical part of this proof is to show the equality in law as couples, and not only for the marginal distributions. By the Skorokhod representation theorem (Theorem 6.7. p.70 in \cite{bill_99}), there exists a probability space $\Omega$ and random variables $(\mu_{n},x_{n})_{n \in \mathbb{N}}$ and $(\mu,x)$ defined on $\Omega$ with values in $\mathcal{M}_{w}(\mathbb{R}^{*}_{+}) \times \mathbb{R}$, such that
\begin{itemize}
\item $\forall n \in \mathbb{N}$, $(\mu_{n},x_{n})$ and $(\nu_{n},\langle \nu_{n}, w \rangle)$ have same law;
\item $(\mu,x)$ and $(\nu',x')$ have same law, so $\mu$ and $x$ have respectively same law as $\nu'$ and $x'$, thus respectively same law as $\nu$ and $\langle \nu, w \rangle$;
\item $(\mu_{n},x_{n})_{n \in \mathbb{N}}$ converges almost surely towards $(\mu,x)$ in $(\mathcal{M}_{w}(\mathbb{R}^{*}_{+}),d^{w}_{v}) \times \mathbb{R}$.
\end{itemize}
First, we have for $n \in \mathbb{N}$, by the previous equalities in law,
\[ \mathbb{E}(\mathbb{1}_{\{ 0 \}}(x_{n}- \langle \mu_{n}, w \rangle)) = \mathbb{E}(\mathbb{1}_{\{ 0 \}}(\langle \nu_{n},w \rangle - \langle \nu_{n}, w \rangle))=1, \]
so $x_{n}$ is almost surely equal to $\langle \mu_{n}, w \rangle$ for every $n \in \mathbb{N}$. Then, let $(\zeta_{p})_{p \in \mathbb{N}}$ an increasing sequence of positive functions in $\mathcal{C}^{\infty}_{c}(\mathbb{R}^{*}_{+})$ that converges pointwise towards the constant function equal to $1$ on $\mathbb{R}^{*}_{+}$, and write almost surely, for every $p \in \mathbb{N}$,
\begin{align*}
x = \lim_{n \rightarrow + \infty} x_{n} = \lim_{n \rightarrow + \infty} \langle \mu_{n}, w \rangle \geq \lim_{n \rightarrow + \infty} \langle \mu_{n}, w \zeta_{p} \rangle = \langle \mu, w \zeta_{p} \rangle,
\end{align*}
so by monotone convergence, $x \geq \langle \mu, w \rangle$. Also, we have
\[ \mathbb{E}(x- \langle \mu, w \rangle) = \mathbb{E}(\langle \nu,w \rangle - \langle \nu, w \rangle)=0, \]
so finally $x$ is almost surely equal to $\langle \mu, w \rangle$. Hence, by the third point above, the previous almost sure equalities, and Lemma~\ref{lemme:deuxiemedistance}, $(\mu_{n})_{n \in \mathbb{N}}$ converges almost surely towards $\mu$ in $(\mathcal{M}_{w}(\mathbb{R}^{*}_{+}),d^{w}_{P})$. By the previous equalities in law, we obtain that $(\nu_{n})_{n \in \mathbb{N}}$ converges in law towards $\nu$ in $(\mathcal{M}_{w}(\mathbb{R}^{*}_{+}),d^{w}_{P})$. We thus have shown that every subsequence of $(\nu_{n})_{n \in \mathbb{N}}$ admits a subsequence converging in law in $(\mathcal{M}_{w}(\mathbb{R}^{*}_{+}),d^{w}_{P})$ towards $\nu$, which concludes.
\end{proof}

\textbf{Proof of Theorem~\ref{theo:melroel}:}
\\\\
First, we can choose the distance $d_{v}^{w}$ to metrize $(\mathcal{M}_{w}(\mathbb{R}^{*}_{+}),v)$; and we can choose the distance $\eth_{w}$ to metrize $(\mathcal{M}_{w}(\mathbb{R}^{*}_{+}),\mathrm{w})$, and use Proposition~\ref{prop:probaresukt} with $d^{w}_{P}$ replaced by $\eth_{w}$. This is justified by the fact that for every $T>0$, the associated Skorokhod distances are topologically equivalent on $\mathbb{D}([0,T],\mathcal{M}_{w}(\mathbb{R}^{*}_{+}))$ (see Lemma B.2.9 in \cite{broduthesis}). The fact that $\mathrm{(i)}$ implies $\mathrm{(ii)}$ in Theorem~\ref{theo:melroel} is straightforward, so we focus on the converse implication. First, the limiting process $(\nu^{*}_{t})_{t \in [0,T]}$ is in $\mathcal{C}([0,T], (\mathcal{M}_{w}(\mathbb{R}^{*}_{+}),\mathrm{w}))$, hence the sequence $((\nu^{K}_{t})_{t \in [0,T]})_{K \in \mathbb{N}}$ and $((\langle \nu^{K}_{t}, w \rangle)_{t \in [0,T]})_{K \in \mathbb{N}}$ are $\mathcal{C}$-tight (Definition 3.25 p.351 in \cite{Jacod1987LimitTF}) in the space $\mathbb{D}([0,T], (\mathcal{M}_{w}(\mathbb{R}^{*}_{+}),v))$ and the space $\mathbb{D}([0,T],\mathbb{R})$. By Corollary 3.33 p.353 in \cite{Jacod1987LimitTF}, the sequence $((\nu^{K}_{t},\langle \nu^{K}_{t}, w \rangle)_{t \in [0,T]})_{K \in \mathbb{N}}$ is $\mathcal{C}$-tight in $\mathbb{D}([0,T], (\mathcal{M}_{w}(\mathbb{R}^{*}_{+}),v) \times \mathbb{R})$. The structure of the proof is then essentially the same as in Proposition~\ref{prop:probaresukt}.
\\\\
By Prokhorov theorem, any subsequence of $((\nu^{K}_{t},\langle \nu^{K}_{t}, w \rangle)_{t \in [0,T]})_{K \in \mathbb{N}}$ admits a subsequence converging in law in $\mathbb{D}([0,T], (\mathcal{M}_{w}(\mathbb{R}^{*}_{+}),v) \times \mathbb{R})$, towards a random variable that we denote as $(\nu'_{t},x'_{t})_{t \in [0,T]}$, and which is almost surely continuous. Furthermore, by convergence of the marginal distributions, $(\nu'_{t})_{t \in [0,T]}$, respectively $(x'_{t})_{t \in [0,T]}$, has same law as $(\nu^{*}_{t})_{t \in [0,T]}$, respectively $(\langle \nu^{*}_{t},w \rangle)_{t \in [0,T]}$.~By the Skorokhod representation theorem (Theorem 6.7. p.70 in \cite{bill_99}), there exists a probability space $\Omega$ and random variables $((\mu^{K}_{t},x^{K}_{t})_{t \in [0,T]})_{K \in \mathbb{N}}$ and $(\mu^{*}_{t},x^{*}_{t})_{t \in [0,T]}$, defined on $\Omega$ with values in the spaces $\mathbb{D}([0,T],(\mathcal{M}_{w}(\mathbb{R}^{*}_{+}),\mathrm{w}) \times \mathbb{R})$ and $\mathcal{C}([0,T],(\mathcal{M}_{w}(\mathbb{R}^{*}_{+}),\mathrm{w}) \times \mathbb{R})$, such that
\begin{itemize}
\item $\forall K \in \mathbb{N}$, $(\mu^{K}_{t},x^{K}_{t})_{t \in [0,T]}$ and $(\nu^{K}_{t},\langle \nu^{K}_{t}, w \rangle)_{t \in [0,T]}$ have same law. In particular, $(\mu^{K}_{t},x^{K}_{t})_{t \in [0,T]}$ is almost surely with values in $\mathbb{D}([0,T], (\mathcal{M}_{w}(\mathbb{R}^{*}_{+}),\mathrm{w}) \times \mathbb{R})$.
\item $(\mu^{*}_{t},x^{*}_{t})_{t \in [0,T]}$ and $(\nu'_{t},x'_{t})_{t \in [0,T]}$ have same law, so in particular $(\mu^{*}_{t})_{t \in [0,T]}$ and $(x^{*}_{t})_{t \in [0,T]}$ have respectively same law as $(\nu'_{t})_{t \in [0,T]}$ and $(x'_{t})_{t \in [0,T]}$, thus same law as $(\nu^{*}_{t})_{t \in [0,T]}$ and $(\langle \nu^{*}_{t},w \rangle)_{t \in [0,T]}$. Also, $(\mu^{*}_{t},x^{*}_{t})_{t \in [0,T]}$ is almost surely continuous.
\item $((\mu^{K}_{t},x^{K}_{t})_{t \in [0,T]})_{K \in \mathbb{N}}$ converges almost surely towards $(\mu^{*}_{t},x^{*}_{t})_{t \in [0,T]}$ in $\mathbb{D}([0,T],(\mathcal{M}_{w}(\mathbb{R}^{*}_{+}),v) \times \mathbb{R})$.
\end{itemize}
By Lemma B.2.8 in \cite{broduthesis}, $((\mu^{K}_{t},x^{K}_{t})_{t \in [0,T]})_{K \in \mathbb{N}}$ converges almost surely uniformly towards $(\mu^{*}_{t},x^{*}_{t})_{t \in [0,T]}$ in $\mathbb{D}([0,T],(\mathcal{M}_{w}(\mathbb{R}^{*}_{+}),v) \times \mathbb{R})$. Thus, we can write almost surely
\begin{align}
\lim_{n \rightarrow + \infty}  \sup_{t \in [0,T]} d_{v}(\mu^{K}_{t},\mu^{*}_{t}) = 0 \label{eq:equautonunejaj} \\
\lim_{n \rightarrow + \infty}  \sup_{t \in [0,T]} |x^{K}_{t}- x^{*}_{t} | = 0. \label{eq:eqdeuxutonunejaj}
\end{align}
In the following, let us show that almost surely, 
\begin{align}
\sup_{t \in [0,T]} |x^{K}_{t}-\langle  \mu^{K}_{t}, w \rangle| = \sup_{t \in [0,T]} |x^{*}_{t} - \langle \mu^{*}_{t}, w \rangle) | = 0.
\label{eq:almostusrleyx}
\end{align}
First, by the previous equalities in law and because $f \in \mathbb{D}([0,T],\mathbb{R}) \mapsto \sup_{t \in [0,T]} |f(t)|$ is well-defined and measurable, we have
\[ \mathbb{E}\left(\mathbb{1}_{\{0\}}\left(\sup_{t \in [0,T]} | x^{K}_{t}- \langle \mu^{K}_{t}, w \rangle | \right)\right) = \mathbb{E}\left(\mathbb{1}_{\{0\}}\left(\sup_{t \in [0,T]} | \langle \nu^{K}_{t}, w \rangle - \langle \nu^{K}_{t}, w \rangle | \right)\right)=1, \]
so almost surely $\sup_{t \in [0,T]} |x^{K}_{t}-\langle  \mu^{K}_{t}, w \rangle| = 0$. Also, for any $t \in [0,T]$ fixed, we can use the same argument as in the proof of Proposition~\ref{prop:probaresukt} to show that almost surely $x^{*}_{t} = \langle \mu^{*}_{t}, w \rangle$. Hence, almost surely, 
$$ \sup_{t \in [0,T] \cap \mathbb{Q}} |x^{*}_{t} - \langle \mu^{*}_{t}, w \rangle | = 0, $$
and $t \in [0,T] \mapsto |x^{*}_{t} - \langle \mu^{*}_{t}, w \rangle |$ is almost surely a continuous function, hence we obtain \eqref{eq:almostusrleyx}. From \eqref{eq:equautonunejaj}, \eqref{eq:eqdeuxutonunejaj} and \eqref{eq:almostusrleyx}, and by definition of $\eth_{w}$, we obtain that almost surely
\begin{align*}
\lim_{n \rightarrow + \infty}  \sup_{t \in [0,T]} \eth_{w}(\mu^{K}_{t},\mu^{*}_{t}) = 0.
\end{align*}
In particular, this implies that $((\mu^{K}_{t})_{t \in [0,T]})_{K \in \mathbb{N}}$ converges almost surely (and this convergence is even uniform on $[0,T]$) towards $(\mu^{*}_{t})_{t \in [0,T]}$ in $\mathbb{D}([0,T],(\mathcal{M}_{w}(\mathbb{R}^{*}_{+}),\mathrm{w}))$, which concludes thanks to the equality in law between $(\mu^{K}_{t})_{t \in [0,T]}$ and $(\nu^{K}_{t})_{t \in [0,T]}$ for all $K \in \mathbb{N}$, and between $(\mu^{*}_{t})_{t \in [0,T]}$ and $(\nu^{*}_{t})_{t \in [0,T]}$.

\section{Simulation parameters and algorithms}
\label{app:simu}

\subsection{Simulation parameters}
\label{app:simupara}

For $x>0$ and $R \geq 0$, we set:
\begin{enumerate}
\item $\ell(x) :=C_{\alpha}x^{\alpha}$,
\item $b(x):= \mathbb{1}_{x > x_{0}}C_{\beta}x^{\beta}$,
\item $f(x,R) := \dfrac{R}{\kappa+R}  C_{\gamma}x^{\gamma}$ (\textit{i.e.} $\phi(R) = \dfrac{R}{\kappa+R}$ and $\psi(x) = C_{\gamma}x^{\gamma}$),
\item $d(x):= C_{\delta}x^{\delta}$,
\item $\varsigma(R) := D(R_{\mathrm{in}}-R),$
\end{enumerate}
where the different parameters are specified in Table~\ref{tabl:simu}. Also, we fix a value for the conversion efficiency coefficient $\chi >1$ appearing in \eqref{eq:eqress}. We choose a deterministic initial condition for the resources $R_{0} \in [0,R_{\max}]$. For the initial state of the population, we pick $K$ random individual energies, chosen independently and according to an initial distribution with compact support denoted as $[x_{\min},x_{\max}]$, and absolutely continuous with respect to Lebesgue measure, with a density $u_{0}$ shown on Figure~\ref{fig:intiialdensitth}, which is given by
$$ \forall x>0, \quad u_{0}(x) := C \left(  \dfrac{(x-x_{\min})(x_{\max}-x)}{(x_{\max}-x_{\min})^{2}} \right)^{5} \mathbb{1}_{x \in [x_{\min},x_{\max}]}, $$
with a constant $C>0$ being such that $\int_{\mathbb{R}^{*}_{+}} u_{0}(x) \mathrm{d}x = 1$, and parameters specified in Table~\ref{tabl:simu}.

\begin{figure}[h!]
\centering
\includegraphics[scale=0.5]{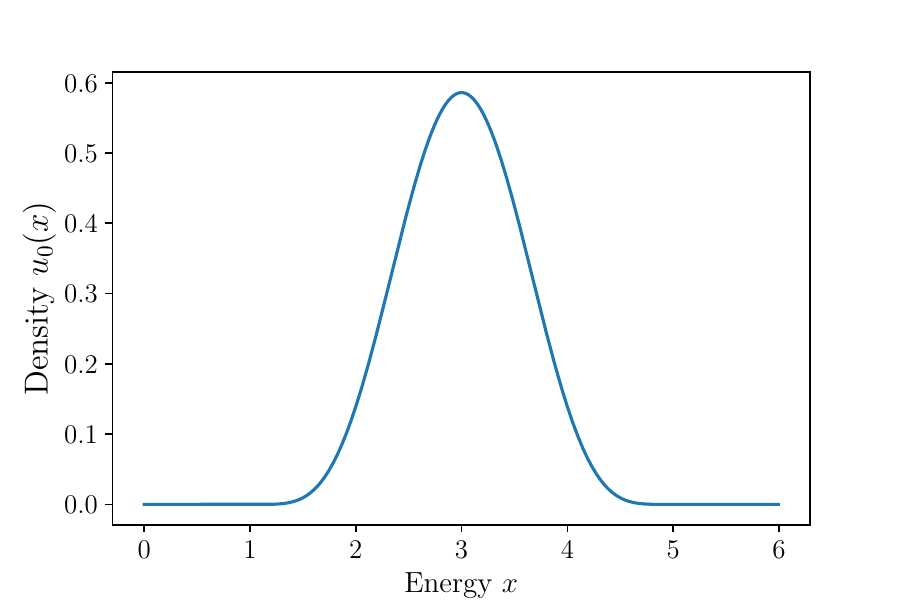}
\caption{Initial density $u_{0}(.)$}
\label{fig:intiialdensitth}
\end{figure}

We simulate the IBM and the PDE model during a time $T>0$. Under the allometric setting with $\alpha \in (0,1)$, for $t \in [0,T]$, we can upper bound the maximal energy an individual can reach by the solution of 
\begin{center}
        $\left\{
\begin{array}{ll}
f'(t) =   (\phi(R_{\max})C_{\gamma}-C_{\alpha}) f^{\alpha}(t),\\
f(0) = x_{\max} 
\end{array}
\right.$
    \end{center}
taken at time $T$, where $R_{\max}>0$ is defined in \ref{subsec:dyn}. This upper bound is precisely given by 
\begin{align}
M_{T} := \bigg(x_{\max}^{1-\alpha} + (1-\alpha)(\phi(R_{\max})C_{\gamma}-C_{\alpha})T \bigg)^{\frac{1}{1-\alpha}}.
\label{eq:aimechi}
\end{align}

Hence, we work on the fixed energy window $[0,M_{T}]$ for the construction of our algorithms in Appendix~\ref{app:algo}. The set of parameters of Table~\ref{tabl:simu} corresponds to the allometric setting supported by the Metabolic Theory of Ecology \cite{brown_04,malerba_2019}.

\begin{table}[h!]
\begin{center}
\begin{tabular}{|c|c|}
	\hline 
	Parameter & Value \\
	\hline
	$\alpha$ & 0.75 \\
	$\gamma$ & 0.75 \\
	$\beta$ & -0.25 \\
	$\delta$ & -0.25 \\
	$C_{\alpha}$ & 1 \\
	$C_{\gamma}$ & 2 \\
	$C_{\beta}$ & 0.1 \\
	$C_{\delta}$ & 0.05 \\
	$\kappa$ & 5 \\
	$D$ & 0.275 \\
	\hline 
\end{tabular}
\hspace{1.5 cm}
\begin{tabular}{|c|c|}
	\hline 
	Parameter & Value \\
	\hline
	$R_{\mathrm{in}}$ & 2 \\	
	$R_{\max}$ & 2 \\
	$\chi$ & 200 \\
	$K$ & 100, 1000 or 10000 \\
	$x_{0}$ & 1 \\
	$R_{0}$ & 1 \\
	$x_{\min}$ & 1 \\
	$x_{\max}$ & 5 \\
	$T$ & 200 \\
	$M_{T}$ & 208688 \\
	\hline 
\end{tabular}
\end{center}
\caption{Simulation parameters.}
\label{tabl:simu}
\end{table}

\subsection{Algorithms for the IBM and the PDE model}
\label{app:algo}

\subsubsection{Algorithm for the IBM}

We approximate the IBM with a Gillespie algorithm \cite{GILLESPIE1976403}. We introduce two parameters $\overline{b},\overline{d}>0$, and recall the structure of one step of this classical algorithm in Section III.3.1 of \cite{broduthesis}. Importantly, note that contrary to existing simulations in the literature (see Section 2.2 in \cite{FRITSCH20151}), we have unbounded jump rates, so we need to justify that our algorithm approximates well the IBM (\textit{i.e.} that the artificial upper bounds $\overline{b}$ and $\overline{d}$ chosen for birth and death rates are sufficiently high). We refer to the discussion of Section III.3 in \cite{broduthesis} and work with parameters $\overline{d} = 2.10^{4}$ and $\overline{b}= 0.1$.

\subsubsection{Algorithm for the PDE model}

We use a finite differences scheme of order 1 to approximate the transport term in \eqref{eq:indivedp}. To simulate the temporal evolution of the system, we simply use a classical Euler scheme. There is a numerical approximation due to these low order schemes, especially in the context of unbounded growth, birth and death rates. Nevertheless, in this article, our goal is not to go deep in the optimization of our algorithm, nor to give a mathematical proof of the convergence of our numerical scheme. Still, we refer to Section III.4.1.2 in \cite{broduthesis} for the difficulties we encountered in the calibration of the energy discretization grid. Indeed, we propose a specific choice of energy discretization (see Section III.4.1 in \cite{broduthesis} for details) adapted to our birth dynamics, inspired by usual strategies depicted in the literature \cite{doumicos,anarat}.

\bibliographystyle{alpha}
\bibliography{biblio}

\end{document}